\documentclass[oneside, 11pt]{amsart} \topmargin      -10mm \textwidth      160 true mm
\usepackage{amsmath, amssymb, amsfonts, amstext, amsthm, amscd}
\usepackage[mathscr]{eucal}
\usepackage{enumerate}
\usepackage{tikz,color,soul}
\usetikzlibrary{matrix,arrows}
\usepackage[utf8]{inputenc}
\usepackage[english]{babel}

\usepackage{mathrsfs}

\usepackage{verbatim}

\usepackage[new]{old-arrows}

\usepackage{longtable}

\usetikzlibrary{arrows.meta, positioning,arrows}
\newtheorem{thm}{Theorem}[subsection]
\newtheorem{prop}[thm]{Proposition}
\newtheorem{lemma}[thm]{Lemma}
\newtheorem{cor}[thm]{Corollary}
\newtheorem{defn}[thm]{Definition}
\newtheorem{ex}[thm]{Example}
\newtheorem{rmk}[thm]{Remark}

\numberwithin{equation}{section}

\newcommand{\C}{\mathbb C}

\newcommand{\R}{\mathbb R}
\newcommand{\Z}{\mathbb Z}

\def \mc{\mathcal}
\allowdisplaybreaks

\begin{document}

\title{Stability of equivariant vector bundles over toric varieties}

\author{Jyoti Dasgupta}
\address{Department of Mathematics, Indian Institute of Technology-Madras, Chennai, India}
\email{jdasgupta.maths@gmail.com}

\author{Arijit Dey}
\address{Department of Mathematics, Indian Institute of Technology-Madras, Chennai, India}
\email{arijitdey@gmail.com}

\author{Bivas Khan}
\address{Department of Mathematics, Indian Institute of Technology-Madras, Chennai, India}
\email{bivaskhan10@gmail.com}



\begin{abstract}
	We give a complete answer to the question of (semi)stability of tangent bundle of any nonsingular projective complex toric variety with Picard number 2 by using combinatorial crietrion of (semi)stability of an equivariant sheaf. We also give a complete answer to the question of (semi)stability of tangent bundle of all toric Fano 4-folds with Picard number \(\leq\) 3 which are classified by Batyrev \cite{batyrev}.  We have constructed a collection of equivariant indecomposable rank 2 vector bundles on Bott tower and pseudo-symmetric toric Fano varieties. Further in case of Bott tower, we have shown the existence of an equivariant stable rank 2 vector bundle with certain Chern classes with respect to a suitable polarization.


\end{abstract}

\maketitle

\section{Introduction}

Let $X$ be a toric variety of dimension $n$, equipped with
an action of the $n$-dimensional torus $T$ with an associated fan $\Delta$ over an algebraically closed field $k$. A quasi-coherent sheaf $\mc{E}$ on $X$ is said to be $T$-equivariant or simply an equivariant sheaf if it admits a lift of the $T$-action on $X$, which is 
linear on the stalks of $\mc{E}$.  An equivariant structure on a sheaf $\mc{E}$ need not be unique.  Any line bundle on a toric variety has an equivariant structure.  It is well known that any locally free sheaf $\mc{E}$ on $X$ is equivariant if and only if $t^{*}\mc{E} \,\cong\, \mc{E}$ for every $t\in T$ (see \cite[Proposition 1.2.1]{kly}). Equivariant vector bundles over a nonsingular complete toric variety up to isomorphism were first classified by Kaneyama \cite{kane}, \cite{kaneyama2} by involving both combinatorial and linear algebraic data modulo an equivalence relation. Recently this work has been generalized for equivariant principal $G$--bundles over smooth complex toric variety, where $G$ is a complex linear algebraic group \cite{BDPIJM}. Later in a foundational paper \cite{kly}, Klyachko classified equivariant vector bundles more systematically. In this paper, he gave a complete classification of equivariant bundles over arbitrary toric variety  in terms of a family of decreasing filtrations on a fixed finite dimensional vector space indexed by one dimensional cones satisfying certain compatibility condition \cite[Theorem 2.2.1]{kly}. Most of the topological and algebraic invariants of equivariant vector bundles like Chern classes, global sections, cohomology spaces; he could decode from this filtration data. As a major application, later he used classification of equivariant vector bundles over $\mathbb{P}^2$ to prove Horn's conjecture on eigenvalues of sums of Hermitian matrices (\cite{kly2}). Recently over the complex numbers, this classification result has been generalized for equivariant principal $G$-bundles over any complex toric variety using two different approaches  (\cite{BDPTG}, \cite{KM1}, \cite{KM2}), where $G$ is a complex reductive algebraic group. 

In an unpublished preprint \cite{kly_sheaf}, Klyachko gave a generalization of the above classification theorem for equivariant torsion free sheaves, and gave a sketch without all details. Thereafter, Perling introduced the notion of $\Delta$-families \(\{E^{\sigma}_m\}_{\sigma \in \Delta, m \in M}\) for any quasi-coherent equivariant sheaf $\mc E$ which is constructed from the $T$-eigenspace decompositions of the modules of sections, together with the multiplication maps for regular $T$-eigenfunctions. He showed that the category of $\Delta$-families is equivalent to the category of equivariant quasi-coherent sheaves \cite{perling}. When the sheaf $\mathcal E$ is torsion free, corresponding $\Delta$-family induces a family of multifiltrations of subspaces \(\{E^{\sigma}_m\}_{\sigma \in \Delta, m \in M}\) on a fixed finite dimensional vector space \(\mathbf{E}^0\) satisfying certain compatibility condition (see Theorem \ref{per5.18}). Further if we restrict ourselves to reflexive sheaves then the entire Perling data becomes a family of increasing full finite dimensional filtered vector spaces $(\mathbf{E}^0, \{E^{\rho}(i)\}_{\rho \in \Delta(1)})$ without any compatibility condition, where \(E^{\rho}_m=E^{\rho}(\langle m, v_{\rho} \rangle)\). Conversely any such family of filtered vector spaces corresponds to an equivariant reflexive sheaf \cite[Theorem 5.19]{perling}. This crucial observation of Perling is the starting point of the paper. Further the first Chern class of an equivariant coherent sheaf can be computed from its associated $\Delta$-family (see \cite[Corollary 3.18]{kool}). 

Let $H$ be an equivariant very ample line bundle (equivalently, $T$-invariant very ample divisor) of $X$. An equivariant torsion free sheaf $\mc{E}$ on $X$ is said to be equivariantly (semi)stable with respect to $H$ if $\mu(\mc{F}) (\leq)< \mu(\mc{E})$ for every proper equivariant subsheaf $\mc{F}\,\subset \,\mc{E}$ (see Section \ref{sec:stability}). From the uniqueness of the Harder-Narasimhan filtration it follows easily that the notions of semistability and equivariant semistability of an equivariant torsion free sheaf on a nonsingular projective toric variety are equivalent. Further, the notions of equivariant stability and stability also coincide for any equivariant torsion free  sheaf(see \cite[Theorem 2.1]{biswas2018stability}). When $\mc{E}$ is an equivariant reflexive sheaf, to determine its (semi)stability it is enough to consider equivariant reflexive subsheaves of $\mc{E}$ (see Remark \ref{ref2}). 

The purpose of this paper is two fold. First we study (semi)stability of tangent bundle of a nonsingular projective toric variety with Picard number atmost $3$ (in Section 4 and 5). Secondly we construct new examples of rank $2$ equivariant vector bundles which are indecomposable or even stable over a large collection of nonsingular projective toric varieties of arbitrary dimension (in Section 6). 
Both these results rely on the key fact that one can combinatorially classify equivariant reflexive subsheaves of an equivariant reflexive sheaf (see Corollary \ref{RS1}). This turns out to be central theme of the paper. In fact with this technique, theoretically it is possible to check (semi)stability of any equivariant torsion free sheaf on a nonsingular projective toric variety. But as the Picard number grows and the fan structure becomes more and more complicated, the task of computing degree of subsheaves becomes cumbersome. We hope one can write a computer program to check (semi)stability of any equivariant torsion free sheaf from its given combinatorial data and the fan structure of the toric variety with respect to any polarization. 

Tangent bundles $\mc{T}_{X}$ are natural examples of equivariant vector bundles on nonsingular toric varieties. The filtration data \( \left( \mathscr{T}, \{\mathscr{T}^{\rho}(i)\}_{\rho \in \Delta(1), i \in \Z} \right)  \) associated to \(\mathcal{T}_X\) is relatively simple, it has a two step filtration of flag type $(1,n-1)$ for each $\rho \in \Delta (1)$ (see Corollary \ref{pertangb}). The first main step of this paper is to show that equivariant reflexive subsheaves of \(\mathcal{T}_X\) are in one-one correspondence with induced subfiltrations $\left( \mathbf{F}^0, \{F^{\rho}(i) \}_{\rho \in \Delta(1)} \right)$ of  \( \left( \mathscr{T}, \{\mathscr{T}^{\rho}(i)\}_{\rho \in \Delta(1), i \in \Z} \right) \). In fact this holds for all equivariant torsion free (respectively, reflexive) subsheaves of any equivariant torsion free (respectively, reflexive) sheaf (see Corollary \ref{RS1}). This result is a natural generalization of \cite[Proposition 4.1.1]{bj},  where equivariant subbundles of an equivariant vector bundle were classified. We first apply this result to give a very simple proof of stability of tangent bundle of a projective space (see Proposition \ref{stabtanP}). Next we study (semi)stability of tangent bundle of a nonsingular projective toric variety with Picard number 2. A theorem of Kleinschmidt \cite[Theorem 7.3.7]{Cox} tells us that any such variety $X$ is isomorphic to \(\mathbb{P}(\mathcal{O}_{\mathbb{P}^s}  \oplus \mathcal{O}_{\mathbb{P}^s}(a_1) \oplus \cdots \oplus \mathcal{O}_{\mathbb{P}^s}(a_r)  )\), where \(s, r \geq 1\), \(s+r = \text{dim}(X)\) and \(0 \leq a_1 \leq \ldots \leq a_r\) are integers. In this case, tangent bundle will be always unstable with respect to any polarization whenever \((a_1, \ldots, a_r) \neq (0, 0, \ldots, 0, 1)\) (see Theorem \ref{M1}). When \((a_1, \ldots, a_r)\,=\, (0, 0, \ldots, 0, 1)\), we give a necessary and sufficient condition for (semi)stability of tangent bundle with respect to any polarization (see Theorem \ref{M3}). As a corollary we give a complete answer to (semi)stability of tangent bundle with respect to anticanonical divisor \(-K_X\) for any Fano toric variety with Picard number $2$ (see Corollary \ref{stabonPic2cor}). This generalizes a very recent result of \cite[Theorem 9.3]{biswas2018stability}.

By the result of Kobayashi \cite{kobayashi} and L$\ddot{\text{u}}$bke \cite{lubke}, stability of tangent bundle with respect to $-K_X$ for a nonsingular Fano variety is considered to be algebraic geometric analogue of 
existence of K$\ddot{\text{a}}$hler-Einstein metric on a smooth manifold. It is an open question if tangent bundle of a nonsingular Fano variety with Picard number $1$ is stable with respect to $-K_X$. Though the conjecture is known for many cases (see \cite{Ramanan},\cite{PW},\cite{Hwang},\cite{tian},\cite{fahlaoui} etc.), this question is wide open in general.  If the Picard number is $> 1$, tangent bundle is not necessarily stable due to the geometry of contractions of extremal rays, for $3$-folds this has been studied completely by Steffens \cite{steffens}. By the result described in previous paragraph we have settled this question completely for any nonsingular toric Fano variety with Picard number $2$. In Section 5, we study (semi)stability of tangent bundle of nonsingular Fano toric  $4$-folds with Picard number $3$. In \cite{batyrev}, \cite{sato}, Batyrev and subsequently Sato have given a complete list of isomorphism classes of all nonsingular Fano toric $4$-folds. There are in total one hundred twenty four non-isomorphic toric Fano $4$-folds. Among them there are twenty eight isomorphism classes with Picard number $3$, out of which eight are toric blow ups and nineteen of them are projectivizations of split vector bundle over a toric variety and the last one is neither a blow up nor a projectivization of splittable vector bundle. Among them six are stable, three are strictly semistable and rest of them have unstable tangent bundle with respect to the anticanonical polarization (see Table 1, Section 5). 

In Section 6, first we consider a class of nonsingular projective toric varieties, known as Bott towers.  A Bott tower of height $n$ 
\begin{equation*}
M_n \rightarrow M_{n-1} \rightarrow \cdots \rightarrow M_2 \rightarrow M_1 \rightarrow M_0=\{ \text{point} \}
\end{equation*}
 is defined inductively as an iterated projective bundle so that each stage $M_k$ of the tower is of the form $\mathbb{P}(\mathcal{O}_{M_{k-1}}  \oplus \mathcal{L}$) for an arbitrarily chosen line bundle $\mathcal{L}$ over the previous stage $M_{k-1}$. Bott towers were shown to be deformation equivalent to Bott-Samelson varieties by Grossberg and Karshon in \cite{grossbergkarshon}. In this section we construct a collection (finite) of indecomposable rank $2$ equivariant vector bundles over $M_k$ ($k \ge 2$) (Proposition \ref{existance of indecomp b on Bott T}). Further, we show that among these collection there exists a stable rank $2$ vector bundle over $M_k$ ($k \ge 2$), for certain Chern classes with respect to a suitable choice of polarization (Proposition \ref{existance of stable b on Bott T}). The approach here is to construct a dimension $2$ filtration corresponding to an equivariant vector bundle such that it does not have any induced subfiltration which violates the stability. In the next subsection we consider pseudo-symmetric toric varieties which are very important examples of toric varieties appears in classification of projective Fano toric varieties (see \cite{sato} for details). In \cite{Alexandru}, Cotignoli and Sterian have constructed indecomposable rank $2$ vector bundle over pseudo-symmetric toric Fano varieties other than product of $\mathbb P^{1}$'s. It is not clear to us if they are equivariant or not. In this subsection, we construct a collection of rank $2$ equivariant indecomposable vector bundles on any pseudo-symmetric toric Fano variety (Proposition \ref{pseudo}).

We summarize our results as follows. 
\begin{enumerate}
	\item Classification of equivariant torsion free (respectively, reflexive) subsheaves of a given equivariant torsion free (respectively, reflexive) sheaf. 
	\item A simple proof for stability of tangent bundle of a projective space.
	\item A necessary and sufficient condition for (semi)stability of tangent bundle of a nonsingular projective toric variety with Picard number $2$ with respect to any polarization. 
	\item A complete answer to the question of (semi)stability of tangent bundle for Fano toric $4$-folds with Picard number $3$ from the classification due to Batyrev \cite{batyrev}.
	\item Construction of equivariant indecomposable as well as stable rank $2$ vector bundles over a Bott tower.  
	\item Construction of equivariant indecompossable rank $2$ vector bundles on pseudo-symmetric Fano toric variety. 
\end{enumerate}

After this work got complete, we came to know about the work of Hering,  Nill  and S\"{u}ss, where they have studied (semi)stability of tangent bundles of smooth toric variety for Picard number $2$. Their result \cite[Theorem 1.4]{Hering_Nill_Suss} matches with our result in section 4.

\noindent
{\bf Acknowledgements:}  The last author thank the Council of Scientific and Industrial Research (CSIR) for their financial support. The second author would like to thank NBHM and SERB for partial financial support. 

\section{Preliminaries and some basic facts}

In this section we briefly review some basic definitions and results on toric varieties and torus equivariant sheaves which will be needed later. 

\subsection{Toric Varieties}

Let  \(T \cong (k^*)^n\) be the \(n\)-dimensional algebraic torus, where \(k\) is any algebraically closed field. A toric variety \(X\) of dimension \(n\) is a normal variety which contains \(T\) as an open dense subset such that the torus multiplication extends to an action of \(T\) on \(X\). Toric varieties have a rich combinatorial structure which arises due to the action of the dense torus. We recall some basic facts about toric varieties which will be used in subsequent sections. For more details see \cite{Cox}, \cite{Ful} and \cite{Oda}.

Let \(M=\text{Hom}(T, k^*)  \cong \Z^n\) be the character lattice of \(T\) and \(N=\text{Hom}(M, \Z)\) be the dual lattice. We denote by \(\langle , \rangle : M \times N \rightarrow \Z \) the natural pairing between \(M\) and \(N\). Let $\Delta$ be a fan in \(N\otimes_{\Z} \R \) which defines a nonsingular projective toric variety \(X=X(\Delta)\) of dimension \(n\) over \( k\) under the action of \(T\). Let \(S_{\sigma}:=\sigma^{\vee} \cap M \) be the affine semigroup and  \(U_{\sigma}:=\text{Spec }k[S_{\sigma}]\) be the affine toric variety corresponding to a cone $\sigma \in \Delta$. Let \(x_{\sigma}\) denote the distinguished point of \(U_{\sigma}\) (see \cite[Section 2.1]{Ful}). Then \(T_{\sigma}:=\text{Stab}(x_{\sigma})\) is a subtorus of \(T\) with character lattice \(M_{\sigma}:=M / S_{\sigma}^{\perp}\), where \(S_{\sigma}^{\perp}:=\sigma^{\perp} \cap M \). The \(T\)-invariant closed subvariety corresponding to a cone $\sigma$ is denoted by \(V(\sigma)\), which is the closure of the \(T\)-orbit through \(x_{\sigma}\) and \(\text{dim }V(\sigma)=n-\text{dim }\sigma\). We denote the set of all cones of dimension \(d\) in $\Delta$ by $\Delta(d)$. Elements of $\Delta(1)$ are called rays. Each ray $\rho$ has a unique minimal ray generator which we denote by \(v_{\rho}\). Sometime we will use the ray $\rho$ and its minimal generator $v_{\rho}$ interchangeably. Each ray $\rho$ corresponds to a \(T\)-invariant prime divisor \(D_{\rho}:=V(\rho)\). 

The following proposition on toric intersection theory will be extensively used in latter sections, while computing slope of equivariant sheaves over a toric variety.

\begin{prop}$($\cite[Corollary 6.4.3, Lemma 6.4.4, Lemma 12.5.2]{Cox}$)$ \label{wallreln}
Let $X(\Delta)$ be a nonsingular projective toric variety. To compute the intersection product \(D_{\rho} \cdot V(\tau) \),  where \(\tau \in \Delta(n-1)\) is a wall, i.e. $\tau=\sigma \cap \sigma'$ for some \(\sigma, \sigma' \in \Delta(n)\), write 
\(	\sigma  =\text{Cone}(v_{\rho_1}, \ldots, v_{\rho_n}), 
	\sigma' =\text{Cone}(v_{\rho_2}, \ldots, v_{\rho_{n+1}}) \text{ and } 
	\tau =\text{Cone}(v_{\rho_2}, \ldots, v_{\rho_n}).\)
Then \(v_{\rho_1}, \ldots, v_{\rho_{n+1}} \) satisfy the linear relation \(v_{\rho_1} + \sum\limits_{i=2}^n b_i v_{\rho_i} + v_{\rho_{n+1}}=0, \ b_i \in \Z \), called the wall relation. Then 

 \begin{center}
 	$D_{\rho} \cdot V(\tau)= \left\{ \begin{array}
	{r@{\quad \quad}l}
	0 & \text{ for all }  \rho \notin \{\rho_1, \ldots, \rho_{n+1}\} \\ 
	1 & \text{ for } \rho \in \{\rho_1, \rho_{n+1}\} \\
	b_i &  \text{ for } \rho \in \{\rho_2, \ldots, \rho_{n}\}.
\end{array} \right. 
$
 \end{center}
More generally, for distinct rays $\rho_1, \ldots, \rho_d \in \Delta(1)$ we have 

\begin{center}
	$D_{\rho_1} \cdot D_{\rho_2} \cdots D_{\rho_d}= \left\{ \begin{array}{ll}
	[V(\sigma)]  \in A^{\bullet}(X)& \text{ if } \sigma=\text{Cone}(\rho_1, \ldots, \rho_d) \in \Delta\\
	0 & \text{ otherwise. }  
	\end{array} \right. 
	$
\end{center}
Here \(	[V(\sigma)]\) denotes the rational equivalence class of \(V(\sigma)\) in the Chow ring \(A^{\bullet}(X)\).
\end{prop}

We recall the fan structures of the following two classes of toric varieties which will be used in Section \ref{sec:stability-of-tangent-bundle-on-fano-4-folds-with-picard-number-3} while studying (semi)stability of tangent bundle of toric Fano 4-folds.

\subsubsection{Projectivization of direct sum of line bundles on toric varieties}\label{pbundle}
 Let $D_0, D_1, \ldots, D_m$ be \(T\)-invariant Cartier divisors on a nonsingular toric variety $X=X(\Delta)$. Then the fan $\Delta'$ of $X'= \mathbb{P}(\mathcal{O}_X(D_0) \oplus \mathcal{O}_X(D_1) \oplus \cdots \oplus \mathcal{O}_X(D_m))$ is described as following (see \cite[Page 58]{Oda} for details). Let $h_0, h_1, \ldots, h_m$ be the $\Delta$-linear support functions corresponding to $D_0, D_1, \ldots, D_m$ respectively (see \cite[Section 2.1]{Oda}). Choose the standard $\Z$-basis $\{e_1, \ldots, e_m\}$ of $\R^m$ and let $e_0=-e_1- \ldots - e_m$. Consider the $\R$-linear map $\Phi: \R^n \rightarrow \R^n \oplus \R^m, \text{ given by } y  \mapsto (y, -\sum_{j=0}^m h_j(y)e_j).$ Now let $\tilde{\sigma}_i = \text{Cone}(e_0, \ldots, \widehat{e_i}, \ldots, e_m)$ for each $0\leq i \leq m$ (henceforth by $\widehat{e}_i$ we mean that $e_i$ is omitted from the relevant collection). Let $\tilde{\Delta}$ be the fan in $\R^m$ generated by $\tilde{\sigma}_i$ for $0\leq i \leq m$. Then $\Delta'=\{\Phi(\sigma) + \tilde{\sigma} : \sigma \in \Delta, \tilde{\sigma}\in \tilde{\Delta}\}$.  

\subsubsection{Blowup of a toric variety along an invariant subvariety}\label{blowup}
Let \(X=X(\Delta)\) be a nonsingular toric variety. Let $\tau$ be a cone in $\Delta$ and \(\tilde{X}=Bl_{V(\tau)}(X)\) be the blowup of \(X\) along the \(T\)-invariant subvariety \(V(\tau)\). Let $\tilde{\Delta}$ be the fan  corresponding to \(\tilde{X}\). Then \(\tilde{\Delta}=\{\sigma \in \Delta : \tau \nsubseteq \sigma\} \cup \bigcup\limits_{\tau \preceq \sigma} \Delta^*_{\sigma}(\tau)\), where \(\Delta^*_{\sigma}(\tau)=\{\text{Cone}(A) : A \subseteq \{u_{\tau}\} \cup \sigma(1) , \tau(1) \nsubseteq A \}\) for any cone $\sigma$ containing $\tau$, denoted as $\sigma \succeq \tau $ and \(u_{\tau}=\sum\limits_{\rho \in \tau(1)} v_{\rho}\) (see \cite[Definition 3.3.17]{Cox}).

\subsection{Equivariant sheaves}\label{sec:equivariant-sheaves}

We briefly recall the combinatorial description of equivariant sheaves on toric varieties introduced by Perling \cite{perling}  which will be used in latter sections. 
 
Let $X\,=\,X(\Delta)$  be a toric variety corresponding to a fan $\Delta$. For each cone $\sigma \in \Delta$, define a relation \(\leq_{\sigma }\) on \(M\) by setting \(m \leq_{\sigma } m'\) if and only if \(m'-m \in S_{\sigma}\). We write \(m <_{\sigma} m'\) if \(m \leq_{\sigma } m'\)  holds but  \(m' \leq_{\sigma } m\) does not hold.  A $\sigma$-family, denoted by \(\widehat{E}^{\sigma}\), is a family of \(k\)-vector spaces \(\{E^{\sigma}_m\}_{m \in M}\) together with a vector space homomorphism \(\chi^{\sigma}_{m,m'} : E^{\sigma}_{m} \rightarrow E^{\sigma}_{m'}\), whenever \(m \leq_{\sigma } m'\) such that \(\chi^{\sigma}_{m,m}=1\) and \(\chi^{\sigma}_{m,m''}=\chi^{\sigma}_{m',m''} \circ \chi^{\sigma}_{m,m'} \) for every triple \(m \leq_{\sigma } m' \leq_{\sigma } m''\).

\begin{rmk}\label{RM1}{\rm
	Note that \(\chi^{\sigma}_{m,m'}\) is an isomorphism whenever \(m'-m \in S_{\sigma}^{\perp}\) $($see \cite[Lemma 5.3]{perling}$)$, hence we restrict our attention to $\sigma$-families having \(\chi^{\sigma}_{m,m'}=1\) (and hence \(E^{\sigma}_{m} = E^{\sigma}_{m'}\)) for all \(m'-m \in S_{\sigma}^{\perp}\).}
\end{rmk}

Let $\mathcal{E}$ be an equivariant quasi-coherent sheaf on the toric variety \(X=X(\Delta)\) (see \cite{perling} for detail definition of equivariant sheaves). The \(T\)-action on $\mathcal{ E }$ gives rise to an isomorphism \(\Phi_t: t^* \mathcal{E } \stackrel{\cong} \longrightarrow  \mathcal{E }\) for all \(t \in T\). This induces an action of \(T\) on the space of global sections \(E^{\sigma}:=\Gamma(U_{\sigma}, \mathcal{E})\) given by \(t \cdot f=\Phi_t(t^* f) \), where \(f \in E^{\sigma}\) and \(t^* f \in \Gamma(U_{\sigma}, t^*\mathcal{E}) \) is its canonically lifted section. Hence we get the \(T\)-isotypical decomposition \(E^{\sigma}=\bigoplus\limits_{m \in M} E^{\sigma}_{m}\), which makes \(E^{\sigma}\) naturally an \(M\)-graded \(k[S_{\sigma}]\)-module as follows. Recall that the action of the torus on the affine open variety \(U_{\sigma}\) induces the \(T\)-isotypical decomposition $k[S_{\sigma}]=\bigoplus\limits_{m \in S_{\sigma}} k \chi^m$. Then the \(M\)-graded \(k[S_{\sigma}]\)-module structure on \(E^{\sigma}\) is given by the following multiplication: \(\chi^{\sigma}_{m,m'} : E^{\sigma}_{m} \rightarrow E^{\sigma}_{m'}, ~ e \mapsto \chi^{m'-m} \cdot e,\) where \(m, m' \in M\) and \(m'-m \in S_{\sigma}\). Then the following three categories are equivalent (see \cite[Proposition 5.5]{perling}):
\begin{enumerate}[(i)]
	\item Equivariant quasi-coherent sheaves over \(U_{\sigma}\),
	
	\item \(M\)-graded \(k[S_{\sigma}]\)-modules with morphisms of degree 0, 
	
	\item $\sigma$-families. 
\end{enumerate}

For each pair $\tau \preceq \sigma$, we denote by \(i_{\tau \sigma} : U_{\tau} \hookrightarrow U_{\sigma}\) the inclusion. Let \(\widehat{E}^{\sigma}\) be a $\sigma$-family. We denote by \(E^{\sigma}:=\bigoplus_{m \in M} E^{\sigma}_m \) the corresponding \(M\)-graded \(k[S_{\sigma}]\)-module. The pull back \(i_{\tau \sigma}^*E^{\sigma}=E^{\sigma} \otimes_{k[S_{\sigma}] } k[S_{\tau}]\) is naturally an \(M\)-graded \(k[S_{\tau}]\)-module (see \cite[Section 5.2]{perling}) and hence corresponds to a $\tau$-family (by the above equivalence of categories), which we denote by \(i_{\tau \sigma}^* \widehat{E}^{\sigma}\).

Going from affine toric varieties to general toric varieties, following notion of $\Delta$-families was introduced by Perling.

\begin{defn} \cite[Definition 5.8]{perling} \label{delta fam} A collection $\{\widehat{E}^{\sigma} \}_{\sigma \in \Delta}$ of $\sigma$-families is called a $\Delta$-family, denoted by $\widehat{E}^{\Delta}$, if for each pair $\tau \preceq \sigma$ there exists an isomorphism of families $\eta_{\tau \sigma}: i_{\tau \sigma}^{*}(\widehat{E}^{\sigma}) \cong \widehat{E}^{\tau} $ such that for each triple $\rho \preceq \tau \preceq \sigma$ there is an equality $\eta_{\rho \sigma}=\eta_{\rho \tau} \circ  i_{\rho \tau}^*\eta_{\tau \sigma}$.\\
	A morphism of $\Delta$-families is a collection of morphisms $\{ \widehat{\phi}^{\sigma}: \widehat{E}^{\sigma} \rightarrow \widehat{F}^{\sigma}\}_{\sigma \in \Delta}$ such that for all $\sigma, \tau$ and $\tau \preceq \sigma$, the following diagram commutes:
	
	\begin{center}
		{\footnotesize
			\begin{tikzpicture}[description/.style={fill=white,inner sep=2pt}]
			
			\matrix (m) [matrix of math nodes, row sep=3em,
			column sep=3em, text height=1.5ex, text depth=0.25ex]
			{ i_{\tau \sigma}^{*}(\widehat{E}^{\sigma})  & &  i_{\tau \sigma}^{*}(\widehat{F}^{\sigma}) \\
				\widehat{E}^{\tau} & & \widehat{F}^{\tau}\\ };
			\path[->] (m-1-1) edge node[auto] {$ i_{\tau \sigma}^{*} \widehat{\phi}^{\sigma}$}(m-1-3);
			\path[->] (m-2-1) edge node[below] {$\widehat{\phi}^{\tau}$}(m-2-3);
			\path[->] (m-1-1) edge node[auto] {$ \eta^E_{\tau \sigma}$}(m-2-1);
			\path[->] (m-1-3) edge node[auto] {$\eta^F_{\tau \sigma} $} (m-2-3);
			
			\end{tikzpicture}
		}
	\end{center}
	
\end{defn}

\begin{thm}\cite[Theorem 5.9]{perling}
The category of $\Delta$-families is equivalent to the category of equivariant quasi-coherent sheaves on \(X\).
\end{thm}

In order to classify equivariant coherent sheaves, Perling introduced the following notion of finite $\Delta$-family.

\begin{defn}\label{finitedelfam} \cite[Definition 5.10]{perling}
	A $\sigma$-family 	\(\widehat{E}^{\sigma}\) is said to be finite if:
	\begin{enumerate}[(i)]
		\item \(E^{\sigma}_m \)'s are finite dimensional for all \(m \in M\),
		\item for each chain \(\ldots <_{\sigma} { m_{i-1}} <_{\sigma} { m_i} <_{\sigma} \ldots \) of characters in \(M\) there exists an \(i_0 \in \Z \) such that \(E^{\sigma}_{m_i}=0\) for  all \(i < i_0\),
		\item there are only finitely many vector spaces \(E^{\sigma}_m\) such that the map	\(\bigoplus\limits_{m' {<_{\sigma}} m} E^{\sigma}_{m'} \longrightarrow E^{\sigma}_m \), defined by the summation of the \(\chi^{\sigma}_{m',m}\), is not surjective.
	\end{enumerate}
A \(\Delta\)-family is said to be finite if all of its $\sigma$-families are finite.
\end{defn}

\begin{prop}\cite[Proposition 5.11]{perling}
	A quasi-coherent equivariant sheaf is coherent if and only if its associated $\Delta$-family is finite.	
\end{prop}

Note that given a \(\sigma\)-family \(\widehat{E}^{\sigma}\), the collection \(\{E^{\sigma}_m, \chi^{\sigma}_{m,m'}\}\) forms a directed system of vector spaces. Let \(\mathbf{E}^{\sigma}:= \varinjlim\limits_{m \in M} \widehat{E}^{\sigma}\). Any element of \(\mathbf{E}^{\sigma}\) can be written as equivalence classes \([e,m]\), where \(e \in E^{\sigma}_m\) and \([e,m]=[e',m']\) if and only if there exists \(m''\in M\) satisfying \(m, m' \leq_{\sigma }m''\) such that \(\chi^{\sigma}_{m,m''}(e)=\chi^{\sigma}_{m',m''}(e')\).

   If $\mathcal{E}$ is an equivariant torsion free sheaf of rank \(r\) on \(X\). Then for
all \(\sigma \in \Delta\) and \(m \leq_{\sigma } m'\), the maps in the following diagram are injective (see \cite[Proposition 5.13]{perling}):
\begin{equation}\label{torsionfree_maps}
	\begin{tikzpicture}[descr/.style={fill=white,inner sep=2.5pt}]
\matrix (m) [matrix of math nodes, row sep=3em,
column sep=3em]
{ E^{\sigma}_m & E^{\sigma}_{m'} \\
	& \mathbf{E}^{\sigma} \\ };
\path[->,font=\scriptsize]
(m-1-1) edge node[auto] {$ \chi^{\sigma}_{m, m'} $} (m-1-2)
edge node[auto] {} (m-2-2)
edge node[auto,swap] {} (m-2-2)
(m-1-2)  edge node[auto] {} (m-2-2);
\end{tikzpicture}
\end{equation}

Moreover, the restriction map \(\Gamma(U_{\sigma}, \mathcal{E}) \rightarrow \Gamma(U_{\tau}, \mathcal{E}) \) is injective whenever \(\tau \preceq  \sigma\). Now let \(m_{\tau}\) be an integral element of the interior of $\sigma^{\vee} \cap \tau^{\perp} $ such that \(S_{\tau} = S_{\sigma} + \Z_{\geq 0}(-m_{\tau} )\). Note that \(\Gamma(U_{\tau}, i_{\tau \sigma}^*(\mathcal{E} |_{U_{\sigma}})) =\Gamma(U_{\sigma}, \mathcal{E}) \otimes_{k[S_{\sigma}]}  k[S_{\tau}] =\Gamma(U_{\sigma}, \mathcal{E}) \otimes_{k[S_{\sigma}]}  k[S_{\sigma}][\chi^{-m_{\tau}}]\). So there exists a natural inclusion of $\sigma$-families \(\alpha_{\tau \sigma} : \widehat{E}^{\sigma} \rightarrow i_{\tau \sigma}^* \widehat{E}^{\sigma}\). Hence the following composition 
\begin{equation}\label{TF1}
\widehat{E}^{\sigma} \stackrel{\alpha_{\tau \sigma}} \longhookrightarrow i_{\tau \sigma}^* \widehat{E}^{\sigma} \xrightarrow[\cong]{\eta_{\tau \sigma }} \widehat{E}^{\tau}
\end{equation}
is injective (see Definition \ref{delta fam}, \cite[Proposition 5.14]{perling}), which further induces a natural injection 
\begin{equation}\label{TF2}
\begin{split}
\tilde{\eta}_{\tau \sigma} :  ~\mathbf{E}^{\sigma}  \longhookrightarrow \mathbf{E}^{\tau}, \
  [(e, m)]  \mapsto [(\eta^m_{\tau \sigma} \circ \alpha^m_{\tau \sigma})(e), m)]. 
\end{split}
\end{equation}

Moreover, the system of vector spaces $\mathbf{E}^{\sigma}$ together with the homomorphisms \(\tilde{\eta}_{\tau \sigma} \) for \(\tau \preceq \sigma\), forms a directed partially ordered family whose direct limit can be identified with $\mathbf{E}^0$, here \(0\) denotes the zero cone. Note that we have the following isotypical decomposition $$\Gamma(T, \mathcal{E})=k[M]^r=\bigoplus\limits_{m \in M} (\underbrace{k \chi^m \oplus \cdots \oplus k \chi^m}_{r \text{ times}})$$ and the homomorphisms \(\chi^0_{m ,m'}\) are isomorphisms. Hence we can identify \(\mathbf{E}^0=k \chi^m \oplus \cdots \oplus k \chi^m\) ($r$ times) and thus it is a finite dimensional vector space of dimension \(r\). Also the natural inclusions \(\mathbf{E}^{\sigma} \hookrightarrow \mathbf{E}^0 \) obtained in \eqref{TF2} are isomorphisms \(\mathbf{E}^{\sigma} \cong \mathbf{E}^0 \) (see \cite[Corollary 5.16]{perling}). Thus all the vector spaces in the $\Delta$-family \(\widehat{E}^{\Delta}\) can be realized as vector subspaces of $\mathbf{E}^0$.

The above technical reformulation of the finite \(\Delta\)-family leads to the following definition of family of multifiltrations. 

\begin{defn}\label{multfilt}\cite[Definition 5.17]{perling}
	Let $\Delta$ be a fan, \(V\) be a finite-dimensional \(k\)-vector space, and let for each \(\sigma \in \Delta \) a set of vector subspaces \(\{E^{\sigma}_m\}_{m \in M } \) of \(V\) be given. We say that this system is a family of multifiltrations of \(V\) if: 
	\begin{enumerate}
		\item[(i)] For each \(\sigma \in \Delta \) and \(m \leq_{\sigma } m' \), \(E^{\sigma}_m\) is contained in \(E^{\sigma}_{m'}\).
		\item[(ii)] \(V=\bigcup\limits_{m \in M} E^{\sigma}_{m} \) for each $\sigma \in \Delta$.
		\item[(iii)] For each chain \(\ldots <_{\sigma} { m_{i-1}} <_{\sigma} { m_i} <_{\sigma} \ldots \)of characters in \(M\) there exists an \(i_0 \in \Z \) such that \(E^{\sigma}_{m_i}=0\) for  \(i < i_0\).
		\item[(iv)] For every $\sigma \in \Delta$ there exist only finitely many vector spaces \(E^{\sigma}_m\) such that \(E^{\sigma}_m \nsubseteq  \sum\limits_{m' {<_{\sigma}} m} E^{\sigma}_{m'}\).
		\item[(v)] $($Compatibility condition.$)$ For each $\tau \preceq \sigma$ with \(S_{\tau} = S_{\sigma} + \Z_{\geq 0}(-m_{\tau} )\) we consider with respect to the
		preorder \(\leq_{\sigma }\) the ascending chains \(m + i \cdot m_{\tau}\) for \(i \geq 0\). By condition (iv) and because \(V\) is finite dimensional the sequence of subvector spaces \(E^{\sigma}_{m+ i \cdot m_{\tau}}\) necessarily becomes stationary for some \(i^{\tau}_m \in \Z \). We require that \(E^{\tau}_m= E^{\sigma}_{m+ i^{\tau}_m \cdot m_{\tau}}\) for all \(m \in M\).
	\end{enumerate}	

A morphism between families of multifiltrations \(\{E^{\sigma}_m\}_{m \in M } \) and \(\{F^{\sigma}_m\}_{m \in M } \) is a homomorphism of the corresponding ambient vector spaces $\phi : \mathbf{E}^0 \rightarrow \mathbf{F}^0$ which is compatible with these multifiltratons, i.e.  $\phi(E^{\sigma}_m) \subseteq F^{\sigma}_m$.
\end{defn}

\begin{rmk}\label{rmultfilt}{\rm
Note that in Definition \ref{multfilt}, the condition \((iv)\) can be replaced with the following:

\((iv)'\) For every $\sigma \in \Delta$ there exist only finitely many vector spaces \(E^{\sigma}_m\) such that \(E^{\sigma}_m \nsubseteq  \bigcup\limits_{m' {<_{\sigma}} m} E^{\sigma}_{m'}\) $($see \cite[Definition 4.19]{perlingT}$)$.}
\end{rmk}

\begin{thm}\cite[Theorem 5.18]{perling}\label{per5.18}
	The category of equivariant torsion free sheaves is equivalent to the category of families of multifiltrations of finite-dimensional vector spaces.
\end{thm}

A coherent sheaf $\mathcal{ E }$ on \(X\) is reflexive if $\mathcal{ E }$ is isomorphic to its double dual $\mathcal{ E }^{* *}$. Equivalently, a coherent sheaf $\mathcal{ E }$ on \(X\) is reflexive if and only if $\mathcal{ E }$ is torsion free and for each open subset \(U \subset X \) and each closed subset \(Y \subset U \) of codimension \(\geq 2\), the restriction map \( \Gamma(U, \mathcal{ E }) \rightarrow \Gamma(U  \setminus Y, \mathcal{ E })\) is bijective (see \cite[Proposition 1.6]{SRS}). Dual of any coherent sheaf is an example of reflexive sheaf (see \cite[Corollary 1.2]{SRS}). Note that any rank 1 reflexive sheaf is locally free. Now let $\mathcal{E}$ be an equivariant reflexive sheaf on \(X=X(\Delta)\). Choose \(Y=\bigcup\limits_{\text{dim} \tau \geq 2} V(\tau)\), a closed subset of \(X\) of codimension at least two. Then \(\Gamma(X, \mathcal{ E })=\Gamma(X \setminus Y, \mathcal{ E })=\Gamma(X(\Delta_1), \mathcal{ E })\), where \(\Delta_1=\Delta(0) \cup \Delta(1) \). In particular, for the affine toric variety $U_{\sigma}$, we have \(\Gamma(U_{\sigma}, \mathcal{ E })=\Gamma\left( \bigcup_{\rho \in \sigma(1)} U_{\rho}, \mathcal{ E }\right) = \bigcap\limits_{\rho \in \sigma(1)}  \Gamma( U_{\rho}, \mathcal{ E })\) as vector subspaces of $\Gamma(T, \mathcal{ E })$. Hence for each graded component of degree \(m\) we have \(\Gamma(U_{\sigma}, \mathcal{ E })_m = \bigcap\limits_{\rho \in \sigma(1)}  \Gamma( U_{\rho}, \mathcal{ E })_m\). Therefore, as vector subspaces of $\mathbf{E}^0$, we have \(E^{\sigma}_m=\bigcap\limits_{\rho \in \sigma(1)} E^{\rho}_m.\) It follows that the compatibility condition (v) of Definition \ref{multfilt} is redundant. Thus an equivariant reflexive sheaf is completely determined by the family of multifiltrations \(\{E^{\rho}_m\}_{m \in M, \rho \in \Delta(1)}\) of the vector space $\mathbf{E}^0$. Note that there is a canonical identification of \(M / S_{\rho}^{\perp}\) with \(\Z\) via the map \(m \mapsto \langle m, v_{\rho} \rangle \). Hence identifying  \(E^{\rho}_m=E^{\rho}(\langle m, v_{\rho} \rangle)\), we get an increasing full filtrations: $0 \subseteq \ldots \subseteq E^{\rho}(i) \subseteq E^{\rho}(i+1) \subseteq \ldots \subseteq \mathbf{E}^0$.

The following theorem shows that any equivariant reflexive sheaf arises from such filtrations.
\begin{thm}\cite[Theorem 5.19]{perling}\label{reflexive}
	The category of equivariant reflexive sheaves on a toric variety \(X\) is equivalent to the category of vector spaces with full filtrations associated to each ray in \(\Delta(1)\). The morphisms in this category are vector space homomorphisms which are compatible with the filtrations in the $\Delta$-family sense.
\end{thm}

The following proposition provides a necessary and sufficient condition for an equivariant reflexive sheaf to be locally free.

\begin{prop}\cite[Proposition 4.24]{perlingT} \label{compatibility for loc free sheaf}
Let $\mathcal{E }$ be an equivariant reflexive sheaf of rank \(r\) over \(X\) with corresponding filtration $(\mathbf{E}^0, \{E^{\rho}(i)\}_{\rho \in \Delta(1)})$. Then $\mathcal{E }$ is locally free if and only if for each $\sigma \in \Delta$ there is an action of \({T}_{\sigma}\) on \(\mathbf{E}^0\) which gives a decomposition of \(\mathbf{E}^0\) into \({T}_{\sigma}\)-eigenspaces \(\mathbf{E}^0=\bigoplus\limits_{m \in M/ S_{\sigma}^{\perp} } \mathbf{E}^0_m  \) such that 
\begin{equation}\label{compatibility for loc free sheaf1}
E^{\rho}(i)=\bigoplus\limits_{\substack{m \in M/ S_{\sigma}^{\perp}\\ \langle m, v_{\rho} \rangle \leq i }} \mathbf{E}^0_m .  
\end{equation}
\end{prop}

 \begin{rmk}\label{dist}{\rm 
 	Recall that a family of linear subspaces $\{V_\lambda \}_{\lambda \in \Lambda}$ of a finite dimensional vector space $V$ is said to form a distributive lattice if, there exists a basis $B$ of $V$ such that $B \cap V_\lambda$ is a basis of $V_\lambda$ for every $\lambda \in \Lambda$. 
 When \(X\) is nonsingular, the compatibility condition of locally free sheaves \eqref{compatibility for loc free sheaf1} in Proposition \ref{compatibility for loc free sheaf} is equivalent to the following: for each $\sigma \in \Delta$, the collection of subspaces $(\mathbf{E}^0, \{E^{\rho}(i)\}_{\rho \in \sigma(1)})$ forms a distributive lattice	(cf. \cite[Remark $2.2.2$]{kly}, arguments following Theorem 2.1.1 in \cite{bj}).}
 \end{rmk}

The first Chern class of an equivariant coherent sheaf can be expressed using its associated $\Delta$-family as follows:

\begin{prop}\cite[Corollary 3.18]{kool}\label{chern}
	Let \(X=X(\Delta)\) be a nonsingular projective toric variety. Let $\mathcal{E}$ be an equivariant coherent sheaf with associated $\Delta$-family $\widehat{E}^{\Delta}$. Then we have
	\[c_1(\mathcal{E})=-\sum_{\rho \in \Delta(1), i \in \Z} i \text{ dim } E^{[\rho]}(i) D_{\rho} ,\]
	where \(E^{[\rho]}(i)= E^{\rho}(i)/E^{\rho}(i-1)\) and \(E^{\rho}(i)=E^{\rho}_m\) such that \( \langle m, v_{\rho} \rangle=i\).
\end{prop}

\begin{ex}[Filtrations for line bundles]\cite[Section 4.7]{perlingT} \label{lb}{\rm
		Let $\mathcal{L}=\mathcal{O}_X(D)$ be a toric line bundle on $X$ for some $T$-invariant Cartier divisor $D=\sum_{\rho \in \Delta(1)}a_{\rho}D_{\rho}$, $a_{\rho} \in \Z$. Then the associated filtrations $(L, \{L^{\rho}(i)\}_{\rho \in \Delta(1)})$ are given by:
		\[ L^{\rho}(i) = \left\{ \begin{array}{ll}
		0 & i <- a_{\rho} \\ 
		
		 L(=k)& i \geq - a_{\rho}.
		\end{array} \right. \] }
\end{ex}

Next we obtain the filtrations for the dual of an equivariant torsion free sheaf following the proof of \cite[Proposition 4.24]{perlingT} which will be useful to obtain the filtrations of the tangent bundle from that of the cotangent bundle.

\begin{prop}\label{perdual}
	Let $\mathcal{E}$ be an equivariant torsion free sheaf with associated family of multifiltrations \( \{E^{\sigma}_m\}_{m \in M}  \) of the vector space $\mathbf{E}^0$. Then the filtrations associated to its dual reflexive sheaf $\mathcal{E}^*$ are given by \(\left(F, \{F^{\rho}(i)\}_{\rho \in \Delta(1)} \right) \), where \[F=(\mathbf{E}^0)^* \text{ and } F^{\rho}(i)=\left( \frac{\mathbf{E}^0}{E^{\rho}(-i-1)}\right) ^* \text{ for all } \rho \in \Delta(1), \]
	where \(E^{\rho}(i)=E^{\rho}_m\) for any \(m \in M\) satisfying \(\langle m, v_{\rho} \rangle =i\).
\end{prop}

\begin{proof}
	Since $\mathcal{E}$ is torsion free, the singularity set \(S(\mathcal{E})= \{x \in X : \mathcal{E}_x \text{ is not free over } \mathcal{O}_{X,x} \}\) is of codimension at least two. Now for any ray $\rho \in \Delta(1)$, \(U_{\rho}=T \cup O(\rho)\), where \(O(\rho)=T \cdot x_{\rho}\) is a orbit of dimension \(n-1\). Since $\mathcal{ E }$ is equivariant, it must be locally free over \(U_{\rho}\). Then \(E^{\rho}:=\Gamma(U_{\rho}, \mathcal{E})\) is an \(M\)-graded finitely generated free \( k[S_{\rho}] \)-module of rank \(r\) (see \cite[Proposition 5.20]{perling}). We can write
	\begin{equation}\label{prop4.24.1}
		E^{\rho}=\bigoplus_{j=1}^r k[S_{\rho}]e_j,
	\end{equation} where \(e_1, \ldots, e_r\) are homogeneous elements with \(\text{deg }e_j=m_j\) for \(j=1, \ldots, r\). Equivalently, the \(T\)-action on \(E^{\rho}\) is given by \(t \cdot e_j=  \chi^{m_j}(t) e_j\) for \(j=1, \ldots, r\) (see \cite[Proposition 2.31]{perlingT}).
	Set \(L^{\rho}_j:=k[S_{\rho}] e_j \text{ for } j=1, \ldots, r.\) Then for every \(j\) we have:
	\[ (L^{\rho}_j) _m= \left\{ \begin{array}{ll}
	0 & m_j \nleq_{\rho} m\\ 
	
	k \chi^{m-m_j} e_j & m_j \leq_{\rho } m.
	\end{array} \right. \] 
	We denote by \(\mathbf{L}^{\rho}_{j} \) the direct limit of the directed family \(\{(L^{\rho}_j) _m\}_{m \in M}\). Then we see that \((L^{\rho}_j) _m \cong \mathbf{L}^{\rho}_{j}\) for all \(m \geq_{\rho} m_j\). In particular, we have the identification \(\mathbf{L}^{\rho}_{j}=k e_j\).	Thus for \(i=\langle m, v_{\rho} \rangle\) we have
	\[ L^{\rho}_j(i)= \left\{ \begin{array}
	{r@{\quad \quad}l}
	0 & i < \langle m_j, v_{\rho} \rangle\\ 
	
	\mathbf{L}^{\rho}_{j}& i \geq \langle m_j, v_{\rho} \rangle.
	\end{array} \right. \] 
	There is an action of \(T\) on the vector space \(\mathbf{L}^{\rho}_{j}\) as follows:
	\[T \times \mathbf{L}^{\rho}_{j} \longrightarrow \mathbf{L}^{\rho}_{j}, \ (t, l) \longmapsto \chi^{m_j}(t)l .\]
	Since direct limits commutes with direct sum, we have \(\mathbf{E}^{\rho}=\bigoplus\limits_{j=1}^r \mathbf{L}^{\rho}_{j}  \) and thus we get a diagonal action of \(T\) on $\mathbf{E}^{\rho}$ as follows:
	\[T \times \mathbf{E}^{\rho} \longrightarrow  \mathbf{E}^{\rho}, \ (t, e) \longmapsto \text{ diag}(\chi^{m_1}(t), \ldots, \chi^{m_r}(t)) e .\]
	Furthermore we have 
	\begin{equation}\label{perdual1}
	E^{\rho}(i)=\bigoplus_{\langle m_j, v_{\rho} \rangle \leq i }  \mathbf{L}^{\rho}_{j}.
	\end{equation}
	Then using the following commutative diagram we can transfer the \(T\)-action to \(\mathbf{E}^{0}  \)
	\begin{equation}\label{filtdual1}
		\begin{tikzpicture}[description/.style={fill=white,inner sep=2pt}]
		\matrix (m) [matrix of math nodes, row sep=3em,
		column sep=2.5em, text height=1.5ex, text depth=0.25ex]
		{ \mathbf{E}^{\rho}    & &  \bigoplus_{j=1}^r \mathbf{L}^{\rho}_j  \\
			\mathbf{E}^{0}    & &   \bigoplus_{j=1}^r \mathbf{L}^{0}_j   \\ };
		\path[->]  (m-1-1) edge node[auto] {}(m-1-3);
		\path[->] (m-2-1) edge node[below] {}(m-2-3);
		\path[->] (m-1-1) edge node[auto] {$\cong$}(m-2-1);
		\path[->] (m-1-1) edge node[below] {$\cong$}(m-1-3);
		\path[->] (m-1-3) edge node[auto] {$\cong$} (m-2-3);
		\path[->] (m-2-1) edge node[below] {$\cong$} (m-2-3);
		\end{tikzpicture}	
	\end{equation}
	
	and from \eqref{perdual1} we have 
	\begin{equation}\label{perdual2}
	E^{\rho}(i)=\bigoplus_{\langle m_j, v_{\rho} \rangle \leq i }  \mathbf{L}^{0}_{j}.
	\end{equation}
	
	Now for the dual sheaf $\mathcal{E}^*$ we have \(F^{\rho}:=\Gamma(U_{\rho}, \mathcal{E}^*)= \text{Hom }_{k[S_{\rho}]}(E^{\rho}, k[S_{\rho}])\). Define \(f_j \in  F^{\rho}\) for \(j=1, \ldots,r\) by taking \(f_j(e_i)=\delta_{ij}\) (see \eqref{prop4.24.1}). Since the dual action of \(T\) on \(F^{\rho}\) is compatible with the \(M\)-graded \(k[S_{\rho}]\)-module structure of \(F^{\rho}\), it follows that the elements \(f_j\) are homogeneous of degree \(-m_j\) for \(j=1, \ldots, r\). Thus we have \(F^{\rho}= \bigoplus\limits_{j=1}^r L'^{\rho}_j\) where 	\(L'^{\rho}_j:=k[S_{\rho}] f_j.\) Then for every \(j\) we have:
	\[ (L'^{\rho}_j) _m= \left\{ \begin{array}{ll}
	0 & -m_j \nleq_{\rho} m\\ 
	
	k \chi^{m+m_j} f_j & -m_j \leq_{\rho } m
	\end{array} \right. \] 
	As before we see that \((L'^{\rho}_j) _m \cong \mathbf{L'}^{\rho}_{j}\) for all \(m \geq_{\rho} -m_j\). In particular, we have the identification \(\mathbf{L'}^{\rho}_{j}=k f_j=(\mathbf{L}^{\rho}_{j})^*\). Now the action of \(T\) on \(\mathbf{L'}^{\rho}_{j}\) is given by 
	\[T \times \mathbf{L'}^{\rho}_{j} \longrightarrow \mathbf{L'}^{\rho}_{j}, \ (t, l) \longmapsto \chi^{-m_j}(t)l .\]
	Now \(\mathbf{F}^{\rho}=\bigoplus\limits_{j=1}^r \mathbf{L'}^{\rho}_{j}  \) and hence \(\mathbf{F}^{\rho}=(\mathbf{E}^{\rho})^*\). As before we get a diagonal action of \(T\) on $\mathbf{F}^{\rho}$ as follows:
	\[T \times \mathbf{F}^{\rho} \longrightarrow  \mathbf{F}^{\rho}, \ (t, e) \longmapsto \text{ diag}(\chi^{-m_1}(t), \ldots, \chi^{-m_r}(t)) e .\]
	Using the dual diagram of \eqref{filtdual1}, we transfer the \(T\)-action to \((\mathbf{E}^0)^*\). Thus we have
	\begin{equation*}
		\begin{split}
			F^{\rho}(i)& =\bigoplus_{\langle -m_j, v_{\rho} \rangle \leq i }  \mathbf{L'}^{0}_{j}=\left(\bigoplus_{\langle m_j, v_{\rho} \rangle \leq -i-1 }\mathbf{L}^{0}_{j} \right)^{\perp}= \left( E^{\rho}(-i-1)\right)^{\perp}= \left( \frac{\mathbf{E}^{0}}{E^{\rho}(-i-1)}\right) ^* .
		\end{split}
	\end{equation*}
	Hence we get the desired filtrations for $\mathcal{E}^*$.
	\end{proof}

\begin{rmk}{\rm 
Let $\mathcal{E}$ and $\mathcal{F}$ be an equivariant reflexive sheaves with associated filtrations \(\left(\mathbf{E}^0, \{E^{\rho}(i)\}_{\rho \in \Delta(1)} \right) \) and \(\left(\mathbf{F}^0, \{F^{\rho}(i)\}_{\rho \in \Delta(1)} \right) \) respectively. Then arguing similarly as in the proof of Proposition \ref{perdual}, the filtrations associated to their sum and tensor product are given as follows:
\begin{align*}
&\left(\mathbf{E}^0  \oplus \mathbf{F}^0, \{(E \oplus F)^{\rho}(i) \}\right), \text{ where } (E \oplus F)^{\rho}(i) =E^{\rho}(i) \oplus F^{\rho}(i), \\
& \left(\mathbf{E}^0  \otimes \mathbf{F}^0, \{(E \otimes F)^{\rho}(i)\}\right)  , \text{ where } (E \otimes F)^{\rho}(i) =\sum_{s+ t=i} E^{\rho}(s) \otimes F^{\rho}(t).
\end{align*}

}
\end{rmk}

\begin{prop}\label{percotan}
	Let \(X=X(\Delta)\) be a nonsingular complete toric variety of dimension \(n\). Then the filtrations \( \left( \Omega, \{\Omega^{\rho}(i)\}_{\rho \in \Delta(1), i \in \Z} \right)  \) associated to the cotangent bundle \(\Omega_X\) are given by 
	
	\[ \Omega^{\rho}(i) = \left\{ \begin{array}{cc}
	0 & i \leq -1 \\ 
	\text{Span}(v_{\rho})^{\perp} & i=0 \\
	M \otimes_{\Z} {k} & i \geq 1.
	\end{array} \right. \] 
\end{prop}

\begin{proof}
	Let $\sigma \in \Delta(n)$ and write \(\sigma=\text{Cone}(v_1, \ldots, v_n) \). Since \(X\) is nonsingular, \(\{v_1, \ldots, v_n\}\) forms a \(\Z\)-basis of \(N \cong \Z^n \). Let \(u_1, \ldots, u_n\) be the corresponding dual basis of \(M\). Then we have \(U_{\sigma}=\text{Spec }k[\chi^{u_1}, \ldots,\chi^{u_n}]\). Set \(z_j=\chi^{u_j}\) for \(j=1, \ldots, n\). Then $E^{\sigma}:=\Gamma(U_{\sigma}, \Omega_X)$ is a free \(k[S_{\sigma}]\)-module generated by \(dz_1, \ldots, dz_n\).  The action of \(T\) on \(dz_j\) is given by \(t \cdot dz_j=\chi^{u_j}(t)dz_j. \) Thus as an \(M\)-graded \(k[S_{\sigma}]\)-module, \(E^{\sigma}\) is of the form	\(E^{\sigma}=\bigoplus_{j=1}^n k[S_{\sigma}] dz_j\).	Set \(L^{\sigma}_j=k[S_{\sigma}]dz_j \) for \(j=1, \ldots, n\). Then we have 
	\[ (L^{\sigma}_j) _m= \left\{ \begin{array}{ll}
	0 & u_j \nleq_{\sigma} m\\ 
	
	k \chi^{m-u_j} dz_j & u_j \leq_{\sigma } m.
	\end{array} \right. \] 
	Moreover, \(\mathbf{L}^{\sigma}_j=(L^{\sigma}_j) _m \) for all \(m \geq_{\sigma} u_j\). Hence taking \(m=u_j\), we can identify \(\mathbf{L}^{\sigma}_j=k ~ dz_j \text{ for } j=1, \ldots, n\). Note that from the proof of \cite[Proposition 4.24]{perlingT}, we have
	\begin{equation}\label{filtcotan1}
		E^{\rho}(i)=\bigoplus_{\langle u_j, v_{\rho} \rangle \leq i}  \mathbf{L}^{\sigma}_j, \text{ for all } \rho \in \sigma(1).
	\end{equation}

	Now \(E^{\rho}(-1)=E^{\rho}_m\) such that \(\langle m, v_{\rho}\rangle=-1 \). Hence from \eqref{filtcotan1}, we get \(E^{\rho}(-1)=0\), which implies \(E^{\rho}(i)=0 \text{ for all } i \leq -1 .\) Similarly, \(E^{\rho}(1)=E^{\rho}_m\) such that \(\langle m, v_{\rho}\rangle=1 \). Thus from \eqref{filtcotan1}, we get \(E^{\rho}(1)=k ~ dz_1 \oplus \cdots \oplus k ~ dz_n\) which can be identified with 
	\begin{equation}\label{percotan1}
	M \otimes_{\Z} k=k u_1 \oplus \cdots \oplus k u_n \text{ via the identification } dz_j \mapsto u_j .
	\end{equation}
	Thus we get \(E^{\rho}(i)=M \otimes_{\Z} k \text{ for all } i \geq 1 .\)
	
	Let $\rho=\text{Cone }(v_j)$. Now to compute \(E^{\rho}(0)\), we consider an \(m \in M\) satisfying \(\langle m, v_{\rho}\rangle=0\). Then from from \eqref{filtcotan1}, we get \(E^{\rho}(0)=\mathbf{L}^{\sigma}_1 \bigoplus \cdots \bigoplus \widehat{\mathbf{L}}^{\sigma}_j \bigoplus \cdots \bigoplus \mathbf{L}^{\sigma}_n \). Hence we can identify the space \(E^{\rho}(0)\) with \(\text{Span}(v_{\rho})^{\perp} \) (see \eqref{percotan1}). Thus we get the desired filtrations, which we denote by \( \left( \Omega, \{\Omega^{\rho}(i)\}_{\rho \in \Delta(1), i \in \Z} \right)  \). 
\end{proof}
 The following corollary is immediate from Proposition \ref{perdual} and Proposition \ref{percotan}.
\begin{cor}\label{pertangb}
	Let \(X=X(\Delta)\) be a nonsingular complete toric variety of dimension \(n\). Then the filtrations \( \left( \mathscr{T}, \{\mathscr{T}^{\rho}(i)\}_{\rho \in \Delta(1), i \in \Z} \right)  \) associated to the tangent bundle \(\mathcal{T}_X\) are given by 
	
	\[ \mathscr{T}^{\rho}(i) = \left\{ \begin{array}{cc}
	0 & i \leq -2 \\ 
	\text{Span}(v_{\rho}) & i=-1 \\
	N \otimes_{\Z} k & i \geq 0.
	\end{array} \right. \] 
\end{cor}

\subsection{Stability}\label{sec:stability}

Let \(X\) be a nonsingular projective variety of dimension \(n\). Fix a polarization \(H\), i.e. an ample divisor on \(X\). For a torsion free sheaf $\mathcal{E}$ over \(X\), we have \(\text{deg } \mathcal{E}=c_{1}(\mathcal{E}) \cdot H^{n-1}\) and slope $\mu(\mathcal{E})=\frac{\text{deg } \mathcal{E}}{\text{rank}(\mathcal{E})}$. 
\begin{rmk}\label{degpositive}
	For a \(T\)-invariant divisor \(D_{\rho}\) on a nonsingular projective toric variety \(X=X(\Delta)\), we have \(\text{deg } D_{\rho} > 0\) for all \(\rho \in \Delta(1)\) by Nakai-Moishezon criterion \cite[Theorem A.5.1]{hartshorne2013algebraic}. 
\end{rmk}
We consider stability in the sense of Mumford-Takemoto, which is also known as $\mu$-stability. A subsheaf $\mathcal{F}$ of $\mathcal{E}$ is called proper if \(0 < \text{rank}(\mathcal{F})< \text{rank}(\mathcal{E}) \). A torsion free sheaf $\mathcal{E}$ over \(X\) is said to be (semi)stable with respect to \(H\) if for any proper subsheaf $\mathcal{F}$ of $\mathcal{E}$, we have $\mu(\mathcal{F}) (\leq)< \mu(\mathcal{E})$. We say $\mathcal{E}$ is unstable if it is not semistable.

\begin{rmk}\label{ref2}
Let \(X\) be a nonsingular complex toric variety. To check $($semi$)$stability of a reflexive sheaf $\mathcal{E}$, it suffices to consider only proper saturated subsheaves of $\mathcal{E}$ $($see \cite[Proposition 1.2.6]{huybrechts2010geometry}$)$. Since saturated subsheaf of a reflexive sheaf is again reflexive $($see \cite[Lemma 1.1.16]{okonek}$)$, it is enough to consider only reflexive subsheaves of $\mathcal{E}$ for checking its $($semi$)$stability. Furthermore, if $\mathcal{E}$ is equivariant, by \cite[Theorem 2.1]{biswas2018stability}, it is enough to consider only equivariant reflexive subsheaves. 
\end{rmk}

\section{Characterization of equivariant subsheaves of an equivariant sheaf}

We characterize all equivariant subsheaves of a torsion free equivariant sheaf. From now onwards, we take the underlying field \(k=  \C \).

\begin{prop}\label{STF}
	Let \(\mathcal{E}\) be a torsion free equivariant sheaf on \(X\) corresponding to a family of multifiltrations \(\{E^{\sigma}_m\}_{\sigma \in \Delta, m \in M}\) of the vector space \(\mathbf{E}^0\). There is a one-to-one correspondence between equivariant subsheaves of \(\mathcal{E}\) and family of submultifiltrations  \(\{F^{\sigma}_m\}_{\sigma \in \Delta, m \in M}\) of the vector space \(\mathbf{F}^0\), where \(\mathbf{F}^0\) is a subspace of \(\mathbf{E}^0\) and \(F^{\sigma}_m=E^{\sigma}_m \cap \mathbf{F}^0\).
\end{prop}

\begin{proof}
	Let \(\mathcal{F}\) be an equivariant subsheaf of \(\mathcal{E}\) and hence \(F^{\sigma}:=\Gamma(U_{\sigma}, \mathcal{F})\) is a \(T\)-stable subspace of \(E^{\sigma}:=\Gamma(U_{\sigma}, \mathcal{E})\) for any \(\sigma \in \Delta\). The isotypical decomposition of both the spaces are given by \(F^{\sigma}=\bigoplus\limits_{m \in M} F^{\sigma}_m \) and \(E^{\sigma}=\bigoplus\limits_{m \in M} E^{\sigma}_m \), where \(F^{\sigma}_m = E^{\sigma}_m \cap F^{\sigma} \) since \(T\) is reductive (see \cite[Theorem 1.23]{brion}). Thus we obtain $\sigma$-families \(\widehat{F}^{\sigma}\) and \(\widehat{E}^{\sigma}\) associated to $\mathcal{F}$ and $\mathcal{E}$ respectively together with an inclusion of $\sigma$-families \(\widehat{F}^{\sigma} \longhookrightarrow \widehat{E}^{\sigma} .\)

The $\Delta$-family associated to $\mathcal{E}$ (respectively, $\mathcal{F}$) encodes the data for gluing the sheaves $\mathcal{E}_{\sigma}:=\mathcal{E}|_{U_{\sigma}}$ (respectively, $\mathcal{F}_{\sigma}:=\mathcal{F}|_{U_{\sigma}}$) on the affine open sets \(U_{\sigma}\). Since the gluing data of $\mathcal{F}$ is the restriction of the gluing data of $\mathcal{E}$, we get the following commuting diagram as $\tau$-families, where \(\tau \preceq \sigma \):
\begin{equation}\label{STFdiag3}
\begin{tikzpicture}[description/.style={fill=white,inner sep=2pt}]
\matrix (m) [matrix of math nodes, row sep=3em,
column sep=2.5em, text height=1.5ex, text depth=0.25ex]
{ i^{*}_{\tau \sigma} \widehat{F}^{\sigma}    & &   \widehat{F}^{\tau}   \\
	i^{*}_{\tau \sigma} \widehat{E}^{\sigma}    & &   \widehat{E}^{\tau}   \\ };
\path[->]  (m-1-1) edge node[auto] {}(m-1-3);
\path[ ->] (m-2-1) edge node[below] {}(m-2-3);
\path[right hook->] (m-1-1) edge node[auto] {}(m-2-1);
\path[->] (m-1-1) edge node[above] {$\eta'_{\tau \sigma}$}(m-1-3);
\path[->] (m-1-1) edge node[below] {$\cong$}(m-1-3);
\path[right hook ->] (m-1-3) edge node[auto] {} (m-2-3);
\path[->] (m-2-1) edge node[above] {$\eta_{\tau \sigma}$} (m-2-3);
\path[->] (m-2-1) edge node[below] {$\cong$} (m-2-3);
\end{tikzpicture}	
\end{equation}
where $\eta_{\tau \sigma }$ is as in the Definition \ref{delta fam} of $\Delta$-family, similarly $\eta'_{\tau \sigma }$ denotes the corresponding isomorphism. Furthermore, the commutative diagram of \(k[S_{\sigma}]\)-modules
\begin{center}
\begin{tikzpicture}[description/.style={fill=white,inner sep=2pt}]
\matrix (m) [matrix of math nodes, row sep=3em,
column sep=2.5em, text height=1.5ex, text depth=0.25ex]
{ \Gamma(U_{\sigma}, \mathcal{F})   & &  \Gamma(U_{\tau}, i^{*}_{\tau \sigma}\mathcal{F})  \\
\Gamma(U_{\sigma}, \mathcal{E})    & &  \Gamma(U_{\tau}, i^{*}_{\tau \sigma}\mathcal{E})   \\ };
\path[right hook ->]  (m-1-1) edge node[auto] {}(m-1-3);
\path[right hook ->] (m-2-1) edge node[below] {}(m-2-3);
\path[right hook->] (m-1-1) edge node[auto] {}(m-2-1);
\path[right hook ->] (m-1-3) edge node[auto] {} (m-2-3);
\end{tikzpicture}	
\end{center}

induces the following commutative diagram of $\sigma$-families (cf. \eqref{TF1})
\begin{equation}\label{STFdiag3.1}
\begin{tikzpicture}[description/.style={fill=white,inner sep=2pt}]
\matrix (m) [matrix of math nodes, row sep=3em,
column sep=2.5em, text height=1.5ex, text depth=0.25ex]
{ \widehat{F}^{\sigma}    & &  i^{*}_{\tau \sigma} \widehat{F}^{\sigma}   \\
	\widehat{E}^{\sigma}    & &  i^{*}_{\tau \sigma} \widehat{E}^{\sigma}   \\ };
\path[right hook ->]  (m-1-1) edge node[auto] {$\alpha'_{\tau \sigma}$}(m-1-3);
\path[right hook ->] (m-2-1) edge node[below] {$\alpha_{\tau \sigma}$}(m-2-3);
\path[right hook->] (m-1-1) edge node[auto] {}(m-2-1);
\path[right hook ->] (m-1-3) edge node[auto] {} (m-2-3);
\end{tikzpicture}	
\end{equation}

Combining \eqref{TF2}, the diagrams \eqref{STFdiag3}, \eqref{STFdiag3.1}  and \cite[Proposition 5.15, Corollary 5.16]{perling}, we get the following commutative diagram:
\begin{equation}\label{STFdiag1}
\begin{tikzpicture}[description/.style={fill=white,inner sep=2pt}]
\matrix (m) [matrix of math nodes, row sep=3em,
column sep=2.5em, text height=1.5ex, text depth=0.25ex]
{ \mathbf{F}^{\sigma}    & &  \mathbf{F}^{0}  \\
	\mathbf{E}^{\sigma}    & &  \mathbf{E}^{0}  \\ };
\path[->]  (m-1-1) edge node[auto] {}(m-1-3);
\path[->] (m-2-1) edge node[below] {}(m-2-3);
\path[right hook->] (m-1-1) edge node[auto] {}(m-2-1);
\path[->] (m-1-1) edge node[below] {$\cong$}(m-1-3);
\path[right hook ->] (m-1-3) edge node[auto] {} (m-2-3);
\path[->] (m-2-1) edge node[below] {$\cong$} (m-2-3);
\end{tikzpicture}	
\end{equation}

Hence we can realize all the spaces \(E^{\sigma}_m\) as subspace of the space \(\mathbf{E}^{0}\) and the spaces \(F^{\sigma}_m\) as subspace of the space \(\mathbf{F}^{0}\) such that the collection of subpaces \(\{E^{\sigma}_m\}_{m \in M}\) (respectively, \(\{F^{\sigma}_m\}_{m \in M}\)) of \(\mathbf{E}^0\) (respectively, \(\mathbf{F}^0\)) forms a family of multifiltrations. We have  \(F^{\sigma}_m \subseteq E^{\sigma}_m \cap \mathbf{F}^{0} \) for all \(\sigma \in \Delta \) and \(m \in M\). For the reverse inclusion, note that we have the following commutative diagram:
\begin{equation}\label{STFdiag2}
\begin{tikzpicture}[description/.style={fill=white,inner sep=2pt}]
\matrix (m) [matrix of math nodes, row sep=3em,
column sep=2.5em, text height=1.5ex, text depth=0.25ex]
{ F^{\sigma}_m    & &   \mathbf{F}^{0}   \\
	E^{\sigma}_m    & &   \mathbf{E}^{0}   \\ };
\path[right hook ->]  (m-1-1) edge node[auto] {}(m-1-3);
\path[right hook ->] (m-2-1) edge node[below] {}(m-2-3);
\path[right hook->] (m-1-1) edge node[auto] {}(m-2-1);
\path[right hook ->] (m-1-3) edge node[auto] {} (m-2-3);
\end{tikzpicture}	
\end{equation}

By the diagrams \eqref{STFdiag1} and \eqref{STFdiag2}, for  \(e \in E^{\sigma}_m \cap \mathbf{F}^{0}=E^{\sigma}_m \cap \mathbf{F}^{\sigma}\), we have \([e,m]=[e',m']\in \mathbf{F}^{\sigma} \subset \mathbf{E}^{\sigma} \), where \(e' \in F^{\sigma}_{m'}\). Then there exists \(m''\in M\) such that \(m, m' \leq_{\sigma }m''\) and \(\chi^{\sigma}_{m,m''}(e)=\chi^{\sigma}_{m',m''}(e') \in F^{\sigma}_{m''}\). Since \(F^{\sigma}\) is a \(T\)-stable submodule of \(E^{\sigma}\), it has a \(T\)-stable complement, say \(W^{\sigma}\) in \(E^{\sigma}\). Thus \(E^{\sigma}_m= F^{\sigma}_m \oplus W^{\sigma}_m\) for all \(m \in M \). Let us write \(e=e_1+e_2\), where \(e_1 \in F^{\sigma}_m \) and \(e_2 \in W^{\sigma}_m\). Then we have that \(\chi^{\sigma}_{m,m''}(e)=\chi^{\sigma}_{m,m''}(e_1)+ \chi^{\sigma}_{m,m''}(e_2)\) where \(\chi^{\sigma}_{m,m''}(e_1) \in F^{\sigma}_{m''} \) and \(\chi^{\sigma}_{m,m''}(e_2) \in W^{\sigma}_{m''}\). It follows that \(\chi^{\sigma}_{m,m''}(e_2)=0\) and since \(\chi^{\sigma}_{m,m''}\) is injective ($\mathcal{E}$ being torsion free), we have \(e_2=0\), i.e. \(e \in F^{\sigma}_m \). This concludes the proof of the forward direction of the proposition.

	Conversely, given a subspace \(\mathbf{F}^0\) of \(\mathbf{E}^0\), let us define \(F^{\sigma}_m:=E^{\sigma}_m \cap \mathbf{F}^0\) for \(\sigma \in \Delta, m \in M\). Then by Definition \ref{multfilt} and Remark \ref{rmultfilt}, \(\{F^{\sigma}_m\}_{\sigma \in \Delta, m \in M}\) forms a family of multifiltrations \(\{F^{\sigma}_m\}_{\sigma \in \Delta, m \in M}\) of the vector space \(\mathbf{F}^0\), and hence corresponds to a torsion free equivariant sheaf \(\mathcal{F}\) (see \cite[Theorem 5.18]{perling}). It remains to show that \(\mathcal{F}\) is an equivariant subsheaf of \(\mathcal{E}\) which follows from \(\Gamma(U_{\sigma}, \mathcal{F})=\bigoplus\limits_{m \in M} F^{\sigma}_m \subseteq \bigoplus\limits_{m \in M} E^{\sigma}_m = \Gamma(U_{\sigma}, \mathcal{E}).\)
	
\end{proof}

Recall that given a filtration $(V, \{F^pV\})$ on a vector space $V$ and a subspace $W \subseteq V$, there is an induced subfiltration on $W$ by setting $F^p(W):= W \cap F^p(V)$. As an immediate corollary of Proposition \ref{STF} we can characterize reflexive subsheaves in terms of induced subfiltrations.
\begin{cor} \label{RS1}
	Let $ \mathcal{E}$ be an equivariant reflexive sheaf on $X$ with associated filtrations \\*
	$\left( \mathbf{E}^0, \{E^{\rho}(i) \}_{\rho \in \Delta(1)} \right)$. Then  equivariant reflexive subsheaves of $\mathcal{E}$ are in one-to-one correspondence with the induced subfiltrations $\left( \mathbf{F}^0, \{F^{\rho}(i) \}_{\rho \in \Delta(1)} \right)$ of $\left( \mathbf{E}^0, \{E^{\rho}(i) \}_{\rho \in \Delta(1)} \right)$, where $\mathbf{F}^0$ is a subspace of $\mathbf{E}^0$.
\end{cor}

The following proposition provides a combinatorial criterion of (semi)stability of equivariant sheaf (cf. \cite[Equation (19)]{knutson}, \cite[Page 1730]{kool}).

\begin{prop}\label{RS2}
	Let $\mathcal{E}$ be an equivariant torsion free sheaf on a nonsingular projective toric variety. Let \(\{E^{\sigma}_m\}_{\sigma \in \Delta, m \in M}\) be the family of multifiltrations of the vector space \(\mathbf{E}^0\) corresponding to it. Then $\mathcal{E}$ is (semi)stable  if and only if
	
	\[\frac{1}{\text{dim }F} \left( -\sum_{i\in \Z, \rho \in \Delta(1)} i \text{ dim}F^{[\rho]}(i) \text{deg}D_\rho \right) 	(\leq) < \frac{1}{\text{dim }E} \left( -\sum_{i\in \Z, \rho \in \Delta(1)} i \text{ dim}E^{[\rho]}(i) \text{deg}D_\rho\right)\]
	\noindent
	for every proper subspace $F$ of $E$, where $F^{\rho}(i)=F \cap E^{\rho}(i)$ for any ray $\rho$.
\end{prop}

\begin{proof}
Note that by Proposition \ref{chern}, \(\mu(\mathcal{E})= \frac{1}{\text{dim }E} \left( -\sum_{i\in \Z, \rho \in \Delta(1)} i \text{ dim}E^{[\rho]}(i) \text{deg}D_\rho\right).\) Since subsheaf of a torsion free sheaf is again torsion free, using Proposition \ref{STF} and Remark \ref{ref2} the proposition follows.
\end{proof}

The following remark will be helpful for determining which subsheaves of the tangent bundle have maximum possible slope.

\begin{rmk}\label{sub_of_tangent}
	Let \(X=X(\Delta)\) be a nonsingular projective toric variety of dimension \(n\) and  \(\mathcal{F}\) be a proper equivariant reflexive subsheaf of $\mathcal{T}_X$. Let \(\left(F, \{F^{\rho}(i)\}_{\rho \in \Delta(1) \ i \in \Z} \right) \) be the filtrations associated to $\mathcal{F}$, where \(F\) is a vector subspace of \(N \otimes_{\Z} {k}  \cong k^{n}\) of dimension \(l \ (\leq n-1)\) and \(F^{\rho}(i)=F \cap \mathscr{T}^{\rho}(i)\) $($see Proposition \ref{RS1}$)$.  By Proposition \ref{chern}, we have 
	\begin{equation}
	c_1(\mathcal{F})= \left\{ \begin{array}{cc}
	\sum\limits_{v_{\rho} \in F \cap \Delta(1)}D_\rho & \text{ if } F \cap \Delta(1) \neq \emptyset \\ 
	0 & \text{ if } F \cap \Delta(1) = \emptyset.
	\end{array} \right. 
	\end{equation}
	
	Since we are interested in subsheaves of $\mathcal{T}_X$ with maximum possible slope and degree of \(D_{\rho}\) are positive (see Remark \ref{degpositive}), it is enough to consider	proper equivariant reflexive subsheaves with associated filtrations \(\left(F, \{F^{\rho}(i)\}_{\rho \in \Delta(1) \ i \in \Z} \right) \) where \(F= \text{Span}(F \cap \Delta(1)) \).
	
\end{rmk}

\section{Stability of tangent bundle on nonsingular projective toric variety with Picard number $\leq$ 2}

\subsection{Stability of tangent bundle on nonsingular projective toric variety with Picard number 1}
Note that the only nonsingular projective toric variety with Picard group \(\Z\) is the projective space (see \cite[Exercise 7.3.10]{Cox}). It is well known that tangent bundle on projective space is stable (see \cite[Theorem 1.3.2]{okonek}, \cite[Theorem 7.1]{biswas2018stability}). We give a simple proof of this fact using Proposition \ref{RS2}. 

\begin{prop}\label{stabtanP}
	The tangent bundle \(\mathcal{T}_{\mathbb{P}^n}\) is stable for all \(n >0\).
\end{prop}	

\begin{proof}
	Let us fix  some ample divisor \(H\) on \(\mathbb{P}^n\). Let $\Delta$ denote the fan of \(\mathbb{P}^n\) in the lattice \(N=\Z^n\). Let \(e_1, \ldots, e_n\) denote the standard basis of \(\Z^n\) and set \(e_0=-e_1-\cdots-e_n\). Then the fan consists of \(n+1\) rays \(e_0, e_1 \ldots, e_n\) and \(n+1\) maximal cones \(\text{Cone}(e_0, \ldots, \widehat{e}_i, \ldots, e_n)\), where \(i=0, \ldots, n\). We can assume \(n \geq 2\) since the statement is trivial for \(n=1\). The divisors \(D_0,\ldots, D_n \) corresponding to the rays \(e_0, e_1 \ldots, e_n\) are all linearly equivalent and hence we have \(\text{deg }D_0=\ldots=\text{deg }D_n\). Note that \(\mu(\mathcal{T}_{\mathbb{P}^n})=(1+\frac{1}{n})\text{deg }D_0\).
	
By Remark \ref{sub_of_tangent},	let \(\mathcal{F}\) be a proper equivariant reflexive subsheaf of \(\mathcal{T}_{\mathbb{P}^n}\) of rank \( l<n\) with associated filtrations \(\left(F, \{F^{\rho}(i)\}_{\rho \in \Delta(1) \ i \in \Z} \right) \) where \(F= \text{Span}(F \cap \Delta(1)) \). Then we see that \(\mu(\mathcal{F})=\frac{p}{l}\text{deg }D_0 \leq \text{deg }D_0<\mu(\mathcal{T}_{\mathbb{P}^n})\), where \(|F \cap \Delta(1)|=p \leq l \). Hence by Proposition \ref{RS2}, \(\mathcal{T}_{\mathbb{P}^n}\) is stable with respect to \(H\).
\end{proof}	 

\subsection{Stability of tangent bundle on nonsingular projective toric variety with Picard number 2}

 Now we turn to nonsingular projective toric varieties with Picard group \(\Z^2\) which were classified by Kleinschimidt (see \cite[Theorem 7.3.7]{Cox}). He showed that if \(X\) is any nonsingular projective toric variety with \(\text{Pic}(X)\cong \Z^2 \), then there are integers \(s, r \geq 1\), \(s+r = \text{dim}(X)\) and \(0 \leq a_1 \leq \ldots \leq a_r\) such that \(X \cong \mathbb{P}(\mathcal{O}_{\mathbb{P}^s}  \oplus \mathcal{O}_{\mathbb{P}^s}(a_1) \oplus \cdots \oplus \mathcal{O}_{\mathbb{P}^s}(a_r)  ).\) We recall the fan structure of \(X\) from \cite[Example 7.3.5]{Cox}. Let $\Delta$ be the fan of \(X\) in the lattice \(N=\Z^s \times \Z^r\). Let \(\{u_1, \ldots, u_s\}\) and \(\{e_1', \ldots, e_r'\}\) be standard basis of \(\Z^s\) and \(\Z^r\) respectively. Set 
\begin{equation*}
\begin{split}
& v_i=(u_i, {\bf{0}}) \in N \text{ for } 1 \leq i \leq s \text{ ; } e_i=({\bf 0}, e_i') \in N  \text{ for } i=1, \ldots, r,\\   &v_0=-v_1-\cdots-v_s+a_1e_1+ \cdots+a_r e_r \text{ and }  e_0=-e_1- \cdots-e_r.
\end{split}
\end{equation*}
The rays of $\Delta$ are given by \( v_0, v_1 \ldots, v_s, e_0, e_1, \ldots, e_r \) and the maximal cones are given by
\[\text{Cone}(v_0, \ldots, \widehat{v}_j, \ldots, v_s ) + \text{Cone}(e_0, \ldots, \widehat{e}_i, \ldots, e_r ), 
\text{ for all } j=0, \ldots, s \text{ and } i=0, \ldots, r.\]

There is the following relations among the \(T\)-invariant prime divisors:
\begin{equation}
	\begin{split}
	\text{div}(\chi^{v_1^*}) & =D_{v_1}-D_{v_0}, \ldots, \text{div}(\chi^{v_s^*})=D_{v_s}-D_{v_0} \\
	\text{div}(\chi^{e_i^*}) &=D_{e_i}+a_i D_{v_0}-D_{e_0} \text{ for } i=1, \ldots, r.
	\end{split}
\end{equation}
Hence we have 
\begin{equation}\label{relnPic2}
D_{v_0} \sim_{\text{lin}} D_{v_i},  i=1, \ldots, s \text{ and }  D_{e_i} \sim_{\text{lin}} D_{e_0}-a_i D_{v_0}, i=1, \ldots, r. 
\end{equation}
By \eqref{relnPic2}, it follows that \(D_{v_0}\) and \(D_{e_0}\)  generate \(\text{Pic}(X)\). Now we show that \(D_{v_0}\) and \(D_{e_0}\) are not linearly equivalent. Consider the wall \[\tau=\text{Cone}(v_0, \ldots, \widehat{v}_i, \ldots, v_s, e_0, \ldots, \widehat{e}_j, \ldots, \widehat{e}_k, \ldots, e_r), \text{ where } 0 \leq i \leq s, \ 0 \leq j <k \leq r.\]
We can write \(\tau=\text{Cone}(\tau, e_j) \cap \text{Cone}(\tau, e_k) \) and hence the wall relation is given by \(e_0+e_1 \cdots+ e_r=0.\) Thus \(D_{v_0}\cdot V(\tau)=0\) and \(D_{e_0}\cdot V(\tau)=1\) (see Proposition \ref{wallreln}) which implies that \(D_{v_0}\) and \(D_{e_0}\) are not numerically equivalent and hence not linearly equivalent. This also shows that \(D_{v_0}\) and \(D_{e_0}\) are \( \Z\)-linearly independent and hence we have, $\text{Pic}(X)=\Z D_{v_0} \oplus  \Z D_{e_0}.  $
 In particular, the anticanonical divisor is given by 
 \begin{equation}\label{anti_can}
 -K_X=(s+1-a_1-\cdots-a_r) D_{v_0}+(r+1) D_{e_0}.
 \end{equation}

\begin{prop}
	Let \(D=a D_{v_0} + b D_{e_0}\), \(a, b \in \Z\) be a \(T\)-invariant divisor on $X = \mathbb{P}(\mathcal{O}_{\mathbb{P}^s}  \oplus \mathcal{O}_{\mathbb{P}^s}(a_1) \oplus \cdots \oplus \mathcal{O}_{\mathbb{P}^s}(a_r)  )$. Then \(D\) is ample $($respectively, nef$)$ if and only if \(a, b >0 \ ( \text{respectively,  }\\*
	 a, b \geq 0)\). In particular, \(X\) is Fano if and only if \(a_1+\cdots +a_r < s+1\).
\end{prop}

\begin{proof}
  Using Toric Nakai criterion (see \cite[Theorem 2.18]{Oda}), we have \(D\) is ample if and only if \(D\cdot V(\tau) >0 \) for all wall $\tau$. Thus we need to compute \(D\cdot V(\tau)\) for all walls $\tau \in \Delta(s+r-1)$. Note that the walls are of the following three types:
	
	\(\tau_{\{i,j\},0}=\text{Cone}(v_0, \ldots, \widehat{v}_i, \ldots, \widehat{v}_j, \ldots, v_s, \widehat{e}_0, e_1, \ldots, e_r)\), \(0 \leq i <j \leq s\), 
	
	\(\tau_{\{i,j\},k}=\text{Cone}(v_0, \ldots, \widehat{v}_i, \ldots, \widehat{v}_j, \ldots, v_s, e_0,\ldots, \widehat{e}_k, \ldots, e_r)\), \(0 \leq i <j \leq s\), \(0 < k \leq r\) and 
	
	\(\tau_{i,\{j,k\}}=\text{Cone}(v_0, \ldots, \widehat{v}_i, \ldots, v_s, e_0, \ldots, \widehat{e}_j, \ldots, \widehat{e}_k, \ldots, e_r)\), \(0 \leq i \leq s \), \(0 \leq j <k \leq r\).
	
	Note that the wall relation corresponding to the wall \(\tau_{\{i,j\},0}=\text{Cone}(\tau_{\{i,j\},0}, v_i) \cap \text{Cone}(\tau_{\{i,j\},0}, v_j) \) is given by
	\[v_0 + \cdots + v_s-a_1e_1 -\cdots-a_r e_r=0,\]
	which implies \(D_{v_0}\cdot V(\tau_{\{i,j\},0})=1\) and  \(D_{e_0}\cdot V(\tau_{\{i,j\},0})=0\) (see Proposition \ref{wallreln}) . This gives
	\begin{equation}\label{wall1}
	 D\cdot V(\tau_{\{i,j\},0})=a.
	\end{equation}
	
	Similarly, the wall \(\tau_{\{i,j\},k}=\text{Cone}(\tau_{\{i,j\},k}, v_i) \cap \text{Cone}(\tau_{\{i,j\},k}, v_j) \)  gives the relation
	\[v_0 + \cdots + v_s +a_k e_0+b_1e_1 +\cdots+ \widehat{e}_k+ \cdots+b_r e_r=0\] for some integers \(b_1, \ldots, b_{k-1}, b_{k+1}, \ldots, b_r\). Thus  \(D_{v_0}\cdot V(\tau_{\{i,j\},k})=1\) and  \(D_{e_0}\cdot V(\tau_{\{i,j\},k})=a_k\). Hence we have
	\begin{equation}\label{wall2}
	 D\cdot V(\tau_{\{i,j\},k})=a+a_k b.
	\end{equation}
	
	Finally the wall relation for \(\tau_{i,\{j,k\}}=\text{Cone}(\tau_{i,\{j,k\}}, e_j) \cap \text{Cone}(\tau_{i,\{j,k\}}, e_k) \) is as follows 
	\[e_0+e_1 \cdots+ e_r=0.\]
	 So we get \(D_{v_0}\cdot V(\tau_{i,\{j,k\}})=0\) and \(D_{e_0}\cdot V(\tau_{i,\{j,k\}})=1\).  Hence we have
	\begin{equation}\label{wall3}
 D\cdot V(\tau_{i,\{j,k\}})=b.
	\end{equation}
	Now considering the equations \eqref{wall1}, \eqref{wall2} and \eqref{wall3} it follows that \(D\) is ample (respectively, nef) if and only if \(a, b >0 \ ( \text{respectively,  } a, b \geq 0)\). 
	
	The second part of the proposition follows from \eqref{anti_can}.
\end{proof}

We fix a polarization \(H=a D_{v_0} + b D_{e_0}\), \(a, b >0\). Note that from \eqref{relnPic2}, we have \begin{equation}\label{degrel1}
\text{deg }D_{v_0}=\text{deg }D_{v_i} \text{ for } i=1, \ldots,s.
\end{equation}
\noindent
and \(\text{deg }D_{e_0}-\text{deg }D_{e_i}=a_i \text{deg }D_{v_0} \geq 0 \) for \(i=1, \ldots, r\) by Remark \ref{degpositive}. So we have 
\begin{equation}\label{degrel2}
\text{deg }D_{e_0} \geq \text{deg }D_{e_i} \text{ for } i=1, \ldots, r.
\end{equation}

Furthermore, we have \(\text{deg }D_{e_0}=\text{deg }D_{e_r}+a_r \text{deg }D_{v_0} >  a_r \text{deg }D_{v_0} \). Then if \(a_r\) is positive, we get
\begin{equation}\label{degrel3}
\text{deg }D_{e_0} > \text{deg }D_{v_0}.
\end{equation}

From \eqref{anti_can}, we have
\begin{equation}\label{slopeTpic2}
\mu(\mathcal{T}_X)=\left( \frac{s+1-a_1-\cdots-a_r}{s+r} \right) \text{deg }D_{v_0}+ \left( \frac{r+1}{s+r} \right) \text{deg }D_{e_0}
\end{equation}

Denote by $\alpha=\displaystyle \frac{s+1-a_1-\cdots-a_r}{s+r} $ and $\beta=\displaystyle \frac{r+1}{s+r}$, then \( \alpha < 1\) and \(0 < \beta \leq 1 \).

\begin{thm}\label{M1}
Let \(X = \mathbb{P}(\mathcal{O}_{\mathbb{P}^s}  \oplus \mathcal{O}_{\mathbb{P}^s}(a_1) \oplus \cdots \oplus \mathcal{O}_{\mathbb{P}^s}(a_r)  )\), where \(s, r \geq 1\), \(0 \leq a_1 \leq \ldots \leq a_r\) and \(a_r > 0\). Then the tangent bundle $\mathcal{T}_X$ is unstable with respect to any polarization whenever \((a_1, \ldots, a_r) \neq (0, 0, \ldots, 0, 1) \).
\end{thm}

\begin{proof}
Note that \(\mu(\mathcal{T}_X) < (\alpha + \beta )\text{deg }D_{e_0} \) from \eqref{degrel3} and \eqref{slopeTpic2}. Observe that  \(\alpha + \beta > 1\) if and only if \(a_1 + \cdots+ a_r \leq 1 \), i.e. \((a_1, \ldots, a_r) = (0, 0, \ldots, 0, 1) \). When \(\alpha + \beta \leq 1\), i.e. \((a_1, \ldots, a_r) \neq (0, 0, \ldots, 0, 1) \), then we see that $\mu(\mathcal{T}_X) < \text{deg }D_{e_0} $. From Proposition \ref{RS1}, it follows that for \(r=1\) (respectively, \(r \geq 2\)), \(\mathcal{O}_X(D_{e_0} +D_{e_1}) \)  (respectively, \(\mathcal{O}_X(D_{e_0}) \)) is a rank 1 reflexive subsheaf of $\mathcal{T}_X$ corresponding to the vector subspace \(\text{Span}(e_0)\) of \(N_{\C} \). Hence $\mathcal{T}_X$ is unstable.
\end{proof}

Next let us consider \(X = \mathbb{P}(\mathcal{O}_{\mathbb{P}^s}^r  \oplus \mathcal{O}_{\mathbb{P}^s}(1)  )\), where \(s, r \geq 1\). Then the relations in \eqref{relnPic2} simplifies to the following form 
\begin{equation}\label{relnPic2case2}
D_{v_0} \sim_{\text{lin}} D_{v_i},  i=1, \ldots, s; \ D_{e_0} \sim_{\text{lin}} D_{e_i}, i=1, \ldots, r-1 \text{ and } D_{e_r} \sim_{\text{lin}} D_{e_0}  - D_{v_0}. 
\end{equation}


\begin{lemma}\label{deg}
\(X = \mathbb{P}(\mathcal{O}_{\mathbb{P}^s}^r  \oplus \mathcal{O}_{\mathbb{P}^s}(1)  )\), where \(s, r \geq 1\). Then,

\begin{enumerate}
	\item \(\text{deg }D_{v_0}=\sum\limits_{i=r}^{r-1+s} \binom{r-1+s}{i} a^{r-1+s-i} b^{i}\)
	\item \(\text{deg }D_{e_0} =\sum\limits_{i=r-1}^{r-1+s} \binom{r-1+s}{i} a^{r-1+s-i} b^{i}\).
\end{enumerate} 
\end{lemma}

\begin{proof}
We first compute \(H^{s+r-1}\) as follows. \[H^{s+r-1}=(a D_{v_0} + b D_{e_0})^{s+r-1}
=\sum\limits_{i=0}^{r-1+s} \binom{r-1+s}{i}  a^{r-1+s-i} b^i \ D_{v_0}^{r-1+s-i} D_{e_0}^i .\]
From \eqref{relnPic2case2}, we have
\begin{equation}\label{binom1}
\begin{split}
& D_{v_0}^{r-1+s-i} =0 \text{ for } i < r-1, \\
& D_{e_0}^{r+j} =D_{e_0} \cdots D_{e_{r-1}} (D_{e_r}+D_{v_0})^j=D_{e_0} \cdots D_{e_{r-1}} D_{v_0}^j \text{ for } j>0.
\end{split}
\end{equation}

Put \(i=r+j\), \(j>0\), then the \(i\)-th term of the binomial expression \(H^{s+r-1}\) takes the form 
\begin{equation}\label{binom2}
D_{v_0}^{r-1+s-i} D_{e_0}^i=D_{v_0}^{s-1-j} D_{e_0}^{r+j}=D_{e_0} \cdots D_{e_{r-1}} D_{v_0}^{s-1}  
\end{equation}
 
 From \eqref{binom1} and \eqref{binom2} we see that
 
 \begin{equation}\label{H}
 \begin{split}
 H^{s+r-1}=& \binom{r-1+s}{r-1} a^s b^{r-1} \ D_{e_0}^{r-1} D_{v_0}^s + \binom{r-1+s}{r} a^{s-1} b^r  D_{e_0}^r D_{v_0}^{s-1}\\
 & + \sum\limits_{i=r+1}^{r-1+s} \left( \binom{r-1+s}{i} \ a^{r-1+s-i} b^i \right) D_{e_0} \cdots D_{e_{r-1}} D_{v_0}^{s-1}.
 \end{split}
 \end{equation}
 
 Now see that \begin{equation*}
 \begin{split}
 &D_{e_0}^{r-1} D_{v_0}^s \cdot D_{e_r}=1,\ D_{e_0}^r D_{v_0}^{s-1} \cdot D_{e_r}=0, \ D_{e_0} \cdots D_{e_{r-1}} D_{v_0}^{s-1} \cdots D_{e_r}=0.
 \end{split}
 \end{equation*} Hence we have 
 \begin{equation}\label{D_er}
\text{deg }D_{e_r} =\binom{r-1+s}{r-1} a^s b^{r-1}.
 \end{equation}

Similarly, we can see that 
\begin{equation}\label{D_v0}
\text{deg }D_{v_0}=\sum\limits_{i=r}^{r-1+s} \binom{r-1+s}{i}  a^{r-1+s-i} b^i
\end{equation}

Finally, from \eqref{relnPic2case2}, \eqref{D_er} and \eqref{D_v0} we get, 
\[\text{deg }D_{e_0} =\sum\limits_{i=r-1}^{r-1+s} \binom{r-1+s}{i} a^{r-1+s-i} b^{i} .\]
\end{proof}

The following lemma is very crucial in studying the stability of the tangent bundle of \(\mathbb{P}(\mathcal{O}_{\mathbb{P}^s}^r  \oplus \mathcal{O}_{\mathbb{P}^s}(1)  )\), where \(s, r \geq 1\).

\begin{lemma}\label{RSforr>1}
Let \(X = \mathbb{P}(\mathcal{O}_{\mathbb{P}^s}^r  \oplus \mathcal{O}_{\mathbb{P}^s}(1)  )\), where \(s, r \geq 1\). Then 
\[max\{\mu(\mathcal{F}) : \mathcal{F} \text{ is a proper subsheaf of } \mathcal{T}_X\}=\text{deg }D_{e_0}+\frac{1}{r}\left( \text{deg }D_{e_0}-  \text{deg }D_{v_0}\right).\]
\end{lemma}

\begin{proof}
Without loss of generality, we consider only proper equivariant reflexive subsheaves of \(\mathcal{T}_X\) (see Remark \ref{ref2}). 	
Let $\mathcal{F}$ be a proper equivariant reflexive subsheaf of \(\mathcal{T}_X\) with associated filtrations \(\left(F, \{F^{\rho}(i)\}_{\rho \in \Delta(1) \ i \in \Z} \right) \). In view of Remark \ref{sub_of_tangent}, to find the maximum of \(\{\mu(\mathcal{F}) : \mathcal{F} \text{ is a proper subsheaf of } \mathcal{T}_X\}\) it is enough to consider the following cases:

\begin{enumerate}[(i)]
	\item \(\text{rank}(\mathcal{F})=r\) and \(F \cap \Delta(1) =\{e_0, \ldots, e_r\}\). In this case \(c_1(\mathcal{F})=(r+1)D_{e_0}-D_{v_0}\) and hence $\mu(\mathcal{F})=\text{deg }D_{e_0}+\frac{1}{r} (\text{deg }D_{e_0}-\text{deg }D_{v_0})$.
	
	\item \(\text{rank}(\mathcal{F})=j+1\) and \(F \cap \Delta(1) =\{e_0, \ldots, e_j\}\), where  \(0 \leq j \leq r-2\). In this case \(c_1(\mathcal{F})=(j+1)D_{e_0}\) and hence $\mu(\mathcal{F})=\text{deg }D_{e_0} $.
	
	\item \(\text{rank}(\mathcal{F})=j+k+2\) and \(F \cap \Delta(1)=\{v_0, \ldots, v_j, e_0, \ldots, e_k\}\), where \(0 \leq j< s\) and \(-1 \leq k \leq r-2 \) (here \(k=-1\) should be interpreted as no \(e_k\) term belongs to \(F \cap \Delta(1)\)). In this case  \(c_1(\mathcal{F})=(j+1)D_{v_0}+(k+1)D_{e_0}\) and hence $\mu(\mathcal{F})=\frac{1}{j+k+2}((j+1)\text{deg }D_{v_0} +(k+1)\text{deg }D_{e_0} ) < \text{deg }D_{e_0} $.
	
	\item \(\text{rank}(\mathcal{F})=s+j+2\) and \(F \cap \Delta(1)=\{v_0, \ldots, v_s,e_r, e_0, \ldots, e_j\}\), where \(-1 \leq j \leq r-3 \) (here \(j=-1\) should be interpreted as no \(e_j\) term belongs to \(F \cap \Delta(1)\)). In this case   \(c_1(\mathcal{F})=(j+2)D_{e_0}+s D_{v_0}\) and hence $\mu(\mathcal{F})=\frac{1}{s+j+2}((j+2)\text{deg }D_{e_0} +s \ \text{deg }D_{v_0} ) < \text{deg }D_{e_0} $.
	
	\item \(\text{rank}(\mathcal{F})=r+j+1\) and \(F \cap \Delta(1)=\{v_0, \ldots, v_j, e_0, \ldots, e_r\}\), where \(0 \leq j \leq s-2 \). In this case \(c_1(\mathcal{F})=(r+1)D_{e_0}+j D_{v_0}\) and hence $\mu(\mathcal{F})=\frac{1}{r+j+1}((r+1)\text{deg }D_{e_0}+j~\text{deg }D_{v_0})< \text{deg }D_{e_0}$.
\end{enumerate}

Thus the desired maximum is achieved from the case \((i)\).
\end{proof}

\begin{thm}\label{M3}
Let \(X = \mathbb{P}(\mathcal{O}_{\mathbb{P}^s}^r  \oplus \mathcal{O}_{\mathbb{P}^s}(1)  )\), where \(s \geq 1, r \geq 1\). Consider the polarization \(H=a D_{v_0} + b D_{e_0}\), \(a, b >0\). Then the tangent bundle $\mathcal{T}_X$ is \(H\)-(semi)stable if and only if \[ \sum\limits_{i=r-1}^{r-1+s} \binom{r-1+s}{i}  a^{r-1+s-i} b^i (\leq)< \frac{(sr+s+r)}{s(r+1)} \sum\limits_{i=r}^{r-1+s} \binom{r-1+s}{i}  a^{r-1+s-i} b^i.\]

\end{thm}

\begin{proof}
Using Proposition \ref{RS2} and Lemma \ref{RSforr>1}, we obtain that 
\begin{equation*}
\begin{split}
\mathcal{T}_X \text{ is (semi)stable }  & \Longleftrightarrow \text{deg }D_{e_0}+\frac{1}{r}\left( \text{deg }D_{e_0}-  \text{deg }D_{v_0}\right) (\leq) <  \mu(\mathcal{T}_X)\\
& \Longleftrightarrow \text{deg }D_{e_0}( \leq) < \frac{(sr+s+r)}{s(r+1)} \text{deg }D_{v_0} \ ( \text{see } \eqref{slopeTpic2}). 
\end{split}
\end{equation*}
Now the theorem follows by Lemma \ref{deg}.
\end{proof}

The following remark is useful for studying (semi)stability of tangent bundle of product of nonsingular Fano varieties.

\begin{rmk}\label{prodstab} {\rm
		Let \(Y_1\) and \(Y_2\)	be two nonsingular Fano varieties of dimension \(n_1\) and \(n_2\) respectively. Then \(X=Y_1 \times Y_2\) is also a nonsingular Fano variety whose dimension is \(n=n_1+n_2\). Also one can see that \(\mathcal{T}_X=\pi^*_1 \mathcal{T}_{Y_1} \oplus \pi^*_2 \mathcal{T}_{Y_2} \) and \(\mu(\mathcal{T}_X)=  \mu(\pi^*_1 \mathcal{T}_{Y_1})=\mu(\pi^*_2 \mathcal{T}_{Y_2})\), where \(\pi_i:X \rightarrow Y_i\), \(i=1, 2\) is the projection map. Now if both $\mathcal{T}_{Y_1}$ and \(\mathcal{T}_{Y_2}\) are semistable, then \(\mathcal{T}_X\) is strictly semistable (see \cite[Examples 3.2]{steffens}).
		
	}
\end{rmk}

\begin{cor}\label{stabonPic2cor}
	Let \(X = \mathbb{P}(\mathcal{O}_{\mathbb{P}^s}  \oplus \mathcal{O}_{\mathbb{P}^s}(a_1) \oplus \cdots \oplus \mathcal{O}_{\mathbb{P}^s}(a_r)  )\), \(s, r \geq 1\), \(0 \leq a_1 \leq \ldots \leq a_r\), with \(a_1+\cdots +a_r < s+1\), i.e. \(X\) is Fano. Then with respect to the ample anticanonical divisor \(-K_X\), we have the following:
	
	\begin{enumerate}
		\item The tangent bundle $\mathcal{T}_X$ is unstable whenever \((a_1, \ldots, a_r) \neq (0, 0, \ldots, 0, 1) \) and \(a_r>0\).
		\item If \(r=1\) and $a_1=1$, the tangent bundle $\mathcal{T}_X$ is unstable for \(s \geq 2\). It is strictly semistable for \(s=1\).
		\item If \(r>1\) and \((a_1, \ldots, a_r) = (0, 0, \ldots, 0, 1) \), the tangent bundle $\mathcal{T}_X$ is (semi)stable if and only if \[ \sum\limits_{i=r-1}^{r-1+s} \binom{r-1+s}{i}  s^{r-1+s-i} (r+1)^i (\leq)< \frac{(sr+s+r)}{s(r+1)} \sum\limits_{i=r}^{r-1+s} \binom{r-1+s}{i}  s^{r-1+s-i} (r+1)^i .\]
	
		\item If \(a_r=0\), the tangent bundle $\mathcal{T}_X$ is strictly semistable.
		\end{enumerate}
	
\end{cor}

\begin{proof}
	Clearly, (1) follows from Theorem \ref{M1}.
	
	For (2), using Theorem \ref{M3}, we get that \(\mathcal{T}_X\) is (semi)stable  if and only if \((2s+1)s^s(  \leq) <  (s+2)^s\) (see also \cite[Theorem 8.1]{biswas2018stability}). 
	
	Note that for \(s=1\), the equality holds, hence in this case \(\mathcal{T}_X\) is strictly semistable.
	
	For \(s\geq 2\), using induction it can be shown that \((2s+1)s^s > (s+2)^s\) holds. Hence \(\mathcal{T}_X\) unstable whenever \(s\geq 2\).
	
	Furthermore (3) follows from Theorem \ref{M3}, for the particular values \(a=s\), \(b=r+1\) and (4) is immediate from Remark \ref{prodstab}.
\end{proof}

\section{Stability of tangent bundle on Fano 4-folds with Picard number 3}\label{sec:stability-of-tangent-bundle-on-fano-4-folds-with-picard-number-3}

In this section we are interested in (semi)stability of tangent bundles (with respect to the anticanonical divisor) of toric Fano 4-folds with Picard number 3 which were classified by Batyrev \cite[Section 4]{batyrev}.

\subsection{ Stability of tangent bundle of $\mathbb{P}^2$-bundle over \(\mathbb{P}^1 \times \mathbb{P}^1\)}

Let \(X=\mathbb{P}(\mathcal{O}_{\mathbb{P}^1 \times \mathbb{P}^1} \oplus \mathcal{O}_{\mathbb{P}^1 \times \mathbb{P}^1}(\alpha,0) \\*
 \oplus \mathcal{O}_{\mathbb{P}^1 \times \mathbb{P}^1}(\beta,\gamma)) \). Let $\Delta$ be the fan of \(X\) whose rays are given by
\begin{equation*}
\begin{split}
&\mathbf{u}_0 =(-1, 0, \alpha, \beta), \mathbf{u}_1=(1, 0, 0, 0), \mathbf{v}_0=(0, -1, 0, \gamma),\mathbf{v}_1=(0, 1, 0, 0) \text{ and} \\
&\mathbf{e}_0=(0,0,-1,-1), \mathbf{e}_1=(0,0,1,0), \mathbf{e}_2=(0,0,0,1),
\end{split}
\end{equation*}
and the maximal cones are given by
\begin{equation*}
\begin{split}
& \text{Cone}(\mathbf{e}_0,\ldots, \widehat{\mathbf{e}}_i, \ldots,\mathbf{e}_2) + \text{Cone}(\mathbf{u}_j,\mathbf{v}_k) , \text{ where } 0 \leq i \leq 2 \text{ and } 0 \leq j,k \leq 1.
\end{split}
\end{equation*}
We have the following relations 
\begin{equation*}
\begin{split}
& D_{\mathbf{u}_0} \sim_{\text{lin}} D_{\mathbf{u}_1}, D_{\mathbf{v}_0}\sim_{\text{lin}} D_{\mathbf{v}_1}, D_{\mathbf{e}_1} \sim_{\text{lin}} D_{\mathbf{e}_0}-\alpha D_{\mathbf{u}_0}, D_{\mathbf{e}_2} \sim_{\text{lin}} D_{\mathbf{e}_0}-\beta D_{\mathbf{u}_0}-\gamma D_{\mathbf{v}_0}.
\end{split}
\end{equation*}

Hence $\text{Pic}(X)=\Z D_{\mathbf{u}_0} \oplus \Z D_{\mathbf{v}_0} \oplus \Z D_{\mathbf{e}_0}.$ Let \(H=aD_{\mathbf{u}_0}+ bD_{\mathbf{v}_0}+c D_{\mathbf{e}_0}\). Note that \(D_{\mathbf{u}_0}^2=0, D_{\mathbf{v}_0}^2=0\) and hence we see that
\[H^3=3ac^2 D_{\mathbf{u}_0} D_{\mathbf{e}_0}^2+3bc^2 D_{\mathbf{v}_0} D_{\mathbf{e}_0}^2+6abc D_{\mathbf{u}_0} D_{\mathbf{v}_0} D_{\mathbf{e}_0} +c^3 D_{\mathbf{e}_0}^3.\]

Using the relations 
\begin{equation*}
\begin{split}
& D_{\mathbf{e}_0}D_{\mathbf{e}_1}D_{\mathbf{e}_2}=0, D_{\mathbf{u}_0} D_{\mathbf{v}_0} D_{\mathbf{e}_0}^2 =1, D_{\mathbf{u}_0} D_{\mathbf{e}_0}^3 = \gamma, D_{\mathbf{v}_0} D_{\mathbf{e}_0}^3=\alpha + \beta, D_{\mathbf{e}_0}^4 = \alpha \gamma+2 \beta \gamma,
\end{split}
\end{equation*}
%
%

we get 
\begin{equation}
\begin{split}
& \text{deg }D_{\mathbf{u}_0}= 3bc^2 + c^3 \gamma, \text{deg }D_{\mathbf{v}_0}= 3ac^2 + c^3 (\alpha+\beta)\\
&\text{deg }D_{\mathbf{e}_0}= 3ac^2 \gamma + 3bc^2 (\alpha+\beta)+ 6abc+ c^3 (\alpha \gamma+2\beta \gamma).
\end{split}
\end{equation}



When \(H=-K_X\), we have $a=2-\alpha-\beta, b=2-\gamma, c=3$.
Following the notation of \cite[Section 4]{batyrev}, \(X=D_7\) when \(\alpha=0, \beta=\gamma=1\), and \(X=D_{17}\) when \(\alpha=1,  \beta=0, \gamma=1\). 
\begin{prop}\label{stabPic3D_7}
Let \(X=\mathbb{P}(\mathcal{O}_{\mathbb{P}^1 \times \mathbb{P}^1} \oplus \mathcal{O}_{\mathbb{P}^1 \times \mathbb{P}^1}(\alpha,0) \oplus \mathcal{O}_{\mathbb{P}^1 \times \mathbb{P}^1}(\beta,\gamma)) \). Then 
\begin{enumerate}
	\item $\mathcal{T}_X$ is unstable if \(\alpha=0, \beta=\gamma=1\). 
	
	\item $\mathcal{T}_X$ is stable if \(\alpha=1,  \beta=0, \gamma=1\). 
\end{enumerate}
\end{prop}

\begin{proof}
\underline{(1) \(\alpha=0, \beta=\gamma=1\) :} Then \(a=b=1\) and \(c=3\).

\(\text{deg }D_{\mathbf{u}_0}=54\), \(\text{deg }D_{\mathbf{v}_0}=54\), \(\text{deg }D_{\mathbf{e}_0}=126\) and \(\mu(\mathcal{T}_X)=121.5\). Note that \(F= \text{Span}(e_0)\) corresponds to the destabilizing subsheaf \(\mathcal{O}_X(D_{\mathbf{e}_0})\). Hence $\mathcal{T}_X$ is unstable.

\underline{(2) \(\alpha=1,  \beta=0, \gamma=1\) :} Then \(a=b=1\) and \(c=3\).

\(\text{deg }D_{\mathbf{u}_0}=54\), \(\text{deg }D_{\mathbf{v}_0}=54\), \(\text{deg }D_{\mathbf{e}_0}=99\), \(\text{deg }D_{\mathbf{e}_1}=45\), \(\text{deg }D_{\mathbf{e}_2}=45\) and \(\mu(\mathcal{T}_X)=101.25\). By Remark \ref{sub_of_tangent}, to check (semi)stability, we only need to consider the following equivariant reflexive subsheaves \(\mathcal{F}\) with associated filtrations \(\left(F, \{F^{\rho}(i)\}_{\rho \in \Delta(1) \ i \in \Z} \right) \).

\underline{rank$(\mathcal{F})= 1$}
	\begin{align*}
	& (i) F=\text{Span}(\mathbf{u}_0), \text{ then } \mu(\mathcal{F}) =54. ~	(ii) F=\text{Span}(\mathbf{v}_0), \text{ then } \mu(\mathcal{F}) =54.\\
	& (iii) F=\text{Span}(\mathbf{e}_0), \text{ then } \mu(\mathcal{F}) =99.  ~ (iv) F=\text{Span}(\mathbf{e}_1), \text{ then } \mu(\mathcal{F}) =45. \\
	&(v) F=\text{Span}(\mathbf{e}_2), \text{ then } \mu(\mathcal{F}) =45. 
	\end{align*}
\underline{rank$(\mathcal{F})= 2$}
\begin{enumerate}[(i)]
	\item \(F=\text{Span}(\mathbf{e}_0,\mathbf{e}_1,\mathbf{e}_2)\), then \(\mu(\mathcal{F}) =94.5. \)
	\item \(F=\text{Span}(\mathbf{u}_0,\mathbf{u}_1,\mathbf{e}_1)\), then \(\mu(\mathcal{F}) =76.5. \)
	\item \(F=\text{Span}(\mathbf{v}_0,\mathbf{v}_1,\mathbf{e}_2)\), then \(\mu(\mathcal{F}) =76.5. \)
\end{enumerate}
\underline{rank$(\mathcal{F})= 3$}
\begin{enumerate}[(i)]
	\item \(F=\text{Span}(\mathbf{e}_0,\mathbf{e}_1,\mathbf{e}_2, \mathbf{u}_0,\mathbf{u}_1)\), then \(\mu(\mathcal{F}) =99. \)
	\item \(F=\text{Span}(\mathbf{e}_0,\mathbf{e}_1,\mathbf{e}_2, \mathbf{v}_0,\mathbf{v}_1)\), then \(\mu(\mathcal{F}) =99. \)
\end{enumerate}
Hence we see that $\mathcal{T}_X $ is stable. 
\end{proof}

\subsection{ Stability of tangent bundle of $\mathbb{P}^1$-bundle over \(\mathbb{P}^1 \times \mathbb{P}^2\)}

Let \(X=\mathbb{P}(\mathcal{O}_{\mathbb{P}^1 \times \mathbb{P}^2} \oplus  \mathcal{O}_{\mathbb{P}^1 \times \mathbb{P}^2}(\alpha, \beta)  )\) with associated fan $\Delta$. The rays of $\Delta$ are given by 
\begin{align*}
& w_0=(-1, 0, 0, \alpha), w_1=(1, 0,0,0), z_0=(0,-1,-1, \beta), z_1=(0,1,0,0),\\
& z_2=(0,0,1,0), e_0=(0,0,0,-1), e_1=(0,0,0,1)
\end{align*}
and the maximal cones are given by \[\text{Cone}(w_i, z_0, \ldots, \widehat{z}_j, \ldots, z_2, e_k), \text{ where } i =0, 1, \ 0 \leq j \leq 2, \ k =0, 1.\]

We have the following relations:  
\begin{equation}\label{relnpic3}
\begin{split}
& D_{w_0} \sim_{\text{lin}}  D_{w_1},  D_{z_0} \sim_{\text{lin}}  D_{z_j}, \text{ for } j=1, 2; D_{e_1} \sim_{\text{lin}} D_{e_0}-\alpha D_{w_0}- \beta D_{z_0}.
\end{split}
\end{equation}	

So we get $\text{Pic}(X)=\Z D_{w_0} \oplus \Z D_{z_0} \oplus \Z D_{e_0}.$ Let \(H=a D_{w_0}+b D_{z_0} +c D_{e_0}\). Note that \(D_{w_0}^2=0,  D_{z_0}^3=0\) and hence 
\[H^3=3ab^2 D_{w_0} D_{z_0}^2+ 3ac^2 D_{w_0} D_{e_0}^2+ 6abc D_{w_0} D_{z_0} D_{e_0}+3bc^2 D_{z_0} D_{e_0}^2+3b^2c D_{z_0}^2 D_{e_0}+c^3 D_{e_0}^3.\]
Also we have 
\begin{equation*}
\begin{split}
&  D_{e_0} \cdot D_{e_1}=0, D_{w_0} D_{z_0} D_{e_0}^2=\beta, \\
& D_{w_0} D_{e_0}^3=\beta^2, D_{z_0}^2 D_{e_0}^2=\alpha, D_{z_0} D_{e_0}^3=2 \alpha \beta, D_{e_0}^4=3\alpha \beta^2.\\
\end{split}
\end{equation*}



Now we consider the polarization \(H=-K_X\), i.e. \(a=2-\alpha\), \(b=3-\beta\), \(c=2\). Note that \(X=D_1, D_6, D_{18} \text{ and } D_{19}\) when \((\alpha, \beta)=(1, 2), (1,1), (-1, 2)  \text{ and }(-1,1)\) respectively, following the notations of \cite[Section 4]{batyrev}.

\begin{prop}\label{stabpic3D_1}
Let \(X=\mathbb{P}(\mathcal{O}_{\mathbb{P}^1 \times \mathbb{P}^2} \oplus  \mathcal{O}_{\mathbb{P}^1 \times \mathbb{P}^2}(\alpha, \beta)  )\). Then 
\begin{enumerate}
	\item $\mathcal{T}_X$ is unstable for \((\alpha, \beta)=(1,1) ,(1,2) \text{ and } (-1,2) \).
	
	\item $\mathcal{T}_X$ is stable for \((\alpha, \beta)=(-1,1)\).
\end{enumerate}
\end{prop}

\begin{proof}(1) Proof follows from the following table	
	
	\begin{longtable}{|c|c|c|c|c|c|c|c|c|}
		\hline 
	\((\alpha, \beta)\)  &  \(\text{deg} D_{w_0}\) &  $\text{deg} D_{z_0}$  &  \(\text{deg} D_{e_0}\) & \(\mu(\mathcal{T}_X)\) & \(F \)  & $c_1(\mathcal{F}) $ &\(\mu(\mathcal{F})\) \\ 
		\hline \hline
		 $(1,1)$ &  56 & 76 & 144 & 124 & $\text{Span}({e}_0)$ & $D_{e_0} +D_{e_1}$ & 156\\ 
		\hline 
		 $(1,2)$ &  62 & 80 & 225 & 148 & $\text{Span}({e}_0)$& $D_{e_0} +D_{e_1}$  & 228\\ 
		\hline 	
		 $(-1,2)$ &  62 & 64 & 75 & 100 & $\text{Span}(w_0, e_0)$& $D_{e_0} +D_{e_1}+D_{w_0} +D_{w_1}$  & 104\\ 
		\hline 
	\end{longtable}	
where $\mathcal{F}$ denotes the equivariant reflexive subsheaf of $\mathcal{T}_X$ corresponding to the subspace \(F\) of $\C^4$.
%
%
%
%
%
%

(2)
\underline{\(\alpha=-1, \beta=1\) :} Then \(a=3,b=2, c=2\).

 \(\text{deg }D_{w_0} =56\), \(\text{deg }D_{z_0} =68\), \(\text{deg }D_{e_0}  = 48\) and \(\mu(\mathcal{T}_X) =100\). Also \(\text{deg }(D_{e_0} +D_{e_1})=84\). Note that \(\mathcal{O}_X(D_{w_0}), \mathcal{O}_X(D_{w_1}), \mathcal{O}_X(D_{z_0}), \mathcal{O}_X(D_{z_1}), \mathcal{O}_X(D_{z_2}), \mathcal{O}_X(D_{e_0} +D_{e_1})\) are the only rank 1 equivariant reflexive subsheaves of $\mathcal{T}_X$ and all of them has degree less than \(\mu(\mathcal{T}_X)\).

Next we consider higher rank equivariant reflexive subsheaves of $\mathcal{T}_X$. The maximum possible slopes can occur only from the following situations.

\noindent
\underline{rank$(\mathcal{F})= 2$}
\begin{enumerate}[(i)]
	\item \(F=\text{Span}(e_0, e_1, z_j)\) for \(j=0,1,2\), then \(\mu(\mathcal{F})=76.\)
		\item \(F=\text{Span}(w_0, w_1, e_0, e_1)\), then \(\mu(\mathcal{F})=98.\)
\end{enumerate}
\underline{rank$(\mathcal{F})= 3$}
\begin{enumerate}[(i)]
	\item \(F=\text{Span}(e_0, e_1, z_0, z_1, z_2)\), then \(\mu(\mathcal{F})=96.\)
	\item \(F=\text{Span}(w_0, w_1, e_0, e_1,z_j)\) for \(j=0,1,2\), then \(\mu(\mathcal{F})=88.\)
\end{enumerate}

Hence in this case $\mathcal{T}_X$ is stable.
\end{proof}

\subsection{ Stability of tangent bundle of $\mathbb{P}^1$-bundle over $\mathbb{P}(\mathcal{O}_{\mathbb{P}^2} \oplus \mathcal{O}_{\mathbb{P}^2} (a_1)), \ a_1=1, 2$}

 Let \(X=\mathbb{P}(\mathcal{O}_{X'} \oplus \mathcal{O}_{X'}(\alpha, \beta) )\), where \(X'=\mathbb{P}(\mathcal{O}_{\mathbb{P}^2} \oplus \mathcal{O}_{\mathbb{P}^2} (a_1))\). The rays of the fan $\Delta$ of \(X\) are
\begin{equation*}
\begin{split}
& \mathbf{v}_0=(-1, -1, a_1, \alpha), \mathbf{v}_1=(1, 0, 0, 0), \mathbf{v}_2=(0, 1, 0, 0),\\
& \mathbf{e}_0'=(0, 0, -1, \beta), \mathbf{e}_1'=(0, 0, 1, 0), \mathbf{e}_1=(0, 0, 0, 1), \mathbf{e}_0=(0, 0, 0, -1)
\end{split}
\end{equation*}
and the maximal cones are
\begin{equation*}
\text{Cone}(\mathbf{v}_0, \ldots,\widehat{\mathbf{v}}_j, \ldots,\mathbf{v}_2, \mathbf{e}_p',\mathbf{e}_q), \text{ where }j=0, 1, 2 \text{ and } 0 \leq p, q \leq 1.
\end{equation*}

Note that we have the following relations
\begin{equation*}
\begin{split}
& D_{\mathbf{v}_0} \sim_{\text{lin}} D_{\mathbf{v}_1} \sim_{\text{lin}} D_{\mathbf{v}_2}, D_{\mathbf{e}_1'} \sim_{\text{lin}} D_{\mathbf{e}_0'} -a_1D_{\mathbf{v}_0}, D_{\mathbf{e}_1} \sim_{\text{lin}} D_{\mathbf{e}_0} - \alpha D_{\mathbf{v}_0} - \beta D_{\mathbf{e}_0'}.
\end{split}
\end{equation*}
Hence $\text{Pic}(X)=\Z D_{\mathbf{v}_0} \oplus \Z  D_{\mathbf{e}_0'} \oplus \Z   D_{\mathbf{e}_0}.$ Using Toric Nakai criterion and the fact that \(\mathbb{P}(\mathcal{O}_{X'} \oplus \mathcal{O}_{X'}(D) )  \cong \mathbb{P}(\mathcal{O}_{X'} \oplus \mathcal{O}_{X'}(-D)) \) for any divisor \(D\) on \(X'\), we only need to consider the following cases (comparing with the primitive relations listed in \cite[Proposition 3.1.2]{batyrev} ):
\begin{equation*}
\begin{split}
D_3& =\mathbb{P}(\mathcal{O}_{\mathcal{B}_2} \oplus \mathcal{O}_{\mathcal{B}_2}(1,1)), D_9=\mathbb{P}(\mathcal{O}_{\mathcal{B}_2} \oplus \mathcal{O}_{\mathcal{B}_2}(1,0)) \\
D_8&=\mathbb{P}(\mathcal{O}_{\mathcal{B}_2} \oplus \mathcal{O}_{\mathcal{B}_2}(0,1)) , D_{16}=\mathbb{P}(\mathcal{O}_{\mathcal{B}_2} \oplus \mathcal{O}_{\mathcal{B}_2}(-1,1)) \\
D_2&=\mathbb{P}(\mathcal{O}_{\mathcal{B}_1} \oplus \mathcal{O}_{\mathcal{B}_1}(0,1)), D_5= \mathbb{P}^1 \times \mathcal{B}_1, D_{12} =\mathbb{P}^1 \times \mathcal{B}_2,
\end{split}
\end{equation*}
where \(X'=\mathcal{B}_1\) for \(a_1=2\) and \(X'=\mathcal{B}_2\) for \(a_1=1\) from the notations of \cite[Remark 2.5.10]{batyrev}.

Let \(H=a D_{\mathbf{v}_0} + b D_{\mathbf{e}_0'} +c D_{\mathbf{e}_0}\). We have the following relations 
\[D_{\mathbf{v}_0}^3=D_{\mathbf{v}_1}^3=D_{\mathbf{v}_2}^3=0, D_{\mathbf{e}_0'} D_{\mathbf{e}_1'}=0,   D_{\mathbf{e}_0}  D_{\mathbf{e}_1}=0.\]
So we have 
\begin{equation*}
\begin{split}
H^3=& 3ab^2 D_{\mathbf{v}_0} D_{\mathbf{e}_0'}^2+3ac^2 D_{\mathbf{v}_0} D_{\mathbf{e}_0}^2 +3a^2b D_{\mathbf{v}_0}^2 D_{\mathbf{e}_0'}+ 6abc D_{\mathbf{v}_0} D_{\mathbf{e}_0'} D_{\mathbf{e}_0} +3a^2c D_{\mathbf{v}_0}^2 D_{\mathbf{e}_0} \\
& +b^3 D_{\mathbf{e}_0'}^3 + 3bc^2 D_{\mathbf{e}_0'} D_{\mathbf{e}_0}^2 + 3b^2c D_{\mathbf{e}_0'}^2 D_{\mathbf{e}_0} + c^3 D_{\mathbf{e}_0} ^3.
\end{split}
\end{equation*}

%
%
%

Furthermore, we have 
\begin{equation*}
\begin{split}
&D_{\mathbf{v}_0}^2 D_{\mathbf{e}_0'}^2=0, D_{\mathbf{v}_0}^2 D_{\mathbf{e}_0}^2=\beta, D_{\mathbf{v}_0} D_{\mathbf{e}_0'}^3=0, D_{\mathbf{v}_0} D_{\mathbf{e}_0'} D_{\mathbf{e}_0}^2= \alpha +a_1 \beta , \\
&D_{\mathbf{v}_0} D_{\mathbf{e}_0}^3= 2 \alpha \beta +a_1 \beta^2, D_{\mathbf{v}_0} D_{\mathbf{e}_0'}^2 D_{\mathbf{e}_0}=a_1, D_{\mathbf{e}_0'}^2 D_{\mathbf{e}_0}^2=a_1 \alpha+a_1^2 \beta, D_{\mathbf{e}_0'}^3 D_{\mathbf{e}_0}=a_1^2,\\
& D_{\mathbf{e}_0'} D_{\mathbf{e}_0}^3=(\alpha+a_1\beta)^2, D_{\mathbf{e}_0'}^4=0, D_{\mathbf{e}_0}^4=3 \alpha^2\beta + 3a_1 \alpha \beta^2 + a_1^2 \beta^3.
\end{split}
\end{equation*}

Now let us fix the polarization \(H=-K_X\), i.e. \(a=3-a_1-\alpha\), \(b=2-\beta\), \(c=2\).
\begin{prop}\label{D_2}
The tangent bundle of the following toric Fano 4-folds are unstable

\begin{tabular}{ccccccc}
 $(i) D_{12}$ & $(ii) D_9$	& \((iii) D_8\) & \((iv) D_3\) & $(v) D_{16}$ &	$(vi) D_5 $ &  $(vii) D_2$ 
\end{tabular} 
\end{prop}

\begin{proof}
Proof follows from the following table	
	
\begin{longtable}{|c|c|c|c|c|c|c|c|c|}
	\hline 
	\(a_1\)	& \((\alpha, \beta)\)  &  \(\text{deg} D_{\mathbf{v}_0}\) &  $\text{deg} D_{\mathbf{e}_0'}$  &  \(\text{deg} D_{\mathbf{e}_0}\) & \(\mu(\mathcal{T}_X)\) & \(F \)  & \(c_1(\mathcal{F})\)  &\(\mu(\mathcal{F})\) \\ 
	\hline \hline
	1	& $(0,0)$ &  72 & 96 & 56 & 112 & $\text{Span}(\mathbf{e}_0)$& \(D_{\mathbf{e}_0'}+D_{\mathbf{e}_1'}\) & 120\\ 
	\hline 	
	1 & $(1,0)$ & 72 & 98 & 98 & 116 & $\text{Span}(\mathbf{e}'_0)$& \(D_{\mathbf{e}_0'}+D_{\mathbf{e}_1'}\)  & 124 \\
	\hline
	1 & $(0,1)$ & 74 & 98 & 117 & 120 & $\text{Span}(\mathbf{e}_0)$&  \(D_{\mathbf{e}_0}+D_{\mathbf{e}_1}\) & 136 \\
	\hline
	1 & $(1,1)$ & 78 & 104 & 189 & 140 & $\text{Span}(\mathbf{e}_0)$& \(D_{\mathbf{e}_0}+D_{\mathbf{e}_1}\)  & 196 \\
	\hline
	1 & $(-1,1)$ & 70 & 96 & 63 & 108 & $\text{Span}(\mathbf{e}'_0, \mathbf{e}_0)$& \(D_{\mathbf{e}_0'}+D_{\mathbf{e}_1'}+ D_{\mathbf{e}_0}+D_{\mathbf{e}_1}\)  & 111 \\
	\hline
	2& $(0,0)$ & 72 & 150 & 62 & 124 & $\text{Span}(\mathbf{e}'_0)$& \(D_{\mathbf{e}_0'}+D_{\mathbf{e}_1'}\)  & 156 \\
	\hline
	2 & $(0,1)$ & 76 & 158 & 171 & 144 & $\text{Span}(\mathbf{e}'_0)$&  \(D_{\mathbf{e}_0'}\) & 158 \\
	\hline	
\end{longtable}	
where $\mathcal{F}$ denotes the equivariant reflexive subsheaf of $\mathcal{T}_X$ corresponding to the subspace \(F\) of $\C^4$.
\end{proof}

\subsection{Stability of tangent bundle of \(\mathbb{P}^2 \)-bundle over the Hirzebruch surface $\mathcal{H}_1$}

Let \(X=\mathbb{P}(\mathcal{O}_{\mathcal{H}_1} \oplus \mathcal{O}_{\mathcal{H}_1} \oplus \mathcal{O}_{\mathcal{H}_1}(\alpha, \beta)) \). The rays of the fan $\Delta$ of \(X\) are given by:
\begin{equation*}
\begin{split}
& \mathbf{v}_1=(1,0,0,0), \mathbf{v}_2=(0, 1, 0, 0), \mathbf{v}_3=(-1, 1, 0, \alpha),\mathbf{v}_4=(0, -1, 0, \beta)\\
& \mathbf{e}_0=(0, 0, -1, -1), \mathbf{e}_1=(0, 0, 1, 0), \mathbf{e}_2=(0, 0, 0, 1),
\end{split}
\end{equation*}
and the maximal cones are given by
\begin{equation*}
\begin{split}
& \text{Cone}(\mathbf{e}_0,\ldots, \widehat{\mathbf{e}}_j,\ldots,  \mathbf{e}_2) + \text{Cone}(\mathbf{v}_i, \mathbf{v}_{i+1}), \text{ where } j=0, 1, 2 \text{ and } 1 \leq i \leq 4.
\end{split}
\end{equation*}

Now we have the following relations
\begin{equation*}
\begin{split}
& D_{\mathbf{v}_1} \sim_{\text{lin}} D_{\mathbf{v}_3}, D_{\mathbf{v}_2} \sim_{\text{lin}} D_{\mathbf{v}_4}-D_{\mathbf{v}_3}, D_{\mathbf{e}_1} \sim_{\text{lin}} D_{\mathbf{e}_0}, D_{\mathbf{e}_2} \sim_{\text{lin}} D_{\mathbf{e}_0} - \alpha D_{\mathbf{v}_3} - \beta D_{\mathbf{v}_4}.
\end{split}
\end{equation*}

Hence $\text{Pic}(X)=\Z D_{\mathbf{v}_3} \oplus \Z  D_{\mathbf{v}_4} \oplus \Z   D_{\mathbf{e}_0}.$ Note that using Toric Nakai criterion we see that 	\(X=\mathbb{P}(\mathcal{O}_{X'} \oplus \mathcal{O}_{X'} \oplus \mathcal{O}_{X'}(\alpha, \beta)) \) is Fano if and only if \(\alpha=0, \beta=0,1\). We consider the case for \((\alpha, \beta)=(0,1)\), i.e. \(X=D_{11}\) in the notation of \cite[Section 4]{batyrev}).

\begin{prop}\label{D_{11}}
Let \(X=\mathbb{P}(\mathcal{O}_{\mathcal{H}_1} \oplus \mathcal{O}_{\mathcal{H}_1} \oplus \mathcal{O}_{\mathcal{H}_1}(0, 1)) \). Then $\mathcal{T}_X$ is unstable. 
\end{prop}

\begin{proof}
Note that $-K_X= D_{\mathbf{v}_3} + D_{\mathbf{v}_4}+ 3D_{\mathbf{e}_0}$. Since \(D_{\mathbf{v}_3}^2=0\), we have
\begin{equation*}
\begin{split}
(-K_X)^3=3D_{\mathbf{v}_3}D_{\mathbf{v}_4}^2+27D_{\mathbf{v}_3}D_{\mathbf{e}_0}^2+18D_{\mathbf{v}_3}D_{\mathbf{v}_4}D_{\mathbf{e}_0}+D_{\mathbf{v}_4}^3+27 D_{\mathbf{v}_4}D_{\mathbf{e}_0}^2+9D_{\mathbf{v}_4}^2D_{\mathbf{e}_0}+27 D_{\mathbf{e}_0}^3.
\end{split}
\end{equation*}
%
%
%

Now we compute the following intersection products
 \begin{equation*}
 \begin{split}
 &D_{\mathbf{v}_3}D_{\mathbf{v}_4}^3=0,D_{\mathbf{v}_3}D_{\mathbf{v}_4}D_{\mathbf{e}_0}^2=1 , D_{\mathbf{v}_3}D_{\mathbf{v}_4}^2D_{\mathbf{e}_0}=0, D_{\mathbf{v}_3}D_{\mathbf{e}_0}^3=1, D_{\mathbf{v}_4}^4=0, \\
 &  D_{\mathbf{v}_4}^2D_{\mathbf{e}_0}^2=1, D_{\mathbf{v}_4}^3D_{\mathbf{e}_0}=0, D_{\mathbf{v}_4}D_{\mathbf{e}_0}^3=1, D_{\mathbf{e}_0}^4=1.
 \end{split}
 \end{equation*}


%






Thus we have $\text{deg}D_{\mathbf{v}_3}=54, \text{deg}D_{\mathbf{v}_4}=81, \text{deg}D_{\mathbf{e}_0}=108$ and $\mu(\mathcal{T}_X)= 114.75$.


Let \(F=\text{Span}(\mathbf{e}_1,\mathbf{e}_2, \mathbf{e}_0)\). Then it corresponds to a rank 2  destabilizing reflexive subsheaf $\mathcal{F}$ of $\mathcal{T}_X$ with $\mu(\mathcal{F})=121.5.$ Hence $\mathcal{T}_X$ is unstable.
\end{proof}

\subsection{Stability of tangent bundle of \(\mathbb{P}^1\)-bundle over \(\mathbb{P}(\mathcal{O}_{\mathbb{P}^1} \oplus \mathcal{O}_{\mathbb{P}^1} \oplus \mathcal{O}_{\mathbb{P}^1}(1))\)}

Let \(X=\mathbb{P}(\mathcal{O}_{X'} \oplus \mathcal{O}_{X'}(\alpha, \beta)) \), where \(X'=\mathbb{P}(\mathcal{O}_{\mathbb{P}^1} \oplus \mathcal{O}_{\mathbb{P}^1} \oplus \mathcal{O}_{\mathbb{P}^1}(1))\). The rays of the fan $\Delta$ of \(X\) are given by
\begin{equation*}
\begin{split}
&\mathbf{v}_0=(-1, 0, 1, \alpha), \mathbf{v}_1=(1, 0, 0, 0), \mathbf{e}_1'=(0, 1, 0, 0),\mathbf{e}_2'=(0, 0, 1, 0)\\
&\mathbf{e}_0'=(0, -1, -1, \beta), \mathbf{e}_1=(0, 0, 0, 1), \mathbf{e}_0=(0, 0, 0, -1)
\end{split}
\end{equation*}
and the maximal cones are given by \[\text{Cone}( \mathbf{v}_i,\mathbf{e}_0', \ldots, \widehat{\mathbf{e}_j'},  \ldots, \mathbf{e}_2', \mathbf{e}_k) \text{ for } i=0,1, j=0,1,2 \text{ and } k=0,1.\]

We have the following relations 
\begin{equation*}
\begin{split}
&D_{\mathbf{v}_1}\sim_{\text{lin}} D_{\mathbf{v}_0}, D_{\mathbf{e}_1'} \sim_{\text{lin}} D_{\mathbf{e}_0'}, D_{\mathbf{e}_2'} \sim_{\text{lin}} D_{\mathbf{e}_0'}-D_{\mathbf{v}_0}, D_{\mathbf{e}_1} \sim_{\text{lin}} D_{\mathbf{e}_0}-\alpha_1 D_{\mathbf{v}_0}- \beta D_{\mathbf{e}_0'}.
\end{split}
\end{equation*}

Hence $\text{Pic}(X)=\Z D_{\mathbf{v}_0} \oplus \Z D_{\mathbf{e}_0'} \oplus \Z D_{\mathbf{e}_0}.$ Now using Toric Nakai criterion one can see that \(X=\mathbb{P}(\mathcal{O}_{X'} \oplus \mathcal{O}_{X'}(\alpha, \beta)) \) is Fano if and only if \(\alpha=0, -2 \leq \beta \leq 2\). It suffices to consider $\beta=1, 2$. Note that \(X=D_{10}\) for $(\alpha, \beta)=(0,1)$ and \(X=D_4\) for $(\alpha, \beta)=(0,2)$ from the notations of \cite[Remark 2.5.10, Section 4]{batyrev}.

\begin{prop}\label{D_4}
Let \(X=\mathbb{P}(\mathcal{O}_{X'} \oplus \mathcal{O}_{X'}(0, \beta)) \), where \(X'=\mathbb{P}(\mathcal{O}_{\mathbb{P}^1} \oplus \mathcal{O}_{\mathbb{P}^1} \oplus \mathcal{O}_{\mathbb{P}^1}(1))\) and \(\beta=1,2\). Then $\mathcal{T}_X$ is unstable.	
\end{prop}

\begin{proof}
The anticanonical divisor is given by \(-K_X=D_{\mathbf{v}_0}+(3- \beta)D_{\mathbf{e}_0'}+ 2 D_{\mathbf{e}_0}.\) Since $D_{\mathbf{v}_0}^2=0 , D_{\mathbf{e}_0'}D_{\mathbf{e}_1'}D_{\mathbf{e}_2'}=0, D_{\mathbf{e}_0}D_{\mathbf{e}_1}=0$, we have
\begin{equation*}
\begin{split}
(-K_X)^3=&3(3- \beta)^2D_{\mathbf{v}_0} D_{\mathbf{e}_0'}^2+12 D_{\mathbf{v}_0} D_{\mathbf{e}_0}^2 +12(3- \beta)D_{\mathbf{v}_0} D_{\mathbf{e}_0'}D_{\mathbf{e}_0}+(3- \beta)^3 D_{\mathbf{e}_0'}^3\\
&+12(3- \beta)D_{\mathbf{e}_0'}D_{\mathbf{e}_0}^2+6(3- \beta)^2D_{\mathbf{e}_0'}^2 D_{\mathbf{e}_0} +8 D_{\mathbf{e}_0}^3.
\end{split}
\end{equation*}

%
%
%

Furthermore, we have the following relations 
\begin{equation*}
\begin{split}
& D_{\mathbf{v}_0} D_{\mathbf{e}_0'}^3=0, D_{\mathbf{v}_0} D_{\mathbf{e}_0'}D_{\mathbf{e}_0}^2=\beta, D_{\mathbf{v}_0} D_{\mathbf{e}_0}^3= \beta^2,  D_{\mathbf{e}_0'}^2 D_{\mathbf{e}_0}^2= \beta,\\
& D_{\mathbf{e}_0'}^3 D_{\mathbf{e}_0}=1,
D_{\mathbf{e}_0'} D_{\mathbf{e}_0}^3= \beta^2,
D_{\mathbf{e}_0}^4= \beta ^3, D_{\mathbf{e}_0'}^4=0.
\end{split}
\end{equation*}

Hence we have 

$
\text{deg } D_{\mathbf{v}_0}
= \left\{ \begin{array}{ccc}
56 & \beta=1 \\
62 & \beta=2
\end{array} \right.; \ 
\text{deg } D_{\mathbf{e}_0'}
= \left\{ \begin{array}{ccc}
	92 & \beta=1 \\
	98 & \beta=2
\end{array} \right.; \ 
\text{deg } D_{\mathbf{e}_0}
= \left\{ \begin{array}{ccc}
	112 & \beta=1 \\
	200 & \beta=2.
\end{array} \right.
$

Hence 
$
\mu(\mathcal{T}_X)
= \left\{ \begin{array}{ccc}
116 & \beta=1 \\
140 & \beta=2.
\end{array} \right.
$
\\
Note that \(\mathcal{O}_X(D_{\mathbf{e}_0}+ D_{\mathbf{e}_1})\) is a rank 1 reflexive subsheaf of $\mathcal{T}_X$, whose degree is given by

$
\text{deg } (D_{\mathbf{e}_0} + D_{\mathbf{e}_1} )=2 \text{deg } D_{\mathbf{e}_0} -  \beta \text{deg } D_{\mathbf{e}_0'}  
= \left\{ \begin{array}{ccc}
132 & \beta=1 \\
204 & \beta=2.
\end{array} \right.
$

Hence $\mathcal{T}_X$ is unstable.
\end{proof}

\subsection{Stability of tangent bundle of blow up of $\mathbb{P}^2$ on \(\mathbb{P}(\mathcal{O}_{\mathbb{P}^3} \oplus  \mathcal{O}_{\mathbb{P}^3}(a_1)), a_1=0,1,2\)}

Let \(X'=\mathbb{P}(\mathcal{O}_{\mathbb{P}^3} \oplus  \mathcal{O}_{\mathbb{P}^3}(a_1))\). The fan $\Delta'$ associated to \(X'\) is given as follows. Let \(u_1, u_2, u_3\) be the standard basis of \(\Z^3\) and \(e_1'\) be that of \(\Z\). Set \(v_i=(u_i, 0) \in \Z^4 \) for \(i=1, 2,3\), \(e_1=(0,0,0, e_1') \in \Z^4 \), \(e_0=-e_1\) and \(v_0=-v_1-v_2-v_3+a_1 e_1\). Then $\Delta'(1)=\{v_0, v_1, v_2, v_3, e_0, e_1\}$ and maximal cones are of the form 
\[\text{Cone}(v_0,\ldots, \widehat{v}_j, \ldots, v_3, e_0) \text{ and } \text{Cone}(v_0,\ldots, \widehat{v}_j, \ldots, v_3, e_1) \text{ for } j=0, 1, 2, 3.\]
Note that \(\text{Pic}(X')=\Z D_{v_0}  \oplus \Z  D_{e_0}\). For $\tau=\text{Cone}(v_0, e_1) \in \Delta'$, we have \(V(\tau)=\mathbb{P}^2\). Let \(X=Bl_{V(\tau)}(X)\) with associated fan $\Delta$. Then the rays  of \(\Delta\) are
\begin{equation*}
	\begin{split}
		& v_1=(1,0,0,0), v_2=(0,1,0,0), v_3=(0,0,1,0), v_0=(-1,-1,-1, a_1)\\
		& e_1=(0,0,0,1), e_0=(0,0,0,-1), u_{\tau}=(-1,-1,-1,a_1+1).
	\end{split}
\end{equation*}
We have the following relations
\begin{equation}\label{blow1rel1}
	\begin{split}
		& D_{v_1} \sim_{\text{lin}} D_{v_2} \sim_{\text{lin}} D_{v_3} \sim_{\text{lin}} D_{v_0}+D_{u_{\tau}}, D_{e_1} \sim_{\text{lin}} D_{e_0}-a_1D_{v_0} -(a_1+1)D_{u_{\tau}}. 
	\end{split}
\end{equation}
Hence $\text{Pic}(X)=\Z D_{v_0}  \oplus \Z  D_{e_0} \oplus \Z D_{u_{\tau}}.$ The anticanonical divisor is given by	$-K_X=(4-a_1) D_{v_0}+2 D_{e_0}+(3-a_1)D_{u_{\tau}}.$

Note that \(X=E_1,E_2,E_3\) for \(a_1=2,1,0\) respectively in the notation of \cite[Section 4]{batyrev}.

\begin{prop}\label{E_1}
Let $X=Bl_{V(\tau)}(X')$, where \(X'=\mathbb{P}(\mathcal{O}_{\mathbb{P}^3} \oplus  \mathcal{O}_{\mathbb{P}^3}(a_1))\) and $\tau=\text{Cone}(v_0, e_1) \in \Delta'$. Then
\begin{itemize}
	\item[(1)] $\mathcal{T}_X$ is unstable for \(a_1=1,2\).
	\item[(2)] $\mathcal{T}_X$ is stable for \(a_1=0\).
\end{itemize}
\end{prop}

\begin{proof}
Let \(-K_X=a D_{v_0}+2 D_{e_0}+bD_{u_{\tau}}\), where \(a=4-a_1\) and $b=3-a_1$. Note that we have \(D_{e_0} D_{u_{\tau}}=0, D_{e_0} D_{e_1}=0 \text{ and } D_{v_0} D_{e_1}=0\). So we have 
\begin{equation*}
\begin{split}
(-K_X)^3&=a^3 D_{v_0}^3+12a D_{v_0}D_{e_0}^2+3ab^2 D_{v_0} D_{u_{\tau}}^2+ 6a^2 D_{v_0}^2 D_{e_0}+ 8 D_{e_0}^3+ 3a^2b D_{v_0}^2 D_{u_{\tau}}+b^3 D_{u_{\tau}}^3.
\end{split}
\end{equation*}

Furthermore, we have
\begin{equation*}
\begin{split}
& D_{v_0}^4=-a_1^2-3a_1-3, D_{v_0}^2 D_{e_0}^2=a_1, D_{v_0}^2 D_{u_{\tau}}^2=-a_1^2-a_1, D_{v_0}^3 D_{e_0}=1,\\
& D_{v_0} D_{e_0}^3=a_1^2, D_{v_0} D_{u_{\tau}}^3=a_1^2, D_{v_0}^3 D_{u_{\tau}}=(a_1+1)^2, D_{e_0}^4=a_1^3, D_{u_{\tau}}^4=-a_1^2+a_1-1.
\end{split}
\end{equation*}

\((1)\)\underline{\(a_1=2\):}  Here \(a=2, b=1\). We compute that \(\text{deg }D_{v_0} =76, \text{deg }D_{e_0}=216, \text{deg }D_{u_{\tau}}=21 \text{ and }\mu(\mathcal{T}_X)=151.25.\)

Note that \(\mathcal{O}_X(D_{e_0}+D_{e_1})\) is a rank 1 reflexive subsheaf of \(\mathcal{T}_X\) with degree \(\text{deg }(D_{e_0}+D_{e_1}) =217\). Hence $\mathcal{T}_X$ is unstable.

\underline{\(a_1=1\):} Here \(a=3, b=2\). We have \(\text{deg }D_{v_0}=61, \text{deg }D_{e_0}=125, \text{deg }D_{u_{\tau}}=28 \text{ and } \mu(\mathcal{T}_X) =122.25.\)

Note that \(\mathcal{O}_X(D_{e_0}+D_{e_1})\) is a rank 1 reflexive subsheaf of \(\mathcal{T}_X\) with degree \(\text{deg }(D_{e_0}+D_{e_1})=133\). Hence $\mathcal{T}_X$ is unstable.

(2) \underline{\(a_1=0\):} Here \(a=4, b=3\). We have \(\text{deg }D_{v_0}=48, \text{deg }D_{e_0}=64, \text{deg }D_{u_{\tau}} =37 \text{ and } \mu(\mathcal{T}_X)=107.75.\) Also \(\text{deg }(D_{e_0}+D_{e_1})=91, \text{deg }D_{v_1}=85.\) Note that rank 1 equivariant reflexive subsheaves are \(\mathcal{O}_X (D_{v_0}), \mathcal{O}_X(D_{v_1}), \mathcal{O}_X(D_{v_2}), \mathcal{O}_X(D_{v_3}), \mathcal{O}_X(D_{e_0}+D_{e_1})\) and \(\mathcal{O}_X(D_{u_{\tau}})\).

Next we consider reflexive subsheaves of $\mathcal{T}_X$ of rank 2 and 3. The maximum possible slopes can occur only from the following situations.

\noindent
\underline{rank$(\mathcal{F})= 2$}
\begin{enumerate}[(i)]
	\item \(F=\text{Span}( v_1, e_0, e_1)\), then \(\mu(\mathcal{F})=88.\)
	\item \(F=\text{Span}(v_0, e_0, e_1,u_{\tau})\), then \(\mu(\mathcal{F})=88.\)
\end{enumerate}
\underline{rank$(\mathcal{F})= 3$}
\begin{enumerate}[(i)]
	\item \(F=\text{Span}(v_0, v_1, v_2, v_3)\), then \(\mu(\mathcal{F})=101.\)
	\item \(F=\text{Span}(v_0, v_1, e_0, e_1,u_{\tau})\), then \(\mu(\mathcal{F})=87.\)
\end{enumerate}
Hence in this case $\mathcal{T}_X$ is stable.
\end{proof}

\subsection{Stability of tangent bundle of \(G_1\)-\(G_6\) in the notation of \cite[Section 4]{batyrev}}

Let \(X=G_1\). We write down the associated fan $\Delta$ using the primitive relations from \cite[Proposition 3.1.2]{batyrev}). The rays of $\Delta$ are
\begin{equation*}
\begin{split}
& v_{1}=(1,0,0,0), v_2=(0,1,0,0), v_3=(1,-1,-1,0), v_4=(0,0,1,0)\\
& v_5=(0,0,0,1), v_6=(2,0,-1,-1), v_7=(-1,0,0,0),
\end{split}
\end{equation*}
and the maximal cones are given by the following condition
\begin{equation*}
\begin{split}
\sigma=\text{Cone}(v_i, v_j, v_k, v_l) \in \Delta \Longleftrightarrow & \text{Cone}(v_1, v_7), \text{Cone}(v_2,v_3,v_4), \text{Cone}(v_4,v_5,v_6),\\
&  \text{Cone}(v_5, v_6, v_7), \text{Cone}(v_1, v_2, v_3) \nsubseteq \sigma. 
\end{split}
\end{equation*}
We have the following relations
\begin{equation*}
\begin{split}
& D_{v_2} \sim_{\text{lin}} D_{v_3}, D_{v_5} \sim_{\text{lin}} D_{v_6}, D_{v_4} \sim_{\text{lin}} D_{v_3} + D_{v_6}, D_{v_1} \sim_{\text{lin}} D_{v_7}-D_{v_3}-2 D_{v_6}. 
\end{split}
\end{equation*}
Hence $\text{Pic}(X)=\Z D_{v_3} \oplus \Z D_{v_6} \oplus \Z D_{v_7}.$ The anticanonical divisor is \(-K_X=2 D_{v_3} +D_{v_6}+ 2 D_{v_7}.\)

\begin{prop}\label{G_1}
	The tangent bundle of \(X=G_1\) is unstable.
\end{prop}

\begin{proof}
We have
	\begin{equation*}
	\begin{split}
	(-K_X)^3 &= 8 D_{v_3}^3+ 6 D_{v_3} D_{v_6}^2+ 24D_{v_3} D_{v_7}^2+ 12 D_{v_3}^2 D_{v_6}+ 24 D_{v_3} D_{v_6} D_{v_7}+ 24 D_{v_3}^2 D_{v_7}  \\ 
	& + D_{v_6}^3 + 12 D_{v_6} D_{v_7}^2+ 6D_{v_6}^2 D_{v_7}+ 8 D_{v_7}^3.
	\end{split}
	\end{equation*}

	Using the following relations 
	\begin{equation*}
	\begin{split}
	& D_{v_3}^4=1, D_{v_3}^2 D_{v_6}^2=1, D_{v_3}^2 D_{v_7}^2=1, D_{v_3}^3 D_{v_6}=-1, D_{v_3}^2 D_{v_6} D_{v_7}=1,\\
	&  D_{v_3}^3 D_{v_7}=-1, D_{v_3} D_{v_6} D_{v_7}^2=1, D_{v_3} D_{v_6}^2 D_{v_7}=0, D_{v_3} D_{v_7}^3=3, \\
	& D_{v_3} D_{v_6}^3=-1, D_{v_6}^2 D_{v_7}^2=0, D_{v_6}^3 D_{v_7}=0, D_{v_6} D_{v_7}^3=1, D_{v_6}^4=1, D_{v_7}^4=5,
	\end{split}
	\end{equation*}
	we have \(\text{deg }D_{v_3}=61, \text{deg }D_{v_6}=55, \text{deg }D_{v_7}=176 \text{ and } \mu(\mathcal{T}_X)=132.25.\) Note that \(\mathcal{O}_X(D_{v_1}+ D_{v_7})\) is a destabilizing subsheaf of \(\mathcal{T}_X\) with degree \(181\), hence \(\mathcal{T}_X\) is unstable.
\end{proof}

Let \(X'=\mathbb{P}(\mathcal{O}_{\mathbb{P}^2}  \oplus \mathcal{O}_{\mathbb{P}^2}(\alpha)  \oplus \mathcal{O}_{\mathbb{P}^2}(\beta))\). The fan $\Delta'$ associated to \(X'\) is given as follows. Let \(u_1, u_2\) be the standard basis of $\Z^2$ and \(e_1', e_2'\) also denote the standard basis of $\Z^2$. Set \(v_i=(u_i, 0, 0)\) for \(i=1, 2\) and \(e_j=(0,0, e_j')\) for \(j=1,2\), \(e_0=-e_1-e_2\) and \(v_0=-v_1-v_2+ \alpha e_1+ \beta e_2\). Then $\Delta'(1)=\{v_0, v_1, v_2, e_0, e_1, e_2\}$ and the maximal cones are of the form 
\[\text{Cone}(v_0,\ldots, \widehat{v}_i, \ldots, v_2, e_0,\ldots, \widehat{e}_j,\ldots,e_2 ) \text{ for } i, j=0, 1, 2.\]
Note that \(\text{Pic}(X')=\Z D_{v_0} \oplus  \Z D_{e_0}\).

\begin{prop}\label{G_2}
Let \(X\) be the	blow up of $V(\tau)$ on \(X'\), where \(\tau=\text{Cone}(v_0, e_2) \in \Delta' \) and $\alpha=0, \beta=1$ $($note that \(X=G_2\) in the notation of \cite[Section 4]{batyrev}$)$. Then $\mathcal{T}_X$ is unstable.
\end{prop}

\begin{proof}
Then the rays of the fan $\Delta$ associated to \(X\) are as follows 
\begin{equation*}
\begin{split}
& v_1=(1,0,0,0), v_2=(0,1,0,0), v_0=(-1,-1,0,1), e_1=(0,0,1,0)\\
& e_2=(0,0,0,1), e_0=(0,0,-1,-1), u_{\tau}=(-1,-1,0,2).
\end{split}
\end{equation*}
We have the following relations
\begin{equation*}
\begin{split}
 D_{v_1} \sim_{\text{lin}} D_{v_2}  \sim_{\text{lin}} D_{v_0}+D_{u_{\tau}}, D_{e_1} \sim_{\text{lin}} D_{e_0}, D_{e_2} \sim_{\text{lin}} D_{e_0} -D_{v_0}-2D_{u_{\tau}}.
\end{split}
\end{equation*}
Hence $\text{Pic}(X)=\Z D_{v_0}  \oplus \Z  D_{e_0} \oplus \Z D_{u_{\tau}} .$ Then the anticanonical divisor is \(-K_X=2 D_{v_0}+3 D_{e_0}+D_{u_{\tau}}\). We have 
\begin{equation*}
\begin{split}
(-K_X)^3 &=8 D_{v_0}^3 + 54 D_{v_0} D_{e_0}^2+ 6D_{v_0} D_{u_{\tau}}^2+ 36 D_{v_0}^2 D_{e_0} + 36 D_{v_0} D_{e_0} D_{u_{\tau}}\\
& + 12 D_{v_0}^2 D_{u_{\tau}}+ 27 D_{e_0}^3+ 9 D_{e_0} D_{u_{\tau}}^2+ 27 D_{e_0}^2 D_{u_{\tau}}+ D_{u_{\tau}}^3.
\end{split}
\end{equation*}

Using the following relations
\begin{equation*}
\begin{split}
& D_{v_0} D_{e_2}=0, D_{v_0}^4=5, D_{v_0}^3D_{u_{\tau}}=-4, D_{v_0}D_{e_0}D_{u_{\tau}}^2=-1, D_{v_0}^2 D_{e_0}D_{u_{\tau}}=2, D_{v_0}^3 D_{e_0}=-3, \\
& D_{v_0}^2 D_{u_{\tau}}^2=3, D_{v_0} D_{u_{\tau}}^3=-2, D_{v_0}^2 D_{e_0}^2 =1, D_{v_0} D_{e_0}^2 D_{u_{\tau}}=0, D_{v_0} D_{e_0}^3=1, \\
& D_{e_0} D_{u_{\tau}}^3=0, D_{e_0}^2 D_{u_{\tau}}^2=0, D_{e_0}^3 D_{u_{\tau}}=0, D_{e_0}^4=1, D_{u_{\tau}}^4=1,
\end{split}
\end{equation*}
\noindent
we have \(\text{deg }D_{v_0}=44, \text{deg }D_{e_0}=111, \text{deg }D_{u_{\tau}}=29\) and \(\mu(\mathcal{T}_X)=112.5\). Note that $\text{deg }D_{e_1}=111$ and \(\text{deg }D_{e_2}=9\). Now consider \(F=\text{Span}(e_0,e_1,e_2)\), which corresponds to a rank 2 reflexive subsheaf of $\mathcal{T}_X$ with slope \(115.5\). Hence $\mathcal{T}_X$ is unstable.
\end{proof}

\begin{prop}\label{G_3}
Let \(X\) be the	blow up of $V(\tau)$ on \(X'\), where \(\tau=\text{Cone}(v_1, v_2, e_0) \in \Delta' \) and $\alpha=1, \beta=1$ $($note that \(X=G_3\) in the notation of \cite[Section 4]{batyrev}$)$. Then $\mathcal{T}_X$ is unstable.	
\end{prop}

\begin{proof}
Then the rays of the fan $\Delta$ associated to \(X\) are as follows 
\begin{equation*}
\begin{split}
& v_1=(1,0,0,0), v_2=(0,1,0,0), v_0=(-1,-1,1,1), e_1=(0,0,1,0)\\
& e_2=(0,0,0,1), e_0=(0,0,-1,-1), u_{\tau}=(1,1,-1,-1,).
\end{split}
\end{equation*}
We have the following relations
\begin{equation*}
\begin{split}
& D_{v_1} \sim_{\text{lin}} D_{v_2}  \sim_{\text{lin}} D_{v_0}-D_{u_{\tau}}, D_{e_1} \sim_{\text{lin}}D_{e_2} \sim_{\text{lin}} D_{e_0} -D_{v_0}+D_{u_{\tau}}.
\end{split}
\end{equation*}
Hence $\text{Pic}(X)=\Z D_{v_0}  \oplus \Z  D_{e_0} \oplus \Z D_{u_{\tau}} .$ The anticanonical divisor is \(-K_X= D_{v_0}+3 D_{e_0}+D_{u_{\tau}}\). Since \(D_{v_0} D_{u_{\tau}}=0\), we have 
\begin{equation*}
\begin{split}
(-K_X)^3 &= D_{v_0}^3 + 27 D_{v_0} D_{e_0}^2+ 9 D_{v_0}^2 D_{e_0} +27 D_{e_0}^3+ 9 D_{e_0} D_{u_{\tau}}^2+ 27 D_{e_0}^2 D_{u_{\tau}}+ D_{u_{\tau}}^3.
\end{split}
\end{equation*}

Using the following relations
\begin{equation*}
\begin{split}
&  D_{v_0}^4=0, D_{v_0}^2 D_{e_0}^2=1, D_{v_0}^3 D_{e_0}=0, D_{v_0} D_{e_0}^3=2, D_{e_0}^4=0,\\
& D_{e_0}^2 D_{u_{\tau}}^2=-1,  D_{e_0}^3 D_{u_{\tau}}=2, D_{e_0} D_{u_{\tau}}^3=0, D_{u_{\tau}}^4=1,
\end{split}
\end{equation*}
\noindent
we have \(\text{deg }D_{v_0}=81, \text{deg }D_{e_0}=108, \text{deg }D_{u_{\tau}}=28\) and \(\mu(\mathcal{T}_X)=108.25\). Note that $\text{deg }D_{e_1}=\text{deg }D_{e_2}=55$. Now consider \(F=\text{Span}(e_0,e_1,e_2)\), which corresponds to a rank 2 reflexive subsheaf of $\mathcal{T}_X$ with slope \(109\). Hence $\mathcal{T}_X$ is unstable.
\end{proof}

\begin{prop}\label{G_4}
Let \(X\) be the	blow up of $V(\tau)$ on \(X'\), where \(\tau=\text{Cone}(v_0, e_0) \in \Delta' \) and $(\alpha,\beta)=(0,0), (0,1)$ and \((1,1)\) $($note that \(X=G_6, G_4, G_5\) respectively, in the notation of \cite[Section 4]{batyrev}$)$. Then $\mathcal{T}_X$ is stable.	
\end{prop}

\begin{proof}
Then the rays of the fan $\Delta$ associated to \(X\) are as follows 
\begin{equation*}
\begin{split}
& v_1=(1,0,0,0), v_2=(0,1,0,0), v_0=(-1,-1,\alpha,\beta), e_1=(0,0,1,0)\\
& e_2=(0,0,0,1), e_0=(0,0,-1,-1), u_{\tau}=(-1,-1,\alpha-1,\beta-1).
\end{split}
\end{equation*}
We have the following relations
\begin{equation*}
\begin{split}
& D_{v_1} \sim_{\text{lin}} D_{v_2}  \sim_{\text{lin}} D_{v_0}+D_{u_{\tau}}, D_{e_1} \sim_{\text{lin}} D_{e_0} -\alpha D_{v_0}-(\alpha-1) D_{u_{\tau}},\\
& D_{e_2} \sim_{\text{lin}} D_{e_0} -\beta D_{v_0}-(\beta-1) D_{u_{\tau}}.
\end{split}
\end{equation*}
Hence $\text{Pic}(X)=\Z D_{v_0}  \oplus \Z  D_{e_0} \oplus \Z D_{u_{\tau}}$. The anticanonical divisor is \(-K_X=(3-\alpha-\beta) D_{v_0}+3 D_{e_0}+(5-\alpha-\beta) D_{u_{\tau}}\). Since \(D_{v_0} D_{e_0}=0\), we have 
\begin{equation*}
\begin{split}
(-K_X)^3 &=a^3 D_{v_0}^3 + 3ab^2 D_{v_0} D_{u_{\tau}}^2+ 3a^2 b D_{v_0}^2 D_{u_{\tau}}+ 27 D_{e_0}^3 +9b^2 D_{e_0} D_{u_{\tau}}^2+ 27b D_{e_0}^2 D_{u_{\tau}}+ b^3 D_{u_{\tau}}^3.
\end{split}
\end{equation*}

Now consider the following cases.

\underline{\((\alpha, \beta)=(0,0)\) :} Then \(a=3, b=5\). Using the following
\begin{equation*}
\begin{split}
& D_{v_0}^4=3, D_{v_0}^3 D_{u_{\tau}}=-2, D_{v_0}^2 D_{u_{\tau}}^2=1, D_{v_0} D_{u_{\tau}}^3=0, D_{e_0}^4=3, \\
& D_{e_0}^3 D_{u_{\tau}}=-2, D_{e_0}^2 D_{u_{\tau}}^2=1, D_{e_0} D_{u_{\tau}}^3=0, D_{u_{\tau}}^4=-1,
\end{split}
\end{equation*}
we have \(\text{deg }D_{v_0}=\text{deg }D_{e_0}=36, \text{deg }D_{u_{\tau}} =37 \text{ and } \mu(\mathcal{T}_X)=100.25.\) Note also that \(\text{deg }D_{v_1}=\text{deg }D_{v_2}=\text{deg }D_{e_1}=\text{deg }D_{e_2}=73.\)

Next we consider rank 2 equivariant reflexive subsheaves of $\mathcal{T}_X$. We list those having maximum possible slope below.
\begin{itemize}
	\item[(i)] \(F=\text{Span}(v_0, e_0, u_{\tau})\), then $\mu(\mathcal{F})=54.5$.
	\item[(ii)] \(F=\text{Span}(v_0, v_1, v_2)\) or \(\text{Span}(e_0, e_1, e_2)\), then $\mu(\mathcal{F})=91$.
\end{itemize}

Finally we list rank 3 equivariant reflexive subsheaves of $\mathcal{T}_X$ possibly having maximum slope.
\begin{itemize}
	\item[(i)] \(F=\text{Span}(v_0, e_0,e_1, e_2, u_{\tau})\) or \(\text{Span}(v_0,v_1, v_2, e_0, u_{\tau})\), then $\mu(\mathcal{F})=85$.
	\item[(ii)] \(F=\text{Span}(e_0, e_1, e_2, v_1)\), then $\mu(\mathcal{F})=85$.
\end{itemize}
Hence $\mathcal{T}_X$ is stable.

\underline{\((\alpha, \beta)=(0,1)\) :} Then \(a=2, b=4\). Using the following
\begin{equation*}
\begin{split}
& D_{v_0}^4=2, D_{v_0}^3 D_{u_{\tau}}=-1, D_{v_0}^2 D_{u_{\tau}}^2=0, D_{v_0} D_{u_{\tau}}^3=1, D_{e_0}^4=1, \\
& D_{e_0}^3 D_{u_{\tau}}=-1, D_{e_0}^2 D_{u_{\tau}}^2=1, D_{e_0} D_{u_{\tau}}^3=0, D_{u_{\tau}}^4=-2,
\end{split}
\end{equation*}
we have \(\text{deg }D_{v_0}=32, \text{deg }D_{e_0}=63, \text{deg }D_{u_{\tau}} =41 \text{ and } \mu(\mathcal{T}_X)=104.25.\) Note also that \(\text{deg }D_{v_1}=\text{deg }D_{v_2}=73, \text{deg }D_{e_1}=104, \text{deg }D_{e_2}=31.\)

Next we list down rank 2 equivariant reflexive subsheaves of $\mathcal{T}_X$ possibly giving maximum slope.
\begin{itemize}
	\item[(i)] \(F=\text{Span}(e_0, e_1, e_2)\), then $\mu(\mathcal{F})=99$.
	\item[(ii)] \(F=\text{Span}(v_0, e_0, u_{\tau})\), then $\mu(\mathcal{F})=68$.
\end{itemize}

Finally consider the following rank 3 equivariant reflexive subsheaves of $\mathcal{T}_X$ contributing to maximum slope.
\begin{itemize}
	\item[(i)] \(F=\text{Span}(v_0, e_0, e_1, e_2, u_{\tau})\), then $\mu(\mathcal{F}) =90.33$.
	\item[(ii)] \(F=\text{Span}(v_1, v_2, e_1, u_{\tau})\), then $\mu(\mathcal{F})=97 $.
	\item[(iii)] \(F=\text{Span}(v_0, v_1, v_2, e_2)\),  then $\mu(\mathcal{F}) \sim 66.67$.
	\item[(iv)] \(F=\text{Span}(v_1, e_0, e_1, e_2)\),  then $\mu(\mathcal{F}) =90.33$.
\end{itemize}
Hence $\mathcal{T}_X$ is stable.

\underline{\((\alpha, \beta)=(1,1)\) :} Then \(a=1, b=3\). Using the following
\begin{equation*}
\begin{split}
& D_{v_0}^4=1, D_{v_0}^3 D_{u_{\tau}}=0, D_{v_0}^2 D_{u_{\tau}}^2=-1, D_{v_0} D_{u_{\tau}}^3=2, D_{e_0}^4=0, \\
& D_{e_0}^3 D_{u_{\tau}}=0, D_{e_0}^2 D_{u_{\tau}}^2=1, D_{e_0} D_{u_{\tau}}^3=0, D_{u_{\tau}}^4=-3
\end{split}
\end{equation*}
we have \(\text{deg }D_{v_0}=28, \text{deg }D_{e_0}=81, \text{deg }D_{u_{\tau}} =45 \text{ and }  \mu(\mathcal{T}_X)=101.5.\) Note also that \(\text{deg }D_{v_1}=\text{deg }D_{v_2}=73, \text{deg }D_{e_1}=\text{deg }D_{e_2}=53.\)

Next we consider rank 2 equivariant reflexive subsheaves of $\mathcal{T}_X$. We list those having maximum possible slope below.
\begin{itemize}
	\item[(i)] \(F=\text{Span}(v_1, v_2, u_{\tau})\), then $\mu(\mathcal{F})=95.5$.
	\item[(ii)] \(F=\text{Span}(v_0, e_0, u_{\tau})\), then $\mu(\mathcal{F})=77$.
	\item[(iii)] \(F=\text{Span}(e_0, e_1, e_2)\), then $\mu(\mathcal{F})=93.5$.
\end{itemize}

Finally we list rank 3 equivariant reflexive subsheaves of $\mathcal{T}_X$ having maximum possible slope.
\begin{itemize}
	\item[(i)] \(F=\text{Span}(v_0, e_0, e_1, e_2, u_{\tau})\), then $\mu(\mathcal{F}) \sim 86.67 $.
	\item[(ii)] \(F=\text{Span}(v_0, v_1, v_2, e_0, u_{\tau})\), then $\mu(\mathcal{F})=100$.
	\item[(iii)] \(F=\text{Span}( v_1, e_0, e_1, e_2)\), then $\mu(\mathcal{F}) \sim 86.67$.
	\item[(iv)] \(F=\text{Span}(v_0, v_1,e_0, u_{\tau})\), then $\mu(\mathcal{F}) \sim 75.67 $.
\end{itemize}
Hence $\mathcal{T}_X$ is stable.
\end{proof}

In the following table 	we summarize results regarding stability of tangent bundle of toric Fano 4-folds obtained in this paper, following the notations of Batyrev \cite[Section 4]{batyrev}.
\begin{center}
\text{Table 1: Stability of tangent bundle of toric Fano 4-folds}
\end{center}
	
\begin{longtable}{|c|l|l|l|}
	\hline 
Picard No.	&  \(X\)  & Stability of $\mathcal{T}_X$ & Reference  \\ 
	\hline \hline
1	& $\mathbb{P}^4$ &  \text{Stable} & Proposition \ref{stabtanP}\\ 
	\hline 
2	& \(B_1= \mathbb{P}(\mathcal{O}_{\mathbb{P}^3} \oplus  \mathcal{O}_{\mathbb{P}^3}(3)) \) &  \text{Unstable} & Corollary \ref{stabonPic2cor}, (1)\\ 
	\hline 
2	&  \(B_2= \mathbb{P}(\mathcal{O}_{\mathbb{P}^3} \oplus  \mathcal{O}_{\mathbb{P}^3}(2)) \)&  \text{Unstable} & Corollary \ref{stabonPic2cor}, (1)\\ 
	\hline 
2	&  \(B_3= \mathbb{P}(\mathcal{O}_{\mathbb{P}^3} \oplus  \mathcal{O}_{\mathbb{P}^3}(1)) \)&  \text{Unstable} & Corollary \ref{stabonPic2cor}, (2)\\ 
	\hline 
2	&  \(B_4= \mathbb{P}^1 \times \mathbb{P}^3 \)& \text{Strictly semistable} & Remark \ref{prodstab}\\ 
	\hline 
2	& \(B_5=\mathbb{P}( \mathcal{O}_{\mathbb{P}^1} \oplus \mathcal{O}_{\mathbb{P}^1} \oplus \mathcal{O}_{\mathbb{P}^1} \oplus \mathcal{O}_{\mathbb{P}^1}(1)  ) \) &  \text{Strictly semistable} & Corollary \ref{stabonPic2cor}, (3)\\ 
	\hline 
2	&  \(C_1=\mathbb{P}( \mathcal{O}_{\mathbb{P}^2} \oplus \mathcal{O}_{\mathbb{P}^2} \oplus \mathcal{O}_{\mathbb{P}^2}(2)  ) \)&  \text{Unstable} & Corollary \ref{stabonPic2cor}, (1)\\ 
	\hline 
2	& \(C_2=\mathbb{P}( \mathcal{O}_{\mathbb{P}^2} \oplus \mathcal{O}_{\mathbb{P}^2} \oplus \mathcal{O}_{\mathbb{P}^2}(1)  ) \) &\text{Unstable}  & Corollary \ref{stabonPic2cor}, (3)\\ 
	\hline 
2	& \(C_3=\mathbb{P}( \mathcal{O}_{\mathbb{P}^2} \oplus \mathcal{O}_{\mathbb{P}^2}(1) \oplus \mathcal{O}_{\mathbb{P}^2}(1)  ) \)  &\text{Unstable}  & Corollary \ref{stabonPic2cor}, (1)\\ 
	\hline 
2	&  \(C_4=\mathbb{P}^2 \times \mathbb{P}^2 \)& \text{Strictly semistable} & Remark \ref{prodstab}\\ 
	\hline 
3	& \(D_1=\mathbb{P}( \mathcal{O}_{\mathbb{P}^1 \times \mathbb{P}^2} \oplus   \mathcal{O}_{\mathbb{P}^1 \times \mathbb{P}^2}(1,2) )   \) &  \text{Unstable} & Proposition \ref{stabpic3D_1} (1)\\ 
	\hline 
3	& \(D_2=\mathbb{P}( \mathcal{O}_{\mathcal{B}_1} \oplus \mathcal{O}_{\mathcal{B}_1}(0,1))\) & \text{Unstable} & Proposition \ref{D_2}\\ 
	\hline 
3	& \(D_3= \mathbb{P}( \mathcal{O}_{\mathcal{B}_2} \oplus \mathcal{O}_{\mathcal{B}_2}(1,1))  \) & \text{Unstable} & Proposition \ref{D_2} \\ 
	\hline 
3	&  \(D_4=\mathbb{P}( \mathcal{O}_{\mathcal{B}_3} \oplus \mathcal{O}_{\mathcal{B}_3}(0,2))  \) & \text{Unstable} & Proposition \ref{D_4}\\ 
	\hline 
3	&  \(D_5=\mathbb{P}^1 \times  \mathbb{P}( \mathcal{O}_{ \mathbb{P}^2} \oplus  \mathcal{O}_{ \mathbb{P}^2}(2))\)&  \text{Unstable} & Proposition \ref{D_2}\\ 
	\hline 
3	&  \(D_6=\mathbb{P}( \mathcal{O}_{\mathbb{P}^1 \times \mathbb{P}^2} \oplus   \mathcal{O}_{\mathbb{P}^1 \times \mathbb{P}^2}(1,1) )   \) & \text{Unstable} & Proposition \ref{stabpic3D_1} (1)\\ 
	\hline 
3	& \(D_7=\) \tiny{\(\mathbb{P}( \mathcal{O}_{\mathbb{P}^1 \times \mathbb{P}^1} \oplus \mathcal{O}_{\mathbb{P}^1 \times \mathbb{P}^1} \oplus   \mathcal{O}_{\mathbb{P}^1 \times \mathbb{P}^2}(1,1) )   \)}  &  \text{Unstable} & Proposition \ref{stabPic3D_7} (1)\\ 
	\hline 
3	& \(D_8=  \mathbb{P}( \mathcal{O}_{\mathcal{B}_2} \oplus \mathcal{O}_{\mathcal{B}_2}(0,1))  \)  &  \text{Unstable} & Proposition \ref{D_2}\\ 
	\hline 
3	& \(D_9= \mathbb{P}( \mathcal{O}_{\mathcal{B}_2} \oplus \mathcal{O}_{\mathcal{B}_2}(1,0)) \)  &  \text{Unstable} & Proposition \ref{D_2}\\ 
	\hline 
3	& \(D_{10} =\mathbb{P}( \mathcal{O}_{\mathcal{B}_3} \oplus \mathcal{O}_{\mathcal{B}_3}(0,1)) \)  &  \text{Unstable} & Proposition \ref{D_4}\\ 
	\hline 
3	& \(D_{11}= \mathbb{P}( \mathcal{O}_{\mathcal{H}_1} \oplus  \mathcal{O}_{\mathcal{H}_1} \oplus \mathcal{O}_{\mathcal{H}_1}(0,1)  )\) & \text{Unstable} & Proposition \ref{D_{11}}\\ 
	\hline 
3	& \(D_{12}=\mathbb{P}^1 \times  \mathbb{P}( \mathcal{O}_{ \mathbb{P}^2} \oplus  \mathcal{O}_{ \mathbb{P}^2}(1))\) &  \text{Unstable} & Proposition \ref{D_2}\\ 
	\hline 
3	& \(D_{13}=\mathbb{P}^1 \times \mathbb{P}^1 \times   \mathbb{P}^2 \)& \text{Strictly semistable} & Remark \ref{prodstab}\\ 
	\hline 
3	&  \(D_{14}=\mathbb{P}^1 \times \mathbb{P}( \mathcal{O}_{\mathbb{P}^1} \oplus \mathcal{O}_{\mathbb{P}^1} \oplus \mathcal{O}_{\mathbb{P}^1}(1) )   \)&  \text{Strictly semistable} & Remark \ref{prodstab} and\\ 
 & & & Corollary \ref{stabonPic2cor}, (3)\\
	\hline 
3	&  \(D_{15} =\mathcal{H}_1 \times \mathbb{P}^2 \) &  \text{Strictly semistable} & Remark \ref{prodstab}\\ 
	\hline 
3	& \(D_{16} =  \mathbb{P}( \mathcal{O}_{\mathcal{B}_2} \oplus \mathcal{O}_{\mathcal{B}_2}(-1,1)) \) &  \text{Unstable} & Proposition \ref{D_2}\\ 
	\hline 
3	&  \(D_{17}=\) \tiny{\(\mathbb{P}( \mathcal{O}_{\mathbb{P}^1 \times \mathbb{P}^1} \oplus \mathcal{O}_{\mathbb{P}^1 \times \mathbb{P}^1}(1,0) \oplus   \mathcal{O}_{\mathbb{P}^1 \times \mathbb{P}^1}(0,1) )   \)} &  \text{Stable} & Proposition \ref{stabPic3D_7} (2)\\ 
	\hline 
3	&  \(D_{18}=\mathbb{P}( \mathcal{O}_{\mathbb{P}^1 \times \mathbb{P}^2} \oplus   \mathcal{O}_{\mathbb{P}^1 \times \mathbb{P}^2}(-1,2) )   \) &  \text{Unstable} & Proposition \ref{stabpic3D_1} (1)\\ 
	\hline 
3	&\(D_{19}=\mathbb{P}( \mathcal{O}_{\mathbb{P}^1 \times \mathbb{P}^2} \oplus   \mathcal{O}_{\mathbb{P}^1 \times \mathbb{P}^2}(-1,1) )   \)   &  \text{Stable} & Proposition \ref{stabpic3D_1} (2)\\ 
	\hline 
	3 & \(E_1=Bl_{\mathbb{P}^2}(B_2)\) & \text{Unstable} & Proposition \ref{E_1} (1)\\
	\hline 
	3 & \(E_2=Bl_{\mathbb{P}^2}(B_3)\) & \text{Unstable} & Proposition \ref{E_1} (1)\\
	\hline 
	3 & \(E_3=Bl_{\mathbb{P}^2}(B_4)\) & \text{Stable} & Proposition \ref{E_1} (2)\\
	\hline 
        3 & \(G_1\) & \text{Unstable} & Proposition \ref{G_1}\\
	\hline 
	3 & \(G_2=Bl_{\mathbb{P}^1 \times \mathbb{P}^1}(C_2)\) & \text{Unstable} & Proposition \ref{G_2}\\
	\hline 
	3 & \(G_3=Bl_{\mathbb{P}^1}(C_3)\) & \text{Unstable} & Proposition \ref{G_3}\\
	\hline 
	3 & \(G_4=Bl_{\mathcal{H}_1}(C_2)\) & \text{Stable} & Proposition \ref{G_4}\\
	\hline 
	3 & \(G_5=Bl_{\mathbb{P}^1 \times \mathbb{P}^1}(C_3)\) & \text{Stable} & Proposition \ref{G_4}\\
	\hline 
	3 & \(G_6=Bl_{\mathbb{P}^1 \times \mathbb{P}^1}(C_4)\) & \text{Stable} & Proposition \ref{G_4}\\
	\hline 
\end{longtable}

\section{Existence of equivariant indecomposable rank 2 vector bundles}\label{sec:existence-of-rank-2-stable-equivariant-vector-bundle-on-bott-tower}

In this section we construct a collection of equivariant indecomposable rank $2$ vector bundles over some special class of toric varieties of any dimension, namely Bott tower and pseudo-symmetric toric Fano varieties. Moreover, we show that in case of Bott tower, among the constructed vector bundles, there is a vector bundle which is stable with respect to a suitable choice of polarization.

\subsection{Existence of equivariant indecomposable rank 2 vector bundles on Bott tower}

A Bott tower is a tower $M_n \rightarrow M_{n-1} \rightarrow \cdots \rightarrow M_2 \rightarrow M_1 \rightarrow M_0=\{ \text{point} \}$, consisting of nonsingular projective toric varieties constructed as an iterated sequence of $\mathbb{P}^1$-bundles.
We briefly recall the fan $\Delta_k$ of the \(k\)-th stage Bott tower \(M_k\) (see \cite{civanyusuf} for more details). Let \(N=\Z^k\) with standard basis \(e_1, \ldots, e_k\). Rays of $\Delta_k$ are given by
 \begin{align*}
 & v_i   =e_i \text{ for } i=1, \ldots, k; \ v_{2k}  =-e_k \text{ and}\\
  & v_{k+i}  =-e_i+c_{i, i+1}e_{i+1}+\cdots +c_{i, k}e_k \text{ for }i=1, \ldots, k-1, 
 \end{align*} 
 where \(c_{i, j}\)'s are integers, called Bott numbers which can be assumed to be non-negative (see \cite[Theorem 2.2.1]{bj}). There are $2^k$ maximal cones of dimension $k$ generated by these rays such that no cone contains $v_i$ and $v_{k+i}$ simultaneously for $i=1, \ldots, k$. Let $D_i:=D_{v_i}$ denote the invariant prime divisor corresponding to the edge $v_i$ for \(i=1, \ldots, 2k\). We have the following relations among invariant prime divisors:
 \begin{equation}\label{reln}
 \begin{split}
 & D_{k+1} \sim_{\text{lin}} D_1, D_{k+2} \sim_{\text{lin}} D_{2}+c_{1, 2} D_{k+1}, \\
 & D_{k+i} \sim_{\text{lin}} D_{i}+c_{1, i}D_{k+1}+\cdots+c_{i-1, i}D_{k+i-1} \text{ for } i=3, \ldots, k.
 \end{split}
 \end{equation}

 \begin{prop}\label{existance of indecomp b on Bott T}
 	Let \(X=M_k\) with \(k \geq 2\) and \(1\leq p \leq k, 1\leq q \leq 2k, q \neq p,k+p \). Then there exists a collection of rank \(2\) indecomposable equivariant vector bundle \(\mathcal{E}_{p,q}\) on \(X\) with \(c_1(\mathcal{E}_{p,q})=D_p + D_q + D_{k+p}\).
 \end{prop}
 
 \begin{proof}
 	Consider the vector space \(E=\C^2\) and three distinct one dimensional subspaces \(L_p, L_q\) and \(L_{k+p}\) in \(E \). Now define the filtrations \( \left( E, \{ E_{p,q}^{v_j}(i) \}_{j=1, \ldots, 2k} \right) \)  as follows:
 	
 	$
 	E_{p,q}^{v_j}(i) = \left\{ \begin{array}{ccc}
 	
 	0 & i \leqslant -2 \\
 	
 	L_j & i =-1 \\ 
 	
 	E & i \geq 0
 \end{array} \right. 
 $
 for \(j=p, k+p, q\) 
%
%
%
 and 
 $
 E_{p,q}^{v_j}(i) = \left\{ \begin{array}{ccc}
 
 0 & i < 0 \\
 
 E & i \geq  0
\end{array} \right.
$
for all \(j \neq p, q, k+p\).

Hence the filtrations \( \left( E, \{ E_{p,q}^{v_j}(i) \}_{j=1, \ldots, 2k} \right) \) correspond to a rank \(2\) equivariant reflexive sheaf on \(X\), say \(\mathcal{E}_{p,q}\) (see Proposition \ref{reflexive}). Fix a maximal dimensional cone \(\sigma \in \Delta_k \). To prove that \(\mathcal{E}_{p,q}\) is also locally free, we need to show that the collection of subspaces \(\mathfrak{E}_{p,q}^{\sigma}=\{\{ E_{p,q}^{v_j}(i) \}_{v_j \in \sigma(1)} \}\) of \(E\) forms a distributive lattice (see Proposition \ref{compatibility for loc free sheaf}, Remark \ref{dist}). This follows because \(\sigma(1)\) contains at most two of the ray generators \(v_p, v_q, v_{k+p}\), since both \(v_p\) and \(v_{k+p}\) cannot belong to the same cone. 
Note that since \(L_p, L_q\) and \(L_{k+p}\) are distinct, the collection of subspaces \(\{ E_{p,q}^{v_j}(i) \}_{j=1, \ldots, 2k}\) do not form a distributive lattice. Hence by \cite[Corollary 2.2.3]{kly}, \(\mathcal{E}_{p,q}\) is in fact indecomposable. 

Note that for \(j=p, q, k+p\),

$
\text{dim}(E^{[v_j]}(i)) = \left\{ \begin{array}{ccc}

1 & i=-1, 0 \\

0 & \text{otherwise}
\end{array} \right.
$  
\ and for \(j \neq p, q, k+p\) \  
$
\text{dim}(E^{[v_j]}(i)) = \left\{ \begin{array}{ccc}

2 & i=0 \\

0 & \text{otherwise.}
\end{array} \right.
$

\vspace{0.3 cm}
We have  \(c_1 ( \mathcal{E})=D_p+D_q + D_{k+p}\) using Proposition \ref{chern}. 

\end{proof}

 \begin{rmk}
 	The above construction only depends on the choice of \(p,q\). Any three distinct lines \(L_p, L_q\) and \(L_{k+p}\) will give rise to the same equivariant vector bundle $\mathcal{ E }_{p, q}$ since two set of three distinct points in \(\mathbb{P}^1\) are equivalent by an automorphism of \(\mathbb{P}^1\). For \((p, q) \neq (p', q') \), the corresponding vector bundles $\mathcal{ E }_{p, q}$ and $\mathcal{ E }_{p', q'}$ are non isomorphic by \cite[Theorem 1.2.3, Corollary 1.2.4]{kly}.
 \end{rmk}
 
Now we will show that the vector bundle \(\mathcal{E}_{1,2}\) is stable with respect to a suitable choice of polarization. 

\begin{lemma}\label{nncomb}
	Let \(H=D_{k+1}+b D_{k+2}+D_{k+3}+ \cdots+D_{2k}\), where \(b > 0\) be an ample divisor on \(M_k, k \geq 2\). Then \(H^{k-1}\) is a non-negative integral combination of \(V(\tau)\)'s, where \(\tau\) varies over all walls in \(\Delta_k\) such that \(\tau(1) \subseteq \{v_{k+1}, \ldots, v_{2k}\} \).
\end{lemma}

\begin{proof}
	Note that \(H^{k-1}\) is a positive integral combination of monomials of the form \(D^{\underline{\alpha}}:=D^{\alpha_1}_{k+1} \cdots  D^{\alpha_k}_{2k}\) with non-negative integers \(\alpha_1, \ldots, \alpha_k\) satisfying \(\sum\limits_{j=1}^k \alpha_j=k-1\). To prove the lemma, it suffices to write such a monomial \(D^{\underline{\alpha}}\) as a non-negative integral combination of monomials of the form \(D^{\underline{\beta}}=D^{\beta_1}_{k+1} \cdots  D^{\beta_k}_{2k}\) with \(\beta_j \in \{0, 1\} \) for \(j=1, \ldots, k\) (see \cite[Lemma 12.5.2]{Cox}). 
	
	Since \(D_{k+1}^2=0\), without loss of generality we can assume \(\alpha_1 \leq 1 \). Now if \(\alpha_2 > 1\) using relations in \eqref{reln} and observing that \(v_2\) and \(v_{k+2}\) do not form a cone, we can write 
	\begin{equation*}
		\begin{split}
			D^{\underline{\alpha}} &=D^{\alpha_1}_{k+1} D^{\alpha_2-1}_{k+2} ( D_{2}+c_{1, 2} D_{k+1}) D^{\alpha_3}_{k+3}  \cdots  D^{\alpha_k}_{2k}=c_{1, 2} D^{\alpha_1+1}_{k+1} D^{\alpha_2-1}_{k+2} D^{\alpha_3}_{k+3}  \cdots  D^{\alpha_k}_{2k}\\
			&=c_{1, 2} D^{\beta_1}_{k+1} D^{\alpha_2-1}_{k+2} D^{\alpha_3}_{k+3}  \cdots  D^{\alpha_k}_{2k} \text{ where } \beta_1 \leq 1 \text{ if the monomial is non-zero.}
		\end{split}
	\end{equation*}
	
	Hence we have reduced the exponent of \(D_{k+2}\) by one and repeating this process we can write \(D^{\underline{\alpha}}\) as a non-negative integral combination of monomials of the form \begin{equation*}
		\begin{split}
			D^{\underline{\beta}}=D^{\beta_1}_{k+1} \cdots  D^{\beta_k}_{2k} \text{ with } \beta_1, \beta_2 \in \{0, 1\} \text{ and } \beta_3=\alpha_3, \ldots, \beta_k=\alpha_k.
		\end{split}
	\end{equation*}
	
	At the \(i\)-th stage, we arrive at monomials of the form \(D^{\underline{\alpha}}\), where \(\alpha_1, \ldots, \alpha_{i-1} \in \{0,1 \}\). Suppose \(\alpha_{i} > 1\). Then again using relations in \eqref{reln} and observing that \(v_{i}\) and \(v_{k+i}\) do not form a cone, we can write \(D^{\underline{\alpha}}\) as a non-negative integral combination of monomials of the form \(D^{\underline{\beta}}\)'s with \(\beta_{i} < \alpha_{i}\) and \( \beta_{i+1}=\alpha_{i+1}, \ldots, \beta_k=\alpha_k\). If \(\beta_j > 1\) for some \(j=1, \ldots, i-1\), appealing to Stage \(j\), we will write this monomial as a non-negative integral combination of monomials of the form \(D^{\underline{\beta'}}\)'s with \(\beta'_1, \ldots, \beta'_j \in \{0,1 \}\) and \(\beta'_{j+1}= \beta_{j+1}, \ldots, \beta'_k= \beta_k\) . Hence eventually we write \(D^{\underline{\alpha}}\) as a non-negative integral combination of monomials of the form \(D^{\underline{\gamma}}\)'s with \(\gamma_1, \ldots, \gamma_{i} \in \{0,1 \}\) and \( \gamma_{i+1}=\alpha_{i+1}, \ldots, \gamma_k=\alpha_k\). Continuing this process at the \(k\)-th stage we can express \(D^{\underline{\alpha}}\) in the desired form.
\end{proof}


\begin{rmk} \label{coeffb} 
	Using Lemma \ref{nncomb} we can write \(H^{k-1}=\sum_{\tau } a_{\tau} V(\tau) \), where $\tau$ varies over all such walls with $\tau(1) \subseteq \{v_{k+1}, \ldots, v_{2k}\}$, \(a_{\tau} \in \Z_{\geq 0}[b, c_{i, j} : 1 \leq i< j \leq k] \). Also observe that  \(a_{\tau}\) involves \(b\) only if \(v_{k+2} \in \tau(1) \).
\end{rmk}

\begin{prop}\label{existance of stable b on Bott T}
	Let \(X=M_k\) with \(k \geq 2\) and consider the polarization \(H=D_{k+1}+b D_{k+2}+D_{k+3}+ \cdots+D_{2k}\), where \(b > 0\). Then there exists a rank \(2\) stable equivariant vector bundle \(\mathcal{E}\) on \(X\) with \(c_1(\mathcal{E})=2D_1+D_2\), which is \(H\)-stable for sufficiently large \(b\).
\end{prop}

\begin{proof}
	Consider the equivariant vector bundle vector \(\mathcal{E}_{1,2}\) associated to the filtrations \\* \( \left( E, \{ E_{1,2}^{v_j}(i) \}_{j=1, \ldots, 2k} \right) \)  from Proposition \ref{existance of indecomp b on Bott T}. Furthermore, \(\text{deg}(\mathcal{E})=2\text{deg}(D_1)+\text{deg}(D_2)\). Hence $\mu(\mathcal{E})=\text{deg}(D_1)+\frac{1}{2} \text{deg}(D_2)$. 

The only equivariant reflexive subsheaves of \(\mathcal{E}\) are \(\mathcal{O}_X(D_1), \mathcal{O}_X(D_2), \mathcal{O}_X(D_{k+1})\) and \(\mathcal{O}_X\) and both \(\text{deg}(D_1) (=\text{deg}(D_{k+1}))\) and \(\text{deg}(\mathcal{O}_X)\) are less than \(\mu(\mathcal{E})\). It remains to show that \(\text{deg}(D_2) < \mu(\mathcal{E}), \text{ i.e. } \)
\begin{equation}\label{degexistence}
\text{deg}(D_2) <2 \ \text{deg}(D_1).
\end{equation}
Now using Lemma \ref{nncomb} and Remark \ref{coeffb}, we see that $\text{deg}(D_2)=P(c_{i, j}: 1 \leq i<  j \leq k  )$ and $\text{deg}(D_1)=b + Q(b,c_{i, j}: 1 \leq i<  j \leq k  )$, $\text{ where } P(c_{i, j}: 1 \leq i<  j \leq k  ) \in \Z_{\geq 0}[c_{i, j}: 1 \leq i<  j \leq k ] \text{ and }  Q(c_{i, j}: 1 \leq i<  j \leq k  ) \in \Z_{\geq 0}[b, c_{i, j}: 1 \leq i<  j \leq k ].$

So \eqref{degexistence} holds for sufficiently large \(b\), and hence we conclude that \(\mathcal{E}\) is \(H\)-stable.
\end{proof}

\begin{rmk}
It can be shown that for the polarization \(H=b_1 D_{k+1}+\ldots + b_k D_{2k}\) with \( b_i >0\) for all \(i=1, \ldots, k\), the vector bundle \(\mathcal{E}\) constructed above is \(H\)-stable whenever \(b_1 < b_2\) for the cases \(k=2, 3\).
\end{rmk}

\subsection{Existence of equivariant indecomposable rank 2 vector bundles on pseudo\\ -symmetric Fano toric varieties}

A toric Fano variety is called pseudo-symmetric if its fan
contains two centrally symmetric maximal cones, i.e. there exists \(\sigma, \sigma' \in \Delta \), maximal cones such that \(\sigma =-\sigma'\). For any pseudo-symmetric toric Fano variety \(X\), there exists
\(s, p, q \in \Z_{\geq 0} \) and \(k_1, \ldots, k_p, l_1, \ldots, l_q \in \Z_{\geq 0}\) such that
 \begin{equation}\label{pseudosymm}
 X \cong (\mathbb{P}^1)^s \times V^{2k_1} \times \ldots \times V^{2k_p} \times  \widetilde{V}^{2l_1} \times \ldots \times \widetilde{V}^{2l_q},
 \end{equation}
where \(V^n \) (respectively, $\widetilde{V}^n$) is a \(n\)-dimensional toric Fano variety called the $n$-dimensional Del Pezzo variety (respectively, pseudo Del Pezzo variety) (see \cite{ewald}). We briefly recall the fan structures of \(V^n \) and $\widetilde{V}^n$ from \cite[Section 3]{Cinzia}. Let \(v_1, \ldots, v_n\) be a basis of \(N=\Z^n\), where \(n\) is even, say \(n=2r\). Set \(v_0=-v_1-\cdots-v_n\) and \(w_i=-v_i\) for \(i=0, \ldots, n\). Then \(\Delta_{V^n}(1)=\{v_0, w_0, \ldots, v_n, w_n\}\) and \(\Delta_{\widetilde{V}^n}(1)=\{v_0,v_1, w_1, \ldots, v_n, w_n\}\). Explicitly the fans are given as follows:
\begin{align*}
\Delta_{V^n} =&\{\text{Cone}(v_i, w_j : i \in I^r, j \in J^r) \text{ and their faces} | I^r,  J^r \subseteq \{0, \ldots, n\} \text{ disjoint} \};\\
\Delta_{\widetilde{V}^n} =&\{\text{Cone}(v_0, v_i, w_j : i \in I^{r-1}, j \in J^r), \text{Cone}(v_i, w_j : i \in \widetilde{I}^{r+s}, j \in \widetilde{J}^{r-s}) \text{ and their faces}\\
&  | I^{r-1},  J^r \subseteq \{1, \ldots, n\} \text{ disjoint}, s \in \{0, \ldots,r \} \text{ and } \widetilde{I}^{r+s}, \widetilde{J}^{r-s} \text{ a partition of } \{1, \ldots,n \}\}.
\end{align*}

We construct a collection of equivariant indecomposable rank 2 vector bundle on \(X\). When \(X\) is a product of $\mathbb{P}^1$'s, then we are done by Proposition \ref{existance of indecomp b on Bott T}. Let us first prove the existence of a collection of equivariant indecomposable rank 2 vector bundles on Del Pezzo variety \(V^n\). 

Consider the vector space \(E=\C^2\) and three distinct one dimensional subspaces \(L_a, L_b\) and \(L'_a\) in \(E \) where \(0\leq a,b \leq n, a \neq b\). Define the filtrations \( \left( E, \{ E_{\{a,b\},a}^{\rho}(i) \} \right) \)  as follows:

$
E_{\{a,b\},a}^{v_j}(i) = \left\{ \begin{array}{ccc}

0 & i \leqslant -2 \\

L_j & i =-1 \\ 

E & i \geq 0,
\end{array} \right. 
$
for \(j=a, b\); \hspace{1.5 cm}
$
E_{\{a,b\},a}^{w_a}(i) = \left\{ \begin{array}{ccc}

0 & i \leqslant -2 \\

L'_a & i =-1 \\ 

E & i \geq 0
\end{array} \right. 
$
and 

$
E_{\{a,b\},a}^{\rho}(i) = \left\{ \begin{array}{ccc}

0 & i < 0 \\

E & i \geq  0,
\end{array} \right.
$
for any ray except \(v_a, v_b, w_a\).

By Proposition \ref{reflexive}, the filtrations \( \left( E, \{ E_{\{a,b\},a}^{\rho}(i) \} \right) \) correspond to a rank 2 equivariant reflexive sheaf $\mathcal{ E }_{\{a,b\},a}$ on \(V^n\). Since the rays \(v_a, v_b, w_a\) do not form a cone in $\Delta_{V^n}$, it follows that the filtrations satisfy the compatibility condition given in Remark \ref{dist}, and hence $\mathcal{ E }_{\{a,b\},a}$ is locally free by Proposition \ref{compatibility for loc free sheaf}. As the one dimensional subspaces \(L_a, L_b, L'_a\) are distinct, the filtrations \( \left( E, \{ E_{\{a,b\},a}^{\rho}(i) \} \right) \) do not form a distributive lattice which implies that $\mathcal{ E }_{\{a,b\},a}$ does not split and hence is indecomposable.

Consider three distinct one dimensional subspaces \(L_a, L'_a\) and \(L'_b\) in \(E \) where \(0\leq a,b \leq n, a \neq b\). By similar arguments, we have an equivariant indecomposable rank 2 locally free sheaf $\mathcal{ E }_{a,\{a,b\}}$ on \({V}^n\) associated to the filtrations \( \left( E, \{ E_{a,\{a,b\}}^{\rho}(i) \} \right) \) given as follows: 

$
E_{a,\{a,b\}}^{v_a}(i) = \left\{ \begin{array}{ccc}

0 & i \leqslant -2 \\

L_a & i =-1 \\ 

E & i \geq 0,
\end{array} \right. 
$; \hspace{1.5 cm}
$
E_{a,\{a,b\}}^{w_j}(i) = \left\{ \begin{array}{ccc}

0 & i \leqslant -2 \\

L'_j & i =-1 \\ 

E & i \geq 0
\end{array} \right. 
$
for \(j=a, b\);  and 

$
E_{a,\{a,b\}}^{\rho}(i) = \left\{ \begin{array}{ccc}

0 & i < 0 \\

E & i \geq  0,
\end{array} \right.
$
for any ray except \(v_a, w_a, w_b\).

By similar arguments, we have a collection of equivariant indecomposable rank 2 vector bundles on pseudo Del Pezzo variety \(\widetilde{V}^n\), given by $\mathcal{ F }_{\{a,b\},a}$ associated to the filtrations  \( \left( E, \{ E_{\{a,b\},a}^{\rho}(i) \} \right) \), $\mathcal{ F }_{a,\{a,b\}}$ associated to the filtrations  \( \left( E, \{ E_{a,\{a,b\}}^{\rho}(i) \} \right) \)  for \(1\leq a,b \leq n, a \neq b\) and $\mathcal{ F }_{\{0, a\},a}$ associated to the filtrations  \( \left( E, \{ E_{\{0, a\},a}^{\rho}(i) \} \right) \) for \(0< a \leq n\).

When \(X\) is not a product of $\mathbb{P}^1$'s, from \eqref{pseudosymm} at least one of \(k_p\) or \(l_q\) is positive. Without loss of generality, let us assume \(k_p\) is positive. We have an equivariant rank 2 indecomposable vector bundle on \(V^{2k_p}\), say $\mathcal{ E }^{V^{2k_p}}$. Pulling back  $\mathcal{ E }^{V^{2k_p}}$ to \(X\) via the projection map, we get an equivariant rank 2 vector bundle which is still indecomposable by \cite[Remark 3.3]{Alexandru}. 

From the above discussion we get the following proposition.
\begin{prop}\label{pseudo} 
Let \(X\) pseudo-symmetric toric Fano variety. There exists a collection of equivariant indecomposable rank 2 vector bundles on \(X\).
\end{prop}

\bibliographystyle{plain}
\bibliography{refs_paper}

\end{document}